\documentclass[reqno]{amsart}
\usepackage{amsfonts}
\usepackage{a4wide}
\usepackage{cite}
\usepackage[linktocpage]{hyperref}
\usepackage{cleveref}

\usepackage{amsmath,amssymb,amsthm,amsfonts}
\usepackage{mathrsfs}
\usepackage{bbm}

\usepackage{color}

\newtheorem{lemma}{Lemma}[section]
\newtheorem{theorem}{Theorem}[section]

\newtheorem{remark}{Remark}[section]

\numberwithin{equation}{section}

\newcommand{\dis}{\displaystyle}

\newcommand{\R}{\mathbb{R}}

\renewcommand{\S}{\mathbb{S}}
\newcommand{\T}{\mathbb{T}}

\newcommand{\CE}{\mathcal{E}}

\newcommand{\CH}{\mathcal{H}}

\newcommand{\para}{\shortparallel}

\newcommand{\ep}{\epsilon}

\newcommand{\na}{\nabla}

\newcommand{\be}{\beta}
\newcommand{\ga}{\gamma}
\newcommand{\om}{\omega}
\newcommand{\la}{\lambda}
\newcommand{\de}{\delta}

\newcommand{\pa}{\partial}

\newcommand{\eps}{\epsilon}

\newcommand{\Ga}{\Gamma}

\makeatletter
\@namedef{subjclassname@2020}{%
  \textup{2020} Mathematics Subject Classification}
\makeatother

\begin{document}
\title[Quantum Boltzmann equation for soft potentials]{Existence and uniqueness of solutions to the quantum Boltzmann equation for soft potentials}
\author{Zongguang Li}
\address[ZGL]{Department of Mathematics, The Chinese University of Hong Kong,
Shatin, Hong Kong, P.R.~China}
\email{zgli@math.cuhk.edu.hk}

\begin{abstract}
In this paper we consider a modified quantum  Boltzmann equation with the quantum effect measured by a continuous parameter $\delta$ that can decrease from $\delta=1$ for the Fermi-Dirac particles to $\delta=0$ for the classical particles. In case of soft potentials, for the corresponding Cauchy problem in the whole space or in the torus, we establish the global existence and uniqueness of non-negative mild solutions in the function space $L^{\infty}_{T}L^{\infty}_{v,x}\cap L^{\infty}_{T}L^{\infty}_{x}L^1_v$ with small defect mass, energy and entropy but allowed to have large amplitude up to the possibly maximum upper bound $F(t,x,v)\leq \frac{1}{\delta}$. The key point is that the obtained estimates are uniform in the quantum parameter $0< \delta\leq1$. In particular, as $\delta\to 0$ we can recover the results on the classical Boltzmann equation around global Maxwellians for which solutions may have arbitrarily large oscillations.
\end{abstract}

\date{\today}


\maketitle
\thispagestyle{empty}

\section{Introduction}
We consider the following Cauchy problem on the modified quantum Boltzmann equation
 \begin{eqnarray}\label{QBE}
&\dis \pa_tF+v\cdot \na_x F=\mathcal{C}_\de(F),   \quad &\dis F(0,x,v)=F_0(x,v),
\end{eqnarray}
where $F(t,x,v)\geq0$ is an unknown velocity  distribution function of particles with position $x\in \Omega=\R^3$ or $\T^3$ and velocity $v\in \R^3$ at time $t> 0$ and initial data $F_0(x,x)$ is given. The collision operator $\mathcal{C}_\de$ takes the form of
\begin{align*}
\left(\mathcal{C}_\de(F)\right)(v)=\int_{\R^3}\int_{\S^2}B(v-u,\theta)&\left[ F(u')F(v')\left(1-\de F(u) \right)\left(1-\de F(v) \right)\right. \notag\\
 & \left. -F(u)F(v)\left(1-\de F(u') \right) \left(1-\de F(v') \right)\right]d\omega du,
\end{align*}
where $0\leq \de \leq 1$ is a continuous parameter denoting the quantum effect. For $\de=0$, the equation \eqref{QBE} is the classical Boltzmann equation, while for $\de=1$, the equation corresponds to the Boltzmann equation for Fermi-Dirac particles. In this paper we consider the continuous parameter $\de$ taking values in the interval $[0,1]$. Moreover, we consider only the soft potentials under the Grad's angular cutoff assumption. Therefore, the collision kernel $B(v-u,\theta)$ satisfies 
\begin{equation*}
B(v-u,\theta)=|v-u|^\gamma b(\theta),
\end{equation*}
where $-3<\ga<0$, $0\leq b(\theta) \leq C|\cos \theta|$ and $\cos \theta=\frac{(v-u)\cdot \omega}{|v-u|}$. The post-collision velocities $v'$ and $u'$ are defined by
\begin{align}\label{velocity}
\begin{split}
v'=v-\left[(v-u)\cdot \omega \right]\omega, \quad &u'=u+\left[(v-u)\cdot \omega \right]\omega,\\
u'+v'=u+v,\quad |u'|^2&+|v'|^2=|u|^2+|v|^2.
\end{split}
\end{align}
Let the equilibrium state $\mu_{\de,\rho}$ be denoted by
\begin{align}\label{Defmuderho}
\mu_{\de,\rho}(v)=\frac{1}{\de+\rho e^\frac{|v|^2}{2}},
\end{align}
where $\rho>0$ is a constant.
Moreover, if $F(t,x,v)$ is a solution to the modified quantum Boltzmann equation \eqref{QBE} with the initial datum $F_0(x,v)$, the following conservation laws of the initial defect mass, momentum, energy and the defect entropy inequality of $F(t,x,v)$ hold,
\begin{align} 
&\int_{\Omega}\int_{\R^3} \{F(t,x,v)-\mu_{\de,\rho}(v) \}dvdx=\int_{\Omega}\int_{\R^3} \{F_0(x,v)-\mu_{\de,\rho}(v) \}dvdx:=M_0, \label{M}\\
&\int_{\Omega}\int_{\R^3} v\{F(t,x,v)-\mu_{\de,\rho}(v) \}dvdx=\int_{\Omega}\int_{\R^3} v\{F_0(x,v)-\mu_{\de,\rho}(v) \}dvdx:=J_0,\notag
\\
&\int_{\Omega}\int_{\R^3} |v|^2\{F(t,x,v)-\mu_{\de,\rho}(v) \}dvdx=\int_{\Omega}\int_{\R^3} |v|^2\{F_0(x,v)-\mu_{\de,\rho}(v) \}dvdx:=E_0,\label{E}
\end{align}
and 
\begin{align}
\CH_{\de,\rho}(F(t)):=&\int_{\Omega}\int_{\R^3} \left\{F(t,x,v)\log{F}(t,x,v)+\frac{1}{\de}(1-\de F(t,x,v))\log(1-\de{F}(t,x,v))\right.\notag\\
                                & \left.-\mu_{\de,\rho}(v)\log{\mu_{\de,\rho}(v)}-\frac{1}{\de}(1-\de\mu_{\de,\rho}(v))\log(1-\de\mu_{\de,\rho}(v))\right\}dvdx\leq \CH_{\de,\rho}(F(0)).
\label{H}
\end{align}
Furthermore, the functional $\CE_{\de,\rho}(F(t))$ is given by 
\begin{align}\label{DefCE}
\CE_{\de,\rho}(F(t)):=\CH_{\de,\rho}(F(t))+(\log{\rho})M_0+\frac{1}{2}E_0,  
\end{align}
with the initial datum
$\CE_{\de,\rho}(F_0):=\CE_{\de,\rho}(F(0))$.

The quantum Boltzmann equation \eqref{QBE} with $\delta=1$ is a kinetic model which describes the behavior of rarefied gas in non-equilibrium state for particles satisfying the Pauli exclusion principle. For the spatially homogeneous case, we mention \cite{Lu1,LW}. For the inhomogeneous case, Lu \cite{Lu2,Lu3} obtained the global existence and weak stability of weak solutions for Fermi-Dirac particles with very soft potentials, see also \cite{OW,ZL1,ZL2}. Jiang-Xiong-Zhou \cite{JXZ} and Jiang-Zhou \cite{JZ} studied the incompressible Navier-Stokes-Fourier limit and the compressible Euler and acoustic limits from the quantum Boltzmann equation, respectively. In recent years, more attentions have been paid to study how the equation depends on the continuous quantum parameter. Dolbeault \cite{Dolbeault} gave the existence and uniqueness results of the renormalized solutions and obtained solutions to the Boltzmann equation by passing the limit with respect to the quantum parameter. He-Lu-Pulvirenti \cite{HLP} deduced the convergence to the homogeneous Fokker-Planck-Landau equation from the homogeneous quantum Boltzmann equation. Alonso-Bagland-Desvillettes-Lods \cite{ABDL} obtained uniform estimates in the quantum parameter for the entropy dissipation of the Landau-Fermi-Dirac equation. However, there are still many unknowns in the study of the inhomogeneous  quantum Boltzmann equation when ones take into account the effect of the quantum parameter. 

For the case when $\de=0$, the quantum Boltzmann equation becomes the classical Boltzmann equation and there have been extensive works on global existence and large time behavior of solutions. For instance, DiPerna-Lions \cite{DL} proved the global existence of renormalized solutions for general $L^1_{x,v}$ initial data with finite mass, energy and entropy. In the perturbation framework near global Maxwellians, Grad \cite{Grad} and Ukai \cite{Ukai} developed the spatially inhomogeneous well-posedness theory by the spectral analysis and the bootstrap argument, see also \cite{NI,Shizuta,UY}. For more properties of the linearized operator, interested readers may also refer to Ellis-Pinsky \cite{EP}, Baranger-Mouhot \cite{BM} and the references therein. Another important approach using the energy method through the macro-micro decomposition is established by Liu-Yang-Yu \cite{LYY} and Guo \cite{Guo04} in thec2000s. The case of soft potentials $-3<\ga<0$ for which the collision frequency $\nu(v)\sim(1+|v|)^\ga$ is degenerate in large velocities are more complicated to treat. For $-1<\ga<0$, Ukai-Asano \cite{UA82} and Caflisch \cite{Caflisch1,Caflisch2} independently proved the global existence and large-time behavior of the solutions in the whole space and in torus, respectively. When $\ga$ is in the full range $(-3,0)$, Guo \cite{Guo03} constructed the global classical solutions near global Maxwellians and Strain-Guo \cite{SG1,SG} derived the long-time behavior of solutions, see also a recent work \cite{DD} for the study of spectral gap formulation in soft potentials. 

It is also an interesting topic on how to obtain large amplitude solution with extra smallness assumptions. As $\delta\to 0$ corresponding to the limit to the classical Boltzmann equation, the restriction that $0\leq F_0\leq 1/\de$ is reduced to the only non-negativity condition that $F_0\geq 0$ and thus the situation is relatively easier than the one in case of $0<\delta\leq 1$. In case of $\delta=0$, motivated by the original work Guo \cite{Guo09}, Duan-Huang-Wang-Yang \cite{DHWY} developed an $L^\infty_x L^1_v \cap L^\infty_{x,v}$ approach to obtain the global well-posedness of mild solutions under the condition that both $\CE(F_0)$ and the $L^1_x L^\infty_v$ norm of $(F_0-\mu)/\sqrt{\mu}$ are small. The $L^\infty_{x,v}$ norm of $\langle v\rangle^\beta(F_0-\mu)/\sqrt{\mu}$ is only required to be bounded so that the initial data is allowed to have arbitrary large amplitude around the global Maxwellian. See also \cite{Li} by the author of this paper for an extension to the large amplitude results in $L^p_vL^\infty_TL^\infty_x$ spaces. However, for $0\leq \delta\leq 1$, it remains unclear whether there exists a large amplitude solution to the quantum Boltzmann equation for soft potentials with estimates that can be uniform in the quantum parameter, in particular, whether or not one can recover the result in \cite{DHWY} by passing the limit $\delta\to 0$.

The perturbation function $f=f(t,x,v)$ is defined by 
\begin{align}\label{Perturbedform}
F(t,x,v)=\mu_{\de,\rho}(v)+\sqrt{\bar{{\mu}}_{\de,\rho}} f(t,x,v),
\end{align}
 where $\mu_{\de,\rho}(v)$ is given in \eqref{Defmuderho} and $\sqrt{\bar{{\mu}}_{\de,\rho}} f(t,x,v)=\sqrt{\bar{{\mu}}_{\de,\rho}}(v) f(t,x,v)$ with
\begin{align}\label{Defbarmu}
\sqrt{\bar{{\mu}}_{\de,\rho}}(v)=\sqrt{\mu_{\de,\rho}(v)[1-\de\mu_{\de,\rho}(v)]}=\frac{\sqrt\rho e^{{|v|^2}/{4}}}{\de+\rho e^{{|v|^2}/{2}}}.
\end{align}
For simplicity, we use the above notations throughout the paper.
Substituting \eqref{Perturbedform} into \eqref{QBE}, the perturbed equation can be written as
\begin{align}\label{PQBE}
\pa_tf+v\cdot \na_x f+L_\de f=\Ga_\de (f).
\end{align}
The linear operator $L_\de$ and the nonlinear term $\Ga_\de$ are respectively given by
\begin{align}\label{DefLde}
L_\de=\nu_\de-K_\de,
\end{align}
and
\begin{align}\label{DefGade}
\dis
&\Ga_\de (f)(v)=\frac{1}{\sqrt{\bar{\mu}_{\de,\rho}(v)}}\int_{\R^3}\int_{\S^2}B(v-u,\theta)\left[\sqrt{\bar{\mu}_{\de,\rho}}f(u')\sqrt{\bar{\mu}_{\de,\rho}}f(v')\left(1-\de\mu_{\de,\rho}(v)-\de\mu_{\de,\rho}(u)\right)\right.\notag\\
&\quad -\sqrt{\bar{\mu}_{\de,\rho}}f(u)\sqrt{\bar{\mu}_{\de,\rho}}f(v)\left(1-\de\mu_{\de,\rho}(v')-\de\mu_{\de,\rho}(u')\right)+\de\sqrt{\bar{\mu}_{\de,\rho}}f(v')\sqrt{\bar{\mu}_{\de,\rho}}f(u)\left(\mu_{\de,\rho}(v)-\mu_{\de,\rho}(u')\right)\notag\\
&\quad+\de\sqrt{\bar{\mu}_{\de,\rho}}f(u')\sqrt{\bar{\mu}_{\de,\rho}}f(u)\left(\mu_{\de,\rho}(v)-\mu_{\de,\rho}(v')\right)+\de\sqrt{\bar{\mu}_{\de,\rho}}f(v')\sqrt{\bar{\mu}_{\de,\rho}}f(v)\left(\mu_{\de,\rho}(u)-\mu_{\de,\rho}(u')\right)\notag\\
&\quad+\de\sqrt{\bar{\mu}_{\de,\rho}}f(u')\sqrt{\bar{\mu}_{\de,\rho}}f(v)\left(\mu_{\de,\rho}(u)-\mu_{\de,\rho}(v')\right)+\de\sqrt{\bar{\mu}_{\de,\rho}}f(u)\sqrt{\bar{\mu}_{\de,\rho}}f(v)\sqrt{\bar{\mu}_{\de,\rho}}f(u')\notag\\
&\quad+\de\sqrt{\bar{\mu}_{\de,\rho}}f(u)\sqrt{\bar{\mu}_{\de,\rho}}f(v)\sqrt{\bar{\mu}_{\de,\rho}}f(v')-\de\sqrt{\bar{\mu}_{\de,\rho}}f(u')\sqrt{\bar{\mu}_{\de,\rho}}f(v')\sqrt{\bar{\mu}_{\de,\rho}}f(u)\notag\\
&\quad\left.-\de\sqrt{\bar{\mu}_{\de,\rho}}f(u')\sqrt{\bar{\mu}_{\de,\rho}}f(v')\sqrt{\bar{\mu}_{\de,\rho}}f(v)\right]d\om du,
\end{align}
where
\begin{align}
\nu_\de(v)=&\int_{\R^3}\int_{\S^2}B(v-u,\theta)\notag\\ 
&\quad\left[ \mu_{\de,\rho}(u)-\de\mu_{\de,\rho}(u)\mu_{\de,\rho}(u')-\de\mu_{\de,\rho}(u)\mu_{\de,\rho}(v')+\de\mu_{\de,\rho}(u')\mu_{\de,\rho}(v') \right]d\om du,\label{Defnude}\\
K_{\de}f(v)=&\int_{\R^3}\int_{\S^2}B(v-u,\theta)\notag\\
&\left[\frac{\sqrt{\bar{\mu}_{\de,\rho}(u')}}{\sqrt{\bar{\mu}_{\de,\rho}(v)}}\left[ \mu_{\de,\rho}(v')-\de\mu_{\de,\rho}(v')\mu_{\de,\rho}(u)-\de\mu_{\de,\rho}(v')\mu_{\de,\rho}(v)+\de\mu_{\de,\rho}(u)\mu_{\de,\rho}(v) \right]f(u')\right.\notag\\
&\left.\frac{\sqrt{\bar{\mu}_{\de,\rho}(v')}}{\sqrt{\bar{\mu}_{\de,\rho}(v)}} \left[ \mu_{\de,\rho}(u')-\de\mu_{\de,\rho}(u')\mu_{\de,\rho}(u)-\de\mu_{\de,\rho}(u')\mu_{\de,\rho}(v)+\de\mu_{\de,\rho}(u)\mu_{\de,\rho}(v) \right]f(v')\right.\notag\\
&\left.\frac{\sqrt{\bar{\mu}_{\de,\rho}(u)}}{\sqrt{\bar{\mu}_{\de,\rho}(v)}} \left[ \mu_{\de,\rho}(v)-\de\mu_{\de,\rho}(v)\mu_{\de,\rho}(u')-\de\mu_{\de,\rho}(v)\mu_{\de,\rho}(v')+\de\mu_{\de,\rho}(u')\mu_{\de,\rho}(v') \right]f(u)\right]d\om du\label{DefKde}.
\end{align}
By integrating along the backward trajectory, we yield the mild form of the perturbed equation \eqref{PQBE} as follows:
\begin{align}\label{mildQBE}
\dis f(t,x,v)=&e^{-\nu_\de(v)t}f_0(x-vt,v)+\int_0^t e^{-\nu_\de(v)(t-s)}(K_\de f)(s,x-v(t-s),v)ds \notag\\
&+\int_0^t e^{-\nu_\de(v)(t-s)}\Ga_\de(f)(s,x-v(t-s),v)ds.
\end{align}

For given bounded initial data $f_0=f_0(x,v)$ in $L^{\infty}_{v,x}$, we first construct a local solution in $L^\infty_TL^{\infty}_{v,x}$ space. Then under the smallness assumption for $\sup_{0\leq\de\leq1}\left\{ \CE_{\de,\rho}(F_0)+\|f_0\|_{L^1_xL^\infty_v}\right\}$, we establish the $L^\infty_TL^{\infty}_{v,x}\cap L^{\infty}_{T}L^{\infty}_{x}L^1_v$ estimates to extend the obtained solution to a unique global solution which can be governed by the $L^{\infty}_{v,x}$ bound of the initial datum. We define some notations in different normed spaces for later use. For given functions $f=f(t,x,v)$ and $g=g(x,v)$, we define
\begin{align*}
&\|f\|_{L^{\infty}_{T}L^{\infty}_{v,x}}:=\sup_{t\in [0,T]} \sup_{v \in \R^3} \sup_{x \in \Omega}|f(t,x,v)| ,\\
&\|f\|_{L^{\infty}_{T_0,T}L^{\infty}_{x}L^{1}_{v}}:=\sup_{T_0\leq t \leq T} \sup_{x \in \Omega}\int_{\R^3}|f(t,x,v)|dv ,\\
&\|g\|_{L^{\infty}_{v,x}}:= \sup_{v \in \R^3} \sup_{x \in \Omega}|g(x,v)|,\\
&\|g\|_{L_x^1L^\infty_v}:= \int_{\Omega} \left(\sup_{v \in \R^3}|g(x,v)| \right) dx.
\end{align*}
If $T_0=0$, we write $\|f\|_{L^{\infty}_{T}L^{\infty}_{x}L^{1}_{v}}$ instead of $\|f\|_{L^{\infty}_{T_0,T}L^{\infty}_{x}L^{1}_{v}}$. In this paper, if a constant $C$ depends on $\be_1,\be_2,\cdots,\be_n$, we denote $C=C(\be_1,\be_2,\cdots,\be_n)$ in order to show the dependence clearly. The velocity weight function is given by $w_\beta(v)=(1+|v|)^\be$. 

Now, two main results of this paper are stated below.

\begin{theorem}[Local existence]\label{local}
Let $-3<\ga<0$ and $\be > 6$. Assume $F_0(x,v):=\mu_{\de,\rho}(v)+\sqrt{\bar{\mu}_{\de,\rho}}f_0(x,v)$ with $0\leq F_0\leq 1/\de$ and $\|w_\be f_0\|_{L^\infty_{v,x}}<\infty$. Then there are constants $C_1=C_1(\be,\ga)>0$ independent of $\rho$, $C_{*\rho}>0$ depending continuously on $\rho\in(0,\infty)$, and a positive time 
\begin{align}\label{T_1}
T_1:=\frac{C_1}{C_{*\rho}\left(1+\|w_\be f_0\|_{L^{\infty}_{v,x}}+\|w_\be f_0\|^2_{L^{\infty}_{v,x}}\right)}
\end{align}
such that the Cauchy problem on the quantum Boltzmann equation \eqref{QBE} has a unique mild solution $F(t,x,v)=\mu_{\de,\rho}(v)+\sqrt{\bar{{\mu}}_{\de,\rho}} f(t,x,v)$, $(t,x,v)\in[0,T_1]\times \Omega \times \R^3$, in the sense of \eqref{mildQBE}, satisfying $0\leq F(t,x,v)\leq 1/\de$ and 
\begin{align}\label{LE}
\left\|w_\be f 
\right\|_{{L^{\infty}_{T_1}L^{\infty}_{v,x}}}\leq 2\left\|w_\be f_0 
\right\|_{L^{\infty}_{v,x}}.
\end{align}
\end{theorem}

\begin{theorem}[Global existence]\label{global}
Let all the assumptions in Theorem \ref{local} be satisfied. Furthermore, for $3/(3+\ga)< p<\infty$, let $\be > \max\{6,16/(5p-1)\}$. There is a constant $C_2=C_2(\ga,\be)>0$ such that for any constant $M\geq 1$, there are constants $\eps=\eps(\ga,\be,M)>0$ and $\bar{M_\rho}>0$ such that if it holds that $\|w_\be f_0\|_{L^{\infty}_{v,x}}\leq M$ and
\begin{align}\label{smallness}
\sup_{0\leq\de\leq1}\left\{ \CE_{\de,\rho}(F_0)+\|f_0\|_{L^1_xL^\infty_v}\right\}\leq \eps \bar{M_\rho},
\end{align}
then the Cauchy problem on the quantum Boltzmann equation \eqref{QBE} has a unique global mild solution
$F(t,x,v)=\mu_{\de,\rho}+\sqrt{\bar{{\mu}}_{\de,\rho}} f(t,x,v)$, $(t,x,v)\in[0,\infty)\times \Omega \times \R^3$, in the sense of \eqref{mildQBE}, satisfying $0\leq F(t,x,v)\leq 1/\de$ and
\begin{align}\label{GE}
\left\|w_\be f
\right\|_{{L^{\infty}_{T}L^{\infty}_{v,x}}}\leq C_2\bar{C}_{*\rho}M^3,
\end{align}
for any $T\geq 0$. Moreover, both $\bar{M_\rho}$ and $\bar{C}_{*\rho}$ depend continuously on $\rho\in (0,\infty)$. $\bar{M_\rho}$ is defined in \eqref{R3barMrho} and \eqref{T3barMrho} for $\R^3$ and $\T^3$,  respectively, and $\bar{C}_{*\rho}$ is defined in \eqref{barC*rho}. 
\end{theorem}

In what follows we give an example of initial data allowing for the large amplitude in space variable. Indeed, for given $\rho>0$ and $\beta>0$ as in Theorem \ref{global}, we let $M\geq C_\beta/\sqrt{\rho}$ be a constant that can be arbitrarily large, where $C_\beta:=\sup_v (1+|v|)^\beta e^{-|v|^2/4}$ is a finite positive constant. For $\eps>0$ and $0<\delta\leq 1$, we define a set of functions as
\begin{equation}
\label{def.phi}
\mathcal{A}_{\eps,\delta}=\{\phi(x): 0<\phi(x)\leq \min\{1+\frac{\rho}{\delta}, 1+\frac{M\sqrt{\rho}}{C_\beta}\},\|\phi \ln \phi\|_{L^1_x}+\|\phi-1\|_{L^1_x}\leq \eps\}.
\end{equation}   
We now take 
\begin{equation}
\label{examp}
F_0(x,v)=\phi(x)\mu_{\delta,\rho}(v)=\frac{\phi(x)}{\delta+\rho e^{\frac{|v|^2}{2}}},\quad \phi(x)\in\mathcal{A}_{\eps,\delta}.
\end{equation}
It is straightforward to verify that $0<F_0(x,v)\leq 1/\delta$ and $\|w_\beta f_0\|_{L^\infty_{v,x}}\leq M$ hold true. Moreover, direct computations give that  
\begin{align*}
\CE_{\de,\rho}(F_0)=&\int_{\Omega}\int_{\R^3}\Big[\phi\mu_{\de,\rho}\log\phi+(\phi-1)\mu_{\de,\rho}\log\mu_{\de,\rho}\notag\\
&\qquad-(\phi-1)\mu_{\de,\rho}\log(1-\mu_{\de,\rho})+\de(1-\de\mu_{\de,\rho})\log(1+\frac{\de}{\rho}e^{-\frac{|v|^2}{2}}(\phi-1))\Big]\,dvdx.
\end{align*}
Noticing $\log (1+s)\leq s$ for any $s>-1$, it then follows that
\begin{equation*}
\CE_{\de,\rho}(F_0)+\|f_0\|_{L^1_xL^\infty_v}\leq C_{\rho} (\|\phi \ln \phi\|_{L^1_x}+\|\phi-1\|_{L^1_x})\leq C_{\rho}\eps,
\end{equation*}
where the upper bound $C_{\rho}\eps$ can be small as long as $\eps>0$ is small enough. Therefore, all the assumptions on $F_0(x,v)$ in Theorem \ref{global} are satisfied. It is obvious to see that such initial data in \eqref{examp} with $\phi(x)$ in \eqref{def.phi} allows for the large amplitude in space variable, cf.~\cite{DHWY,DW}.

For soft potentials, we first consider the structure of the perturbed modified quantum Boltzmann equation around the equilibrium and its relation with the quantum parameter. Due to the lack of spectral gap of the linearized operator, we need to introduce a smooth cut-off function in order to separate the operator into two parts. One part is degenerate and we can let it be small under integral. The other part provides the decay in $v$, and the decay also depends on the quantum parameter. The proof starts with the mild form of the equation which includes three parts, the initial datum, the term containing the linear operator $K_\de$ and the term containing the nonlinear operator $\Ga_\de$. The initial datum will not cause many troubles for us so we only need to take care of the rest two parts. We first deduce an estimate uniformly in $\de$ for the nonlinear part so that this part can be controlled by the product of $\| w_\be f \|_{L^\infty}$ and $\| f \|_{L^1_v}$. Next we need to treat the second part with $K_\de$ in it. Since we divide it into two parts as mentioned above, we can control the degenerate part easily. For the other part, after one time iteration, we can bound it by $\CE_{\de,\rho}(F_0)$ which is defined in \eqref{DefCE}. Then we can expect that if $\CE_{\de,\rho}(F_0)$ and $\| f \|_{L^1_v}$ are both small, one can close our a priori assumption. We require $\CE_{\de,\rho}(F_0)$ to be sufficiently small, but we still need to obtain the smallness of $\| f \|_{L^1_v}$. Actually, $\| f \|_{L^1_v}$ is able to be small as long as $\| f_0 \|_{L^1_xL^\infty_v}$ and $\CE_{\de,\rho}(F_0)$ are both small uniformly in $\de$. Hence, at last we obtain a global solution by closing the a priori estimate. Notice we also track the effect of another parameter $\rho>0$. Then for every smallness or largeness estimate, we need to know exactly how small or how large it is, which needs more accurate calculations.

The paper is organized as follows. In Section 2, we study the linear operator $L_\de$ and give an important lemma on the entropy. In Section 3, we construct an approximation sequence under the assumptions in Theorem \ref{local} to obtain a local solution by taking the limit. Moreover, the limit function satisfies all the estimates in Theorem \ref{local}. In Section 4, we derive the global estimates using the $L^{\infty}_{T}L^{\infty}_{v,x}\cap L^{\infty}_{T}L^{\infty}_{x}L^1_v$ approach to prove Theorem \ref{global}.

\section{Preliminaries}
It is direct to see that when the quantum parameter $\de=0$, the equation \eqref{QBE} will become the Boltzmann equation in the classical sense. In our proof, one of the key points is to estimate the linear operator $L_\delta$ defined in \eqref{DefLde}. We observe that since $\de$ is bounded, $\mu_{\de,\rho}$ is equivalent to the Maxwellian
\begin{align}\label{Defmu0}
\mu_0(v)=e^{-\frac{|v|^2}{2}}
\end{align}
in the following sense
\begin{align*}
\frac{1}{\rho+1}\mu_0(v)\leq\mu_{\de,\rho}(v)\leq \frac{1}{\rho}\mu_0(v).
\end{align*}
The above inequality provides a clue that we can use some known results for the classic Boltzmann equation around the global Maxwellian to help us estimate the operator $L_\de$. Thus, we introduce some properties for the classical Boltzmann equation. The operator $L_0$ is given by
\begin{align*}
L_0=\nu-K_0,
\end{align*}
where
\begin{align*}
\nu(v)=&\int_{\R^3}\int_{\S^2}B(v-u,\theta)\mu_0(u)d\omega du\sim (1+|v|)^\ga,\\
K_{0}f(v)=\int_{\R^3}\int_{\S^2}&B(v-u,\theta)\sqrt{\mu_0(u)}\left( \sqrt{\mu_0(u')}f(v')+\sqrt{\mu_0(v')}f(u')-\sqrt{\mu_0(v)}f(u) \right)d\omega du.
\end{align*}
Then we have the following lemmas. Interested readers may refer to \cite{BPT,Glassey,Guo,DHWY} for details.

\begin{lemma}\label{K}
Let $-3< \ga < 0$ and $\be \in \R$. There exists a function $k=k(v,\eta)$ such that $K_0f(v)$ can be written as
\begin{align*}
K_0f(v)=\int_{\R^3}k(v,\eta)f(\eta)d\eta,
\end{align*}
with estimates
\begin{align*}
\dis |k(v,\eta)|\leq C|v-\eta|^\ga e^{-\frac{|v|^2}{4}}e^{-\frac{|\eta|^2}{4}}+\frac{C(\ga)}{|v-\eta|^\frac{3-\ga}{2}}e^{-\frac{|v-\eta|^2}{8}}e^{-\frac{\left| |v|^2-|\eta|^2\right|^2}{8|v-\eta|^2}},
\end{align*}
and
\begin{align}\label{Prok}
\int_{\R^3}\left| k_w(v,\eta) \right|d\eta \leq C(\ga) (1+|v|)^{-1},
\end{align}
where 
\begin{align}\label{Defkw}
k_w(v,\eta):=k(v,\eta)\cdot \frac{w_\be(v)}{w_\be(\eta)}.
\end{align}
\end{lemma}
Furthermore, as in \cite{SG}, if we introduce a smooth cut-off function $\chi_m=\chi_m(\tau)$ with $0\leq m\leq 1$, $\chi_m(\tau)=1$ for $\tau \leq m$, $\chi_m(\tau)=0$ for $\tau \geq 2m$ and $0\leq \chi_m \leq 1$ for $m\leq \tau \leq 2m$ and define $K^m_0$ and $K^c_0$ by
\begin{align}\label{Km}
K_0^mf(v)=&\int_{\R^3}\int_{\S^2}B(v-u,\theta)\chi_m(|v-u|)\notag\\
&\quad\left( \frac{\sqrt{\mu_{0}(v')}}{\sqrt{\mu_{0}(v)}}\mu_{0}(u')f(v')+\frac{\sqrt{\mu_{0}(u')}}{\sqrt{\mu_{0}(v)}}\mu_{0}(v')f(u') 
-\sqrt{\mu_0(u)}\sqrt{\mu_0(v)}f(u) \right)d\omega du
\end{align}
and
\begin{align}\label{sepaK}
K^c_0=K_0-K_0^m,
\end{align}
then the following lemma holds.

\begin{lemma}\label{KmKc}
Let $-3< \ga < 0$ and $\be \in \R$. For given function $f=f(v)$, it holds that
\begin{align}\label{BKm}
|K_0^mf(v)|\leq &\int_{\R^3}\int_{\S^2}B(v-u,\theta)\chi_m(|v-u|)\notag\\
&\quad\left( \frac{\sqrt{\mu_{0}(v')}}{\sqrt{\mu_{0}(v)}}\mu_{0}(u')\left|f(v')\right|+\frac{\sqrt{\mu_{0}(u')}}{\sqrt{\mu_{0}(v)}}\mu_{0}(v')\left|f(u') \right|
+\sqrt{\mu_0(u)}\sqrt{\mu_0(v)}\left|f(u)\right| \right)d\omega du\notag\\
\leq&Cm^{3+\ga}e^{-\frac{|v|^2}{20}}\|f\|_{L^\infty_v}. 
\end{align}
Moreover, there are functions $l_1(v,\eta)$, $l_2(v,\eta)$, such that $K_0^cf(v)$ can be written as
\begin{align*}
K_0^cf(v)=\int_{\R^3}\left(l_2(v,\eta)-l_1(v,\eta)\right)f(\eta)d\eta,
\end{align*}
with estimates
\begin{align}\label{Kc0es}
\left|K_0^cf(v)\right|\leq&\int_{\R^3}\int_{\S^2}B(v-u,\theta)\left(1-\chi_m(|v-u|)\right)\notag\\
&\quad\left( \frac{\sqrt{\mu_{0}(v')}}{\sqrt{\mu_{0}(v)}}\mu_{0}(u')\left|f(v')\right|+\frac{\sqrt{\mu_{0}(u')}}{\sqrt{\mu_{0}(v)}}\mu_{0}(v')\left|f(u')\right| 
+\sqrt{\mu_0(u)}\sqrt{\mu_0(v)}\left|f(u)\right| \right)d\omega du\notag\\
\leq&\int_{\R^3}l(v,\eta)\left|f(\eta)\right|d\eta,
\end{align}
and
\begin{align}
l(v,\eta)\leq C|v-\eta|^\ga e^{-\frac{|v|^2}{4}}e^{-\frac{|\eta|^2}{4}}&+\frac{C(\ga)}{|v-\eta|^\frac{3-\ga}{2}}e^{-\frac{|v-\eta|^2}{16}}e^{-\frac{\left| |v|^2-|\eta|^2\right|^2}{16|v-\eta|^2}},\label{Prol2}\\
\int_{\R^3}l_w(v,\eta)e^{\frac{|v-\eta|^2}{20}}d\eta &\leq C(\ga) m^{\ga-1} \frac{\nu(v)}{(1+|v|)^2},\label{Prol1}\\
\int_{\R^3} l_w(v,\eta)e^{\frac{|v-\eta|^2}{20}}d\eta& \leq C(\ga) (1+|v|)^{-1}\notag,
\end{align}
where
\begin{align}\label{lveta}
l(v,\eta):=\left|l_2(v,\eta)\right|+\left|l_1(v,\eta)\right|,\end{align} and
\begin{align}\label{Deflw}
l_w(v,\eta):=\left(\left|l_1(v,\eta)\right|+\left|l_2(v,\eta)\right|\right)\cdot \frac{w_\be(v)}{w_\be(\eta)}.
\end{align}
\end{lemma}
Now we can show the linear operator $L_\de$ has similar properties as above. First we will simplify the explicit form of $L_\de f$. The following lemma is based on direct calculations.
\begin{lemma}\label{muequi}
Let $0\leq\de\leq1$ and $\rho>0$. Let $\mu_{\de,\rho}$, $\bar{\mu}_{\de,\rho}$ and $\mu_0$ be defined in \eqref{Defmuderho}, \eqref{Defbarmu} and \eqref{Defmu0}, respectively, then we have
\begin{align*}
 C_{1\rho}\mu_0(u)\leq\mu_{\de,\rho}(u)-\de\mu_{\de,\rho}(u)\mu_{\de,\rho}(u')-\de\mu_{\de,\rho}(u)\mu_{\de,\rho}(v')+\de\mu_{\de,\rho}(u')\mu_{\de,\rho}(v')\leq C_{2\rho}\mu_0(u),\\
\sqrt{\frac{\bar{\mu}_{\de,\rho}(u)}{\bar{\mu}_{\de,\rho}(v)}}\left[\mu_{\de,\rho}(v)-\de\mu_{\de,\rho}(u)\mu_{\de,\rho}(u')-\de\mu_{\de,\rho}(u)\mu_{\de,\rho}(v')+\de\mu_{\de,\rho}(u')\mu_{\de,\rho}(v')\right]\leq C_{2\rho}\sqrt{\mu_0(u)\mu_0(v)},\\
\sqrt{\frac{\bar{\mu}_{\de,\rho}(u')}{\bar{\mu}_{\de,\rho}(v)}}\left[\mu_{\de,\rho}(v')-\de\mu_{\de,\rho}(v')\mu_{\de,\rho}(u)-\de\mu_{\de,\rho}(v')\mu_{\de,\rho}(v)+\de\mu_{\de,\rho}(u)\mu_{\de,\rho}(v)\right]\leq C_{2\rho}\sqrt{\mu_0(u)\mu_0(v')},\\
\sqrt{\frac{\bar{\mu}_{\de,\rho}(v')}{\bar{\mu}_{\de,\rho}(v)}}\left[\mu_{\de,\rho}(u')-\de\mu_{\de,\rho}(u')\mu_{\de,\rho}(u)-\de\mu_{\de,\rho}(u')\mu_{\de,\rho}(v)+\de\mu_{\de,\rho}(u)\mu_{\de,\rho}(v)\right]\leq C_{2\rho}\sqrt{\mu_0(u)\mu_0(u')},
\end{align*}
where 
\begin{align}\label{DefC12rho}
C_{1\rho}=\frac{\rho^2}{\left(\rho+1\right)^3},\ C_{2\rho}=\frac{\rho+1}{\rho^2}.
\end{align}
\end{lemma}
\begin{proof}
We prove the first two inequalities and the rest two can be treated similarly. By the definition of $\mu_{\de,\rho}$ and $\bar{\mu}_{\de,\rho}$ in \eqref{Defmuderho} and \eqref{Defbarmu} and the fact that $ |u'|^2+|v'|^2=|u|^2+|v|^2$, we have
\begin{align*}
\dis &\mu_{\de,\rho}(u)-\de\mu_{\de,\rho}(u)\mu_{\de,\rho}(u')-\de\mu_{\de,\rho}(u)\mu_{\de,\rho}(v')+\de\mu_{\de,\rho}(u')\mu_{\de,\rho}(v')\notag\\
&=\frac{\rho e^{\frac{|u|^2}{2}}\left(\de+\rho e^{\frac{|v|^2}{2}} \right)}{\left(\de+\rho e^{\frac{|u'|^2}{2}} \right)\left(\de+\rho e^{\frac{|v'|^2}{2}} \right)\left(\de+\rho e^{\frac{|u|^2}{2}} \right)}.
\end{align*}
Then the first estimate holds. Then by a similar argument, it follows that
 \begin{align*}
\dis &\sqrt{\frac{\bar{\mu}_{\de,\rho}(u)}{\bar{\mu}_{\de,\rho}(v)}}\left[\mu_{\de,\rho}(v)-\de\mu_{\de,\rho}(u)\mu_{\de,\rho}(u')-\de\mu_{\de,\rho}(u)\mu_{\de,\rho}(v')+\de\mu_{\de,\rho}(u')\mu_{\de,\rho}(v')\right]\notag\\
&=\frac{\rho e^{\frac{|v|^2}{2}}e^{\frac{|u|^2}{4}}} {\left(\de+\rho e^{\frac{|u'|^2}{2}} \right)\left(\de+\rho e^{ \frac{|v'|^2}{2}} \right)e^{\frac{|v|^2}{4}}}\leq \frac{1}{\rho}e^{-\frac{|u|^2}{4}-\frac{|v|^2}{4}}\leq C_{2\rho}\sqrt{\mu_0(u)\mu_0(v)}.
\end{align*}
This ends the proof of Lemma \ref{muequi}.
\end{proof}

We can separate $K_\de$ into two parts as \eqref{sepaK} with the same cut-off function $\chi_m=\chi_m(\tau)$ as in \eqref{Km}. Let $K^m_\de$ be given by
\begin{align}\label{DefKmde}
K^m_{\de}f(v)&=\int_{\R^3}\int_{\S^2}B(v-u,\theta)\chi_m(|v-u|)\notag\\
&\left[\frac{\sqrt{\bar{\mu}_{\de,\rho}(u')}}{\sqrt{\bar{\mu}_{\de,\rho}(v)}}\left[ \mu_{\de,\rho}(v')-\de\mu_{\de,\rho}(v')\mu_{\de,\rho}(u)-\de\mu_{\de,\rho}(v')\mu_{\de,\rho}(v)+\de\mu_{\de,\rho}(u)\mu_{\de,\rho}(v) \right]f(u')\right.\notag\\
&\left.\frac{\sqrt{\bar{\mu}_{\de,\rho}(v')}}{\sqrt{\bar{\mu}_{\de,\rho}(v)}} \left[ \mu_{\de,\rho}(u')-\de\mu_{\de,\rho}(u')\mu_{\de,\rho}(u)-\de\mu_{\de,\rho}(u')\mu_{\de,\rho}(v)+\de\mu_{\de,\rho}(u)\mu_{\de,\rho}(v) \right]f(v')\right.\notag\\
&\left.\frac{\sqrt{\bar{\mu}_{\de,\rho}(u)}}{\sqrt{\bar{\mu}_{\de,\rho}(v)}} \left[ \mu_{\de,\rho}(v)-\de\mu_{\de,\rho}(v)\mu_{\de,\rho}(u')-\de\mu_{\de,\rho}(v)\mu_{\de,\rho}(v')+\de\mu_{\de,\rho}(u')\mu_{\de,\rho}(v') \right]f(u)\right]d\om du,
\end{align}
and write
\begin{align}\label{DefKcde}
K^c_0=K_0-K_0^m.
\end{align}
With the above definitions and lemmas, we can control $L_\de$ by $L_0$ in terms of the lemma below. 
\begin{lemma}\label{KdeK0}
Recall that the operators $\nu_\de$, $K^m_\de$ and $K^c_\de$ are defined in \eqref{Defnude}, \eqref{DefKmde} and \eqref{DefKcde} respectively, and the constants $C_{1\rho}$, $C_{2\rho}$ are defined in \eqref{DefC12rho}. Let $l=l(v,\eta)$, $k=k(v,\eta)$  be the same functions as in Lemma \ref{KmKc} and Lemma \ref{K}. Then for given function $f=f(v)$, there exists a constant $C>0$ which is independent of $\rho$ such that
\begin{align*}
\frac{1}{C} C_{1\rho}(1+|v|)^\ga\leq \nu_\de(v) \leq C C_{2\rho}(1+|v|)^\ga,\\
|K^m_\de f(v)|\leq Cm^{3+\ga} C_{2\rho}e^{-\frac{|v|^2}{20}}\|f\|_{L^\infty_v},\\
\left|K_\de^cf(v)\right|\leq C_{2\rho}\int_{\R^3}\left|l(v,\eta)\right|\left|f(\eta)\right|d\eta,
\end{align*}
and
\begin{align*}
\left|K_\de f(v)\right|\leq C_{2\rho}\int_{\R^3}\left|k(v,\eta)\right|\left|f(\eta)\right|d\eta.
\end{align*}
\end{lemma}
\begin{proof}
By Lemma \ref{muequi} we have
\begin{align*}
C_{1\rho}\int_{\R^3}\int_{\S^2}B(v-u,\theta)\mu_0(u)d\om du\leq \nu_\de(v)\leq C_{2\rho}\int_{\R^3}\int_{\S^2}B(v-u,\theta)\mu_0(u)d\om du,
\end{align*}
and
\begin{align*}
\left|K^m_{\de}f(v)\right|\leq &C_{2\rho}\int_{\R^3}\int_{\S^2}B(v-u,\theta)\chi_m(|v-u|)\notag\\
&\left( \frac{\sqrt{\mu_{0}(v')}}{\sqrt{\mu_{0}(v)}}\mu_{0}(u')\left|f(v')\right|+\frac{\sqrt{\mu_{0}(u')}}{\sqrt{\mu_{0}(v)}}\mu_{0}(v')\left|f(u') \right|
+\sqrt{\mu_0(u)}\sqrt{\mu_0(v)}\left|f(u)\right| \right)d\omega du,\\
\left|K^c_{\de}f(v)\right|\leq &C_{2\rho}\int_{\R^3}\int_{\S^2}B(v-u,\theta)\left(1-\chi_m(|v-u|)\right)\notag\\
&\left( \frac{\sqrt{\mu_{0}(v')}}{\sqrt{\mu_{0}(v)}}\mu_{0}(u')\left|f(v')\right|+\frac{\sqrt{\mu_{0}(u')}}{\sqrt{\mu_{0}(v)}}\mu_{0}(v')\left|f(u') \right|
+\sqrt{\mu_0(u)}\sqrt{\mu_0(v)}\left|f(u)\right| \right)d\omega du.
\end{align*}
In the last two inequalities above, we have used the fact
$$
\sqrt{\mu_0(u)\mu_0(v')}=\frac{\sqrt{\mu_{0}(u')}}{\sqrt{\mu_{0}(v)}}\mu_{0}(v'), \quad \sqrt{\mu_0(u)\mu_0(u')}=\frac{\sqrt{\mu_{0}(v')}}{\sqrt{\mu_{0}(v)}}\mu_{0}(u').
$$
Then by \eqref{BKm}, \eqref{Kc0es} and the fact that $\frac{1}{C} (1+|v|)^\ga\leq\int_{\R^3}\int_{\S^2}B(v-u,\theta)\mu_0(u)d\om du\leq C (1+|v|)^\ga$ for some constant $C>0$ which is independent of $\rho$, the proof is complete.
\end{proof}

Recalling our definition of $\CE_{\de,\rho}$ in \eqref{DefCE}, we have the following lemma which will be used later.
\begin{lemma}\label{Taylor}
Let $F(t,x,v)$ satisfy \eqref{M}, \eqref{E} and \eqref{H}, then we have
\begin{align*}
\int_{\Omega}\int_{\R^3}&\left\{\frac{|F(t,x,v)-\mu_{\de,\rho}(v)|^2}{4\mu_{\de,\rho}(v)}\chi_{\{ |F(t,x,v)-\mu_{\de,\rho}(v)|\leq \mu_{\de,\rho}(v) \}}\right.\notag\\
&\qquad\qquad\qquad \left. + \frac{|F(t,x,v)-\mu_{\de,\rho}(v)|}{4}\chi_{\{ |F(t,x,v)-\mu_{\de,\rho}(v)|\geq \mu_{\de,\rho}(v) \}}\right\}dv dx\leq \CE_{\de,\rho}(F_0).
\end{align*}
\end{lemma}
\begin{proof}
By the entropy inequality \eqref{H} and Taylor expansion, it follows that
\begin{align*}
\CH_{\de,\rho}(F(t))=\int_{\Omega}\int_{\R^3}& \left\{F(t,x,v)\log{F}(t,x,v)+\frac{1}{\de}(1-\de F(t,x,v))\log(1-\de{F}(t,x,v))\right.\notag\\
                                &\qquad \left.-\mu_{\de,\rho}(v)\log{\mu_{\de,\rho}(v)}-\frac{1}{\de}(1-\de\mu_{\de,\rho}(v))\log(1-\de\mu_{\de,\rho}(v))\right\}dvdx\notag\\
=\int_{\Omega}\int_{\R^3} &\left\{\log\frac{\mu_{\de,\rho}(v)}{1-\de \mu_{\de,\rho}(v)}\left(F(t,x,v)-\mu_{\de,\rho}(v)\right)\right.\notag\\
                                &\qquad \left.+\frac{1}{2}\left(\frac{1}{\bar{F}}+\frac{\de}{1-\de \bar{F}} \right)\left|F(t,x,v)-\mu_{\de,\rho}(v) \right|^2 \right\}dvdx\leq \CH_{\de,\rho}(F(0)),
\end{align*}
where $\bar{F}$ is between $F$ and $\mu_{\de,\rho}$. Since $1-\de F\geq 0$ and $1-\de \mu_{\de,\rho}\geq 0$, we have $1-\de \bar{F}\geq 0$. By a direct calculation, we also obtain 
$$
\log\frac{\mu_{\de,\rho}(v)}{1-\de \mu_{\de,\rho}(v)}=-\log \rho-\frac{|v|^2}{2}.
$$
 Then it holds that 
\begin{align*}
&\int_{\Omega}\int_{\R^3}\frac{1}{2\bar{F}}\left|F(t,x,v)-\mu_{\de,\rho}(v) \right|^2 dvdx\leq \int_{\Omega}\int_{\R^3}\left\{\frac{1}{2}\left(\frac{1}{\bar{F}}+\frac{\de}{1-\de \bar{F}} \right)\left|F(t,x,v)-\mu_{\de,\rho}(v) \right|^2 \right\}dvdx\notag\\
&\leq \CH_{\de,\rho}(F(0))+\log{\rho}\int_{\Omega}\int_{\R^3}\left(F(t,x,v)-\mu_{\de,\rho}(v)\right)dvdx+\frac{|v|^2}{2}\int_{\Omega}\int_{\R^3}\left(F(t,x,v)-\mu_{\de,\rho}(v)\right)dvdx\notag\\
&=\CE_{\de,\rho}(F_0).
\end{align*}
In the last equality above, we have used the conservation of mass and energy in \eqref{M} and \eqref{E}. Noticing that $F\geq 0$, then $|F-\mu_{\de,\rho}|\geq 0$ implies $F\geq 2\mu_{\de,\rho}$ or $F=0$. Then for $F\geq 2\mu_{\de,\rho}$, we have
$$
\frac{\left|F(t,x,v)-\mu_{\de,\rho}(v)\right|}{2\bar{F}}\geq \frac{F(t,x,v)-\mu_{\de,\rho}(v)}{2F}=\frac{1}{2}-\frac{\mu_{\de,\rho}}{2F}\geq\frac{1}{4}.
$$
For $F=0$,
$$
\frac{\left|F(t,x,v)-\mu_{\de,\rho}(v)\right|}{2\bar{F}}\geq \frac{\mu_{\de,\rho}(v)}{2\mu_{\de,\rho}}\geq\frac{1}{4}.
$$
Then the lemma follows from the computations above.
\end{proof}

\begin{remark}
In Lemma \ref{Taylor},  if we let $\rho=(2\pi)^\frac{3}{2}$, $\de\rightarrow0$, then it holds that
\begin{align*}
\CH_{\de,\rho}(F(t))&=\int_{\Omega}\int_{\R^3} \left\{F(t,x,v)\log{F}(t,x,v)-F(t,x,v)-\mu_{\de,\rho}(v)\log{\mu_{\de,\rho}(v)}+\mu_{\de,\rho}(v)\right\}dvdx\notag\\
&=\int_{\Omega}\int_{\R^3} \left\{F(t,x,v)\log{F}(t,x,v)-\mu_{\de,\rho}(v)\log{\mu_{\de,\rho}(v)}\right\}dvdx-M_0.
\end{align*}
$\CE_{\de,\rho}(F_0)$ will become
\begin{align*}
\CE_{\de,\rho}(F_0)&=\int_{\Omega}\int_{\R^3} \left\{F(t,x,v)\log{F}(t,x,v)-\mu_{\de,\rho}(v)\log{\mu_{\de,\rho}(v)}\right\}dvdx+(\frac{3}{2}\log 2\pi-1)M_0+\frac{1}{2}E_0.
\end{align*}
This is consistent with Lemma 2.7 in \cite{DHWY}.
\end{remark}

\section{Approximation Sequence and Local Existence}
In this section we prove Theorem \ref{local}. We first construct the approximation sequence $\{F^n \}^{\infty}_{n=1}$ from \eqref{QBE} as follows:
\begin{align}\label{approseq}
\dis \pa_tF^{n+1}+v\cdot \na_x F^{n+1}=\int_{\R^3}\int_{\S^2}B(v-&u,\theta)\left[ F^n(u')F^n(v')\left(1-\de F^n(u) \right)\left(1-\de F^{n+1}(v) \right)\right. \notag\\
 & \left. -F^n(u)F^{n+1}(v)\left(1-\de F^n(u') \right) \left(1-\de F^n(v') \right)\right]d\omega du,
\end{align}
with $F^{n}(0,x,v)=F_0(x,v)$ and $F^0(t,x,v)=0$.
We will prove that under the assumptions in Theorem \ref{local}, our approximation sequence $F^n=\mu_{\de,\rho}+\sqrt{\bar{\mu}_{\de,\rho}}f^n$ satisfies the pointwise bound $0\leq F^n \leq \frac{1}{\de}$. Moreover, for $T_1$ defined in \eqref{T_1}, it holds that
\begin{align*}
\left\|w_\be f^n 
\right\|_{{L^{\infty}_{T_1}L^{\infty}_{v,x}}}\leq 2\left\|w_\be f_0 
\right\|_{L^{\infty}_{v,x}},
\end{align*}
and $\{w_\be f^n\}_{n=1}^\infty$ is a Cauchy sequence in $L^\infty([0,T_1]\times \Omega \times \R^3)$ with a unique limit. By passing the limit, we can obtain a function $f=\lim_{n\rightarrow \infty}f^n$ which is a mild solution to the quantum Boltzmann equation in the sense of \eqref{mildQBE}. We first prove the pointwise bound of $F^n$.

\begin{lemma}\label{positivity}
The sequence $\{F^n \}^{\infty}_{n=1}$ which is constructed in \eqref{approseq} satisfies $0\leq F^n \leq \frac{1}{\de}$ for $n=0,1,2,\cdots$.
\end{lemma}
\begin{proof}
When $n=0$, $0= F^0\leq 1$. We rewrite \eqref{approseq} as
 \begin{align}\label{reFn}
\dis \pa_tF^{n+1}+v\cdot \na_x F^{n+1}+g^n_1\cdot F^{n+1}   =\tilde{C}_1(F^n),
\end{align}
where
 \begin{align}\label{gn1}
\dis g^n_1(t,x,v)=\int_{\R^3}\int_{\S^2}B(v-u,\theta)&\left[\de F^n(t,x,u')F^n(t,x,v')\left(1-\de F^n(t,x,u) \right)\right. \notag\\
 &\quad \left.+F^n(t,x,u)\left(1-\de F^n(t,x,u') \right) \left(1-\de F^n(t,x,v') \right)\right]d\omega du,
\end{align} and
 \begin{align}\label{tildeC1}
\dis \tilde{C}_1(F^n)=\int_{\R^3}\int_{\S^2}B(v-u,\theta) F^n(u')F^n(v')\left(1-\de F^n(u) \right)d\omega du.
\end{align}
We define $G^n(t,x,v)=1-\de F^n(t,x,v)$. When $n=0$, $0\leq G^n(t,x,v)=1$. By direct calculation we have
 \begin{align}\label{EqG}
\dis &\pa_tG^{n+1}+v\cdot \na_x G^{n+1}=-\de\left( \pa_tF^{n+1}+v\cdot \na_x F^{n+1} \right)\notag\\
&=-\de \int_{\R^3}\int_{\S^2}B(v-u,\theta)\left[ F^n(u')F^n(v')G^n(u)G^{n+1}(v)-F^n(u)\frac{1-G^{n+1}(v)}{\de}G^n(u') G^n(v')\right]d\omega du\notag\\
&=-g^n_2\cdot G^{n+1}+\tilde{C}_2(G^n),\end{align}
where
 \begin{align}\label{gn2}
\dis g^n_2(t,x,v)=& \int_{\R^3}\int_{\S^2}B(v-u,\theta)\notag\\
&\quad\left[\de F^n(t,x,u')F^n(t,x,v')G^n(t,x,u)+F^n(t,x,u)G^n(t,x,u') G^n(t,x,v')\right]d\omega du,
\end{align} and
 \begin{align}\label{tildeC2}
\dis \tilde{C}_2(G^n)=\int_{\R^3}\int_{\S^2}B(v-u,\theta)F^n(u)G^n(u') G^n(v')d\omega du.
\end{align}
Similarly as above, \eqref{EqG} can be written as 
 \begin{align}\label{reGn}
\dis \pa_tG^{n+1}+v\cdot \na_x G^{n+1}+g^n_2\cdot G^{n+1}   =\tilde{C}_2(G^n).
\end{align}
For \eqref{reFn} and \eqref{reGn}, we integrate along the characteristics to get
 \begin{align}\label{mildFn}
\dis F^{n+1}(t,x,v)=e^{-g^n_1(t,x,v)}F_0(x-vt,v)+\int^t_0e^{-\int^t_sg^n_1(\tau,x-v(t-\tau),v)d\tau}\tilde{C}_1(F^n)(s,x-v(t-s),v)ds,
\end{align}
and
 \begin{align}\label{mildGn}
\dis G^{n+1}(t,x,v)=e^{-g^n_2(t,x,v)}G_0(x-vt,v)+\int^t_0e^{-\int^t_sg^n_2(\tau,x-v(t-\tau),v)d\tau}\tilde{C}_2(G^n)(s,x-v(t-s),v)ds,
\end{align}
where 
$$
G_0(x,v)=(1-\de F_0)(x,v).
$$
We prove our lemma by induction. If $0\leq F^n \leq \frac{1}{\de}$, then $0\leq G^n \leq \frac{1}{\de}$, then $g^n_1(t,x,v)\geq 0$, $g^n_2(t,x,v)\geq 0$, $\tilde{C}_1(G^n)(t,x,v)\geq 0$, $\tilde{C}_2(G^n)(t,x,v)\geq 0$ from \eqref{gn1}, \eqref{gn2}, \eqref{tildeC1} and \eqref{tildeC2} respectively. Hence, $F^{n+1}(t,x,v)\geq 0$ and $G^{n+1}(t,x,v)\geq 0$ from \eqref{mildFn} and \eqref{mildGn}. Then we conclude that $0\leq F^{n+1}(t,x,v) \leq \frac{1}{\de}$. Moreover, from our proof, we can also obtain
\begin{align}\label{+gn}
\dis g^n_1(t,x,v)\geq 0,
\end{align}
for $n=0,1,2,\cdots$.
\end{proof}
Substituting $F^n=\mu_{\de,\rho}+\sqrt{\bar{\mu}_{\de,\rho}}f^n$ into the mild form \eqref{mildFn}, we get
\begin{align}\label{mildfn}
\dis w_\be(v) f^{n+1}(t,x,v)=&e^{-{\int^t_0g_1^n(\tau,x-v(t-\tau),v)d\tau}}w_\be(v) f_0(x-vt,v) \notag\\
&+\int_0^t e^{-{\int^t_sg_1^n(\tau,x-v(t-\tau),v)d\tau}}w_\be(v) (K_\de f^n)(s,x-v(t-s),v)ds \notag\\
&+\int_0^t e^{-{\int^t_sg_1^n(\tau,x-v(t-\tau),v)d\tau}}w_\be(v)\Ga_{\de +}(f^n)(s,x-v(t-s),v)ds,
\end{align}
where $K_\de$ is defined in \eqref{DefKde} and
\begin{align}\label{defgade+}
	\dis
	\Ga_{\de +} (f)(t,x,v)=\frac{1}{\sqrt{\bar{\mu}_{\de,\rho}(v)}}&\int_{\R^3}\int_{\S^2}B(v-u,\theta)\notag\\
	&\left[\sqrt{\bar{\mu}_{\de,\rho}}f(t,x,u')\sqrt{\bar{\mu}_{\de,\rho}}f(t,x,v')\left(1-\de\mu_{\de,\rho}(v)-\de\mu_{\de,\rho}(u)\right)\right.\notag\\
	&\quad +\de\sqrt{\bar{\mu}_{\de,\rho}}f(t,x,v')\sqrt{\bar{\mu}_{\de,\rho}}f(t,x,u)\left(\mu_{\de,\rho}(v)-\mu_{\de,\rho}(u')\right)\notag\\
	&\quad+\de\sqrt{\bar{\mu}_{\de,\rho}}f(t,x,u')\sqrt{\bar{\mu}_{\de,\rho}}f(t,x,u)\left(\mu_{\de,\rho}(v)-\mu_{\de,\rho}(v')\right)\notag\\
	&\quad\left.-\de\sqrt{\bar{\mu}_{\de,\rho}}f(t,x,u')\sqrt{\bar{\mu}_{\de,\rho}}f(t,x,v')\sqrt{\bar{\mu}_{\de,\rho}}f(t,x,u)\right]d\om du.
\end{align}
 By \eqref{gn1}, we rewrite $g_1^n$ as
\begin{align}\label{defgn1}
\dis
&g^n_1(t,x,v)=\int_{\R^3}\int_{\S^2}B(v-u,\theta)\left[\sqrt{\bar{\mu}_{\de,\rho}}f^n(u)\left(1-\de\mu_{\de,\rho}(v')-\de\mu_{\de,\rho}(u')\right)
   \right.\notag\\
&\quad-\de\sqrt{\bar{\mu}_{\de,\rho}}f^n(v')\left(\mu_{\de,\rho}(u)-\mu_{\de,\rho}(u')\right)-\de\sqrt{\bar{\mu}_{\de,\rho}}f^n(u')\left(\mu_{\de,\rho}(u)-\mu_{\de,\rho}(v')\right)\notag\\
&\quad-\de\sqrt{\bar{\mu}_{\de,\rho}}f^n(u)\sqrt{\bar{\mu}_{\de,\rho}}f^n(u')-\de\sqrt{\bar{\mu}_{\de,\rho}}f^n(u)\sqrt{\bar{\mu}_{\de,\rho}}f^n(v')\notag\\
&\quad\left.+\de\sqrt{\bar{\mu}_{\de,\rho}}f^n(u')\sqrt{\bar{\mu}_{\de,\rho}}f^n(v')\right]d\om du+\nu_\de(v),
\end{align}
where $\nu_\de(v)$ is defined in \eqref{Defnude}.
Using properties of $g_1^n$, $K_\de$ and $\Ga_{\de +}$, we prove the following lemma, which implies the boundedness of $w_\be f^n$.

\begin{lemma}\label{Localbound}
There exists a constant $C'=C'(\be,\ga)$ independent of $\rho$ and a positive time $T'_1$ defined as 
\begin{align*}
T'_1=\frac{C'}{C_{5\rho}\left( 1+ \left\|w_\be f_0 \right\|_{{L^{\infty}_{T}L^{\infty}_{v,x}}}+\left\|w_\be f_0 \right\|_{{L^{\infty}_{T}L^{\infty}_{v,x}}}^2\right)},
\end{align*}
for some constant $C_{5\rho}$ which is given in \eqref{C5rho}, such that we have
\begin{align}\label{LELemma}
\left\|w_\be f^n \right\|_{{L^{\infty}_{T'}L^{\infty}_{v,x}}}\leq 2\left\|w_\be f_0 \right\|_{L^{\infty}_{v,x}},
\end{align}
for $n=0,1,2,\cdots$.
\end{lemma}
\begin{proof}
For $(t,x,v)\in[0,T]\times \Omega \times \R^3$, since we have Lemma \ref{positivity} for the positivity of $g_1^n$, it follows from \eqref{mildfn} that
\begin{align} \label{LocalPoint}
\dis \left|w_\beta(v) f^{n+1}(t,x,v) \right|&\leq \left\|w_\be f_0 \right\|_{L^{\infty}_{v,x}}+\int_0^t \left|w_\beta(v)(K_\de f^n)(s,x-v(t-s),v)\right|ds\notag \\
&\quad+\int_0^t \left|w_\beta(v)\Ga_{\de +}(f^n)(s,x-v(t-s),v)\right|ds\notag \\
&=\left\|w_\be f_0 \right\|_{L^{\infty}_{v,x}}+I_1+I_2.
\end{align}
Now we estimate $I_1$. By the definition of $K_\de$ \eqref{DefKde} and Lemma \ref{muequi}, it holds that
\begin{align}\label{estI1}
\dis I_1&=\int_0^t \left|w_\beta(v)(K_\de f^n)(s,x-v(t-s),v)\right|ds\notag\\
&\leq C\, C_{2\rho}\int_0^t \int_{\R^3}\left|k_w(v,\eta)\right|\left|w_\be(\eta) f^n(s,x-v(t-s),\eta)\right|d\eta ds\notag\\
&\leq C\, C_{2\rho}\left\|w_\be f^n \right\|_{{L^{\infty}_{T}L^{\infty}_{v,x}}}\int_0^t \int_{\R^3}\left|k_w(v,\eta)\right|d\eta ds\notag\\
&\leq CT C_{2\rho}\left\|w_\be f^n \right\|_{{L^{\infty}_{T}L^{\infty}_{v,x}}},
\end{align}
where $k_w(v,\eta)$ is defined in \eqref{Defkw}, $C_{2\rho}$ is given in \eqref{DefC12rho}. In the last inequality above, we have used the fact that $\int_{\R^3}\left|k_w(v,\eta)\right|d\eta$ is finite uniformly in $v$ by \eqref{Prok}.

For $I_2$, denoting $x_1=x-v(t-s)$, we first observe from \eqref{defgade+} that 
\begin{align*}
\dis
&\left|w_\be(v) \Ga_{\de +} (f^n)(s,x_1,v)\right|\notag\\
&\leq \frac{C}{\sqrt{\bar{\mu}_{\de,\rho}(v)}}\int_{\R^3}\int_{\S^2}B(v-u,\theta)\big{[}w_\be(v)\left|\sqrt{\bar{\mu}_{\de,\rho}}f^n(s,x_1,u')\right|\left|\sqrt{\bar{\mu}_{\de,\rho}}f^n(s,x_1,v')\right| \notag\\
&\qquad+w_\be(v)\left|\sqrt{\bar{\mu}_{\de,\rho}}f^n(s,x_1,v')\right|\left|\sqrt{\bar{\mu}_{\de,\rho}}f^n(s,x_1,u)\right|\left(\mu_{\de,\rho}(v)+\mu_{\de,\rho}(u')\right)\notag\\
&\qquad +w_\be(v)\left|\sqrt{\bar{\mu}_{\de,\rho}}f^n(s,x_1,u')\right|\left|\sqrt{\bar{\mu}_{\de,\rho}}f^n(s,x_1,u)\right|\left(\mu_{\de,\rho}(v)+\mu_{\de,\rho}(v')\right)\notag\\
&\qquad+w_\be(v) \rho^{-\frac{1}{2}}\left|\sqrt{\bar{\mu}_{\de,\rho}}f^n(s,x_1,u')\right|\left|\sqrt{\bar{\mu}_{\de,\rho}}f^n(s,x_1,v')\right|\left\|w_\be f^n \right\|_{{L^{\infty}_{T}L^{\infty}_{v,x}}}\big{]}d\om du\notag\\
&=I_{21}+I_{22}+I_{23}+I_{24}
\end{align*}
by the fact that $0\leq \de\leq 1$, $0\leq \de\mu_{\de,\rho}(v)=\de/\left(\de+\rho e^\frac{|v|^2}{2}\right)\leq 1$ and $\sqrt{\bar{{\mu}}_{\de,\rho}}(v)=\sqrt\rho e^{{|v|^2}/{4}}/(\de+\rho e^{{|v|^2}/{2}})\leq \rho^{-\frac{1}{2}}$. Moreover, by \eqref{velocity} we have
\begin{align}\label{mumu/mu}
\dis \frac{\sqrt{\bar{{\mu}}_{\de,\rho}}(u')\sqrt{\bar{{\mu}}_{\de,\rho}}(v')}{\sqrt{\bar{{\mu}}_{\de,\rho}}(v)}&=\frac{\sqrt\rho e^{{|u'|^2}/{4}}}{\de+\rho e^{{|u'|^2}/{2}}}\frac{\sqrt\rho e^{{|v'|^2}/{4}}}{\de+\rho e^{{|v'|^2}/{2}}}\frac{\de+\rho e^{{|v|^2}/{2}}}{\sqrt\rho e^{{|v|^2}/{4}}}\notag\\
&\leq C_{3\rho}e^{-{|u'|^2}/{4}}e^{-{|v'|^2}/{4}}e^{{|v|^2}/{4}}\notag\\
&=C_{3\rho}e^{-{|u|^2}/{4}},
\end{align}
where 
\begin{align}\label{C3rho}
C_{3\rho}=\frac{\sqrt\rho(\rho+1)}{\rho^2}.
\end{align}
Using \eqref{velocity} one more time, since $|v|^2\leq |u'|^2+|v'|^2$, we have $w_\be(v)\leq Cw_\be(u')w_\be(v')$ for some constant $C$, which yields
\begin{align}\label{I31I34}
\dis
I_{21}&+I_{24}\notag\\&= \frac{C\left(1+\rho^{-\frac{1}{2}}\left\|w_\be f^n \right\|_{{L^{\infty}_{T}L^{\infty}_{v,x}}}\right)}{\sqrt{\bar{\mu}_{\de,\rho}(v)}}\notag\\
&\qquad\qquad\times\int_{\R^3}\int_{\S^2}B(v-u,\theta)e^{-\frac{|u|^2}{4}}\big{|}w_\be(v)\sqrt{\bar{\mu}_{\de,\rho}}f^n(s,x_1,u')\sqrt{\bar{\mu}_{\de,\rho}}f^n(s,x_1,v')\big{|}d\om du \notag\\
&\leq C\left(1+\rho^{-\frac{1}{2}}\left\|w_\be f^n \right\|_{{L^{\infty}_{T}L^{\infty}_{v,x}}}\right)C_{3\rho}\notag\\
&\qquad\qquad\times\int_{\R^3}\int_{\S^2}|v-u|^\ga e^{-\frac{|u|^2}{4}}|w_\be(u') f^n(s,x_1,u')w_\be(v') f^n(s,x_1,v')|d\om du\notag\\
&\leq C\left(1+\rho^{-\frac{1}{2}}\left\|w_\be f^n \right\|_{{L^{\infty}_{T}L^{\infty}_{v,x}}}\right)C_{3\rho}\left\|w_\be f^n \right\|_{{L^{\infty}_{T}L^{\infty}_{v,x}}}^2.
\end{align}
We now turn to $I_{22}$ and $I_{23}$. Recall
\begin{align}\label{defI32}
\dis
I_{22}&= \frac{C}{\sqrt{\bar{\mu}_{\de,\rho}(v)}}\int_{\R^3}\int_{\S^2}B(v-u,\theta)w_\be(v)\notag\\
&\qquad\qquad\qquad\qquad\left|\sqrt{\bar{\mu}_{\de,\rho}}f^n(s,x_1,v')\right|\left|\sqrt{\bar{\mu}_{\de,\rho}}f^n(s,x_1,u)\right|\left(\mu_{\de,\rho}(v)+\mu_{\de,\rho}(u')\right)d\om du.
\end{align}
 Noticing we have $\mu_{\de,\rho}(v)=\left(\de+\rho e^\frac{|v|^2}{2}\right)^{-1}\leq \rho^{-1}e^{-\frac{|v|^2}{2}}$ and $w_{\be}(v)=(1+|v|^2)^{\frac{\be}{2}}\leq C(1+|u'|^2)^{\frac{\be}{2}}(1+|v'|^2)^{\frac{\be}{2}}$, similar calculations as \eqref{mumu/mu} and \eqref{I31I34} show that
\begin{align}\label{I321}
&w_\be(v)\frac{\sqrt{\bar{\mu}_{\de,\rho}}(v')\sqrt{\bar{\mu}_{\de,\rho}}(u)}{\sqrt{\bar{\mu}_{\de,\rho}}(v)}|f^n(v')||f^n(u)|\mu_{\de,\rho}(v)\notag\\
&\leq C\,\frac{C_{3\rho}}{\rho}e^{-\frac{|v'|^2}{4}-\frac{|u|^2}{4}+\frac{|v|^2}{4}-\frac{|v|^2}{2}}w_\be(v)\left\|w_\be f^n \right\|_{{L^{\infty}_{T}L^{\infty}_{v,x}}}^2\notag\\
&\leq C\,\frac{C_{3\rho}}{\rho}e^{-\frac{|u|^2}{4}}\left\|w_\be f^n \right\|_{{L^{\infty}_{T}L^{\infty}_{v,x}}}^2,
\end{align}
and
\begin{align}\label{I322}
&w_\be(v)\frac{\sqrt{\bar{\mu}_{\de,\rho}}(v')\sqrt{\bar{\mu}_{\de,\rho}}(u)}{\sqrt{\bar{\mu}_{\de,\rho}}(v)}|f^n(s,x_1,v')||f^n(s,x_1,u)|\mu_{\de,\rho}(u')\notag\\
&\leq C\,\frac{C_{3\rho}}{\rho}e^{-\frac{|v'|^2}{4}-\frac{|u|^2}{4}+\frac{|v|^2}{4}-\frac{|u'|^2}{2}}w_\be(u')|w_\be(v')f^n(s,x_1,v')||f^n(s,x_1,u)|\notag\\
&\leq C\,\frac{C_{3\rho}}{\rho}e^{-\frac{|v'|^2}{4}-\frac{|u|^2}{4}+\frac{|v|^2}{4}-\frac{|u'|^2}{4}}\left\|w_\be f^n \right\|_{{L^{\infty}_{T}L^{\infty}_{v,x}}}^2\leq C\,\frac{C_{3\rho}}{\rho}e^{-\frac{|u|^2}{4}}\left\|w_\be f^n \right\|_{{L^{\infty}_{T}L^{\infty}_{v,x}}}^2.
\end{align}
Substituting \eqref{I321} and \eqref{I322} into \eqref{defI32}, it holds that
\begin{align}\label{I32}
\dis
I_{22}&\leq C\, C_{4\rho}\left\|w_\be f^n \right\|_{{L^{\infty}_{T}L^{\infty}_{v,x}}}^2\int_{\R^3}\int_{\S^2}|v-u|^\gamma e^{-\frac{|u|^2}{4}}d\om du\notag\\
&\leq C\, C_{4\rho}\left\|w_\be f^n \right\|_{{L^{\infty}_{T}L^{\infty}_{v,x}}}^2,
\end{align}
where 
\begin{align}\label{C4rho}
C_{4\rho}=\frac{C_{3\rho}}{\rho}=\frac{\sqrt\rho(\rho+1)}{\rho^3}.
\end{align}
Recall 
\begin{align*}
\dis
I_{23}&= \frac{C}{\sqrt{\bar{\mu}_{\de,\rho}(v)}}\int_{\R^3}\int_{\S^2}B(v-u,\theta)w_\be(v)\notag\\
&\qquad\qquad\qquad\qquad\left|\sqrt{\bar{\mu}_{\de,\rho}}f^n(s,x_1,u')\right|\left|\sqrt{\bar{\mu}_{\de,\rho}}f^n(s,x_1,u)\right|\left(\mu_{\de,\rho}(v)+\mu_{\de,\rho}(v')\right)d\om du.
\end{align*}
Similarly as \eqref{I321} and \eqref{I322}, we have
\begin{align}\label{I331}
&w_\be(v)\frac{\sqrt{\bar{\mu}_{\de,\rho}}(u')\sqrt{\bar{\mu}_{\de,\rho}}(u)}{\sqrt{\bar{\mu}_{\de,\rho}}(v)}|f^n(s,x_1,u')||f^n(s,x_1,u)|\left(\mu_{\de,\rho}(v)+\mu_{\de,\rho}(v')\right)\notag\\
&\leq C\,C_{4\rho}\left(e^{-\frac{|u|^2}{4}}|f^n(s,x_1,v')||f^n(s,x_1,u)|\right.\notag\\
&\qquad\qquad\qquad\left.+e^{-\frac{|u'|^2}{4}-\frac{|u|^2}{4}+\frac{|v|^2}{4}-\frac{|v'|^2}{2}}w_\be(v')|w_\be(u')f^n(s,x_1,u')||f^n(s,x_1,u)|\right)\notag\\
&\leq C\, C_{4\rho}e^{-\frac{|u|^2}{4}}\left\|w_\be f^n \right\|_{{L^{\infty}_{T}L^{\infty}_{v,x}}}^2,
\end{align}
which implies 
\begin{align}\label{I33}
\dis
I_{23}\leq C\, C_{4\rho}\left\|w_\be f^n \right\|_{{L^{\infty}_{T}L^{\infty}_{v,x}}}^2.
\end{align}
Combining \eqref{I31I34}, \eqref{I32} and \eqref{I33}, we obtain
\begin{align*}
\dis
&\left|w_\be(v)\Ga_{\de +} (f^n)(s,x_1,v)\right|\notag\\&\qquad\qquad\leq C\,C_{3\rho}\left(1+\rho^{-\frac{1}{2}}\left\|w_\be f^n\right\|_{{L^{\infty}_{T}L^{\infty}_{v,x}}}\right)\left\|w_\be f^n \right\|_{{L^{\infty}_{T}L^{\infty}_{v,x}}}^2+C\, C_{4\rho}\left\|w_\be f^n \right\|_{{L^{\infty}_{T}L^{\infty}_{v,x}}}^2.
\end{align*}
Thus, from the above estimate and \eqref{LocalPoint} we have
\begin{align} \label{estI2}
\dis I_2&= \int_0^t \left|w_\beta(v)\Ga_{\de +}(f^n)(s,x-v(t-s),v)\right|ds\notag \\
&\leq CTC_{3\rho}\left(1+\rho^{-\frac{1}{2}}\left\|w_\be f^n \right\|_{{L^{\infty}_{T}L^{\infty}_{v,x}}}\right)\left\|w_\be f^n \right\|_{{L^{\infty}_{T}L^{\infty}_{v,x}}}^2+CT C_{4\rho}\left\|w_\be f^n \right\|_{{L^{\infty}_{T}L^{\infty}_{v,x}}}^2.
\end{align}
We observe that $C_{3\rho}\rho^{-\frac{1}{2}}=C_{2\rho}$. By \eqref{LocalPoint}, \eqref{estI1} and \eqref{estI2}, it follows that
\begin{align}\label{localbound}
\dis \left\|w_\be f^{n+1} \right\|_{{L^{\infty}_{T}L^{\infty}_{v,x}}}&\leq \left\|w_\be f_0 \right\|_{L^{\infty}_{v,x}}+CT C_{2\rho}\left\|w_\be f^n \right\|_{{L^{\infty}_{T}L^{\infty}_{v,x}}}\notag\\
&\quad+CTC_{3\rho}\left(1+\rho^{-\frac{1}{2}}\left\|w_\be f^n \right\|_{{L^{\infty}_{T}L^{\infty}_{v,x}}}\right)\left\|w_\be f^n \right\|_{{L^{\infty}_{T}L^{\infty}_{v,x}}}^2+CT C_{4\rho}\left\|w_\be f^n \right\|_{{L^{\infty}_{T}L^{\infty}_{v,x}}}^2\notag\\
&\leq \left\|w_\be f_0 \right\|_{L^{\infty}_{v,x}}+CTC_{5\rho}\left(\left\|w_\be f^n \right\|_{{L^{\infty}_{T}L^{\infty}_{v,x}}}+\left\|w_\be f^n \right\|_{{L^{\infty}_{T}L^{\infty}_{v,x}}}^2+\left\|w_\be f^n \right\|_{{L^{\infty}_{T}L^{\infty}_{v,x}}}^3  \right),
\end{align}
where \begin{align}\label{C5rho}C_{5\rho}= C_{2\rho}+C_{3\rho}+C_{4\rho}.\end{align}
We choose
\begin{align}\label{defTprime1}
T'_1=\frac{1}{16C\,C_{5\rho}\left(1+ \left\|w_\be f_0 \right\|_{{L^{\infty}_{v,x}}}+\left\|w_\be f_0 \right\|_{{L^{\infty}_{v,x}}}^2\right)}.
\end{align}
Then \eqref{LELemma} holds from \eqref{localbound} and \eqref{defTprime1}.
\end{proof}
Based on Lemma \ref{Localbound}, we can prove that the sequence $\{w_\be f^n\}_{n=1}^\infty$ is a Cauchy sequence for $(t,x,v)\in [0,T''_1]\times \Omega \times \R^3$ where $T''_1$ is defined in the following lemma.
\begin{lemma}\label{cauchyseq}
There exists a constant $C''=C''(\be,\ga)$ which is independent of $\rho$ such that
\begin{align}\label{CSLemma}
\left\|w_\be f^{n+2}- w_\be f^{n+1}\right\|_{{L^{\infty}_{T''}L^{\infty}_{v,x}}}\leq \frac{1}{2}\left\|w_\be f^{n+1}- w_\be f^{n}\right\|_{{L^{\infty}_{T''}L^{\infty}_{v,x}}},
\end{align}
for $n=0,1,2,\cdots$, where the positive time $T''_1$ is defined by
\begin{align*}
T''_1=\min\left\{T'_1,\frac{C''}{C_{6\rho}\left(1+ \left\|w_\be f_0 \right\|_{{L^{\infty}_{T}L^{\infty}_{v,x}}}+\left\|w_\be f_0 \right\|_{{L^{\infty}_{T}L^{\infty}_{v,x}}}^2 \right)}\right\},
\end{align*}
and $C_{6\rho}$ is given in \eqref{C6rho}. 
\end{lemma}
\begin{proof}
We first assume that $(t,x,v)\in [0,T]\times \Omega \times \R^3$ for $T\leq T'_1$ where $T'_1$ is defined in \eqref{defTprime1} so that Lemma \ref{Localbound} holds. By the mild form \eqref{mildfn}, we can write $w_\be f^{n+2}- w_\be f^{n+1}$ as
\begin{align*}
\dis &w_\beta(v)(f^{n+2}-f^{n+1})(t,x,v)\notag\\
&=  w_\beta(v) f_0(x-vt,v)\left( e^{-{\int^t_0g_1^{n+1}(\tau,x-v(t-\tau),v)d\tau}}-e^{-{\int^t_0g_1^n(\tau,x-v(t-\tau),v)d\tau}}  \right)\notag\\
&\quad +\int_0^t w_\beta(v)(K_\de f^{n+1})(s,x-v(t-s),v)\left(e^{-{\int^t_sg_1^{n+1}(\tau,x-v(t-\tau),v)d\tau}}-e^{-{\int^t_sg_1^n(\tau,x-v(t-\tau),v)d\tau}}  \right) ds\notag\\
&\quad +\int_0^t w_\beta(v)\Ga_{\de+}(f^{n+1})(s,x-v(t-s),v)\notag \\
& \quad\quad \times\left(e^{-{\int^t_sg_1^{n+1}(\tau,x-v(t-\tau),v)d\tau}}-e^{-{\int^t_sg_1^n(\tau,x-v(t-\tau),v)d\tau}}\right)ds\notag \\
&\quad +\int_0^t e^{-{\int^t_0g_1^n(\tau,x-v(t-\tau),v)d\tau}} w_\beta(v)\left(K_\de f^{n+1}-K_\de f^{n}\right)(s,x-v(t-s),v) ds \notag\\
&\quad +\int_0^t e^{-{\int^t_0g^n_1(\tau,x-v(t-\tau),v)d\tau}} w_\beta(v)\left(\Ga_{\de+}(f^{n+1})-\Ga_{\de+}(f^{n})\right)(s,x-v(t-s),v) ds.
\end{align*}
By the positivity of $g_1^n$ given in \eqref{+gn} and the fact that $|e^{-a}-e^{-b}|\leq|a-b|$ for any $a,b\geq0$, it follows that
\begin{align}\label{diff}
\dis & \left| w_\beta(v)(f^{n+2}-f^{n+1})(t,x,v) \right|\notag\\
&\leq \left|w_\beta(v) f_0(x-vt,v)\right|\int^t_0\left| (g_1^{n+1}-g_1^n)(\tau,x-v(t-\tau),v) \right|d\tau \notag\\
&\ \ +\int_0^t \left|w_\beta(v)(K_\de f^{n+1})(s,x-v(t-s),v)\right|\int^t_s\left| (g_1^{n+1}-g_1^n)(\tau,x-v(t-\tau),v) \right|d\tau ds\notag\\
&\ \ +\int_0^t \left|w_\beta(v)\Ga_{\de+}(f^{n+1})(s,x-v(t-s),v)\right|\int^t_s\left| (g_1^{n+1}-g_1^n)(\tau,x-v(t-\tau),v) \right|d\tau ds\notag \\
&\ \ +\int_0^t \left|w_\beta(v)\left(K_\de f^{n+1}-K_\de f^{n}\right)(s,x-v(t-s),v)\right| ds\notag \\
&\ \ +\int_0^t \left|w_\beta(v)\left(\Ga_{\de+}(f^{n+1})-\Ga_{\de+}(f^{n})\right)(s,x-v(t-s),v)\right| ds\notag\\
&=I_{3}+I_{4}+I_{5}+I_{6}+I_{7}.
\end{align}
We consider $\int^t_0\left| (g_1^{n+1}-g_1^{n})(\tau,x-v(t-\tau),v) \right|d\tau$ first. From the definition of $g^n_1$ \eqref{defgn1}, we directly obtain that 
\begin{align}\label{1stdiff}
\dis &\int^t_0\left| (g_1^{n+1}-g_1^n)(\tau,x-v(t-\tau),v) \right|d\tau\notag\\
&\leq\int^t_0\int_{\R^3}\int_{\S^2}B(v-u,\theta)\left[\sqrt{\bar{\mu}_{\de,\rho}}(u)\left\|(f^{n+1}- f^{n})(\tau,u)\right\|_{{L^{\infty}_{x}}}\left|1-\de\mu_{\de,\rho}(v')-\de\mu_{\de,\rho}(u')\right|
   \right.\notag\\
&\qquad\qquad+\de\sqrt{\bar{\mu}_{\de,\rho}}(v')\left\|(f^{n+1}- f^{n})(\tau,v')\right\|_{{L^{\infty}_{x}}}\left(\mu_{\de,\rho}(u)+\mu_{\de,\rho}(u')\right)\notag\\
&\qquad\qquad+\de\sqrt{\bar{\mu}_{\de,\rho}}(u')\left\| (f^{n+1}- f^{n})(\tau,u')\right\|_{{L^{\infty}_{x}}}\left(\mu_{\de,\rho}(u)+\mu_{\de,\rho}(v')\right)\notag\\
&\qquad\qquad+\de\sqrt{\bar{\mu}_{\de,\rho}}(u)\sqrt{\bar{\mu}_{\de,\rho}}(u')\left\|f^{n+1}(\tau,u)f^{n+1}(\tau,u')-f^n(\tau,u)f^n(\tau,u')\right\|_{{L^{\infty}_{x}}}\notag\\
&\qquad\qquad+\de\sqrt{\bar{\mu}_{\de,\rho}}(u)\sqrt{\bar{\mu}_{\de,\rho}}(v')\left\|f^{n+1}(\tau,u)f^{n+1}(\tau,v')-f^n(\tau,u)f^n(\tau,v')\right\|_{{L^{\infty}_{x}}}\notag\\
&\qquad\qquad\left.+\de\sqrt{\bar{\mu}_{\de,\rho}}(u')\sqrt{\bar{\mu}_{\de,\rho}}(v')\left\|f^{n+1}(\tau,u')f^{n+1}(\tau,v')-f^n(\tau,u')f^n(\tau,v')\right\|_{{L^{\infty}_{x}}}\right]d\om dud\tau\notag\\
&=I'_1+I'_2+I'_3+I'_4+I'_5+I'_6.
\end{align}
Direct calculations show that $0\leq \de\mu_{\de,\rho}(v)\leq 1$ and $\sqrt{\bar{{\mu}}_{\de,\rho}}(u)\leq \rho^{-\frac{1}{2}}e^{-\frac{|u|^2}{4}}$ hold true for $0\leq \de\leq 1$. We then have
\begin{align}\label{Ipri1}
\dis I'_1&\leq C \rho^{-\frac{1}{2}}\int^t_0\int_{\R^3}\int_{\S^2}B(v-u,\theta)e^{-\frac{|u|^2}{4}}\sqrt{\bar{\mu}_{\de,\rho}}(u)\left\|f^{n+1}- f^{n}\right\|_{L^{\infty}_{T}L^{\infty}_{v,x}}d\om dud\tau\notag\\
&\leq  CT \rho^{-\frac{1}{2}}\left\|f^{n+1}- f^{n}\right\|_{L^{\infty}_{T}L^{\infty}_{v,x}}.
\end{align}
For $I'_2$ and $I'_3$, we can get $e^{-\frac{|u|^2}{4}}$ by the following inequality
\begin{align}\label{getdecay}
\max\left\{\sqrt{\bar{{\mu}}_{\de,\rho}}(v')\mu_{\de,\rho}(u), \sqrt{\bar{{\mu}}_{\de,\rho}}(v')\mu_{\de,\rho}(u'), \sqrt{\bar{{\mu}}_{\de,\rho}}(u')\mu_{\de,\rho}(u), \sqrt{\bar{{\mu}}_{\de,\rho}}(u')\mu_{\de,\rho}(v')\right\}\leq\rho^{-\frac{3}{2}}e^{-\frac{|u|^2}{4}}.
\end{align}
A similar argument as in \eqref{Ipri1} shows that
\begin{align}\label{Ipri23}
\dis I'_2+I'_3&\leq CT \rho^{-\frac{3}{2}}\left\|f^{n+1}- f^{n}\right\|_{L^{\infty}_{T}L^{\infty}_{v,x}}.
\end{align}
For $I'_4+I'_5+I'_6$, we first split the terms in the $L^\infty$ norm as follows:
\begin{align}\label{splitting}
&\left|f^{n+1}(u')f^{n+1}(v')-f^n(u')f^n(v')\right|\notag\\
&\qquad\qquad\qquad\leq\left|f^{n+1}(v') \right|\left|f^{n+1}(u')-f^n(u')\right|+\left|f^n(u') \right|\left|f^{n+1}(v')-f^n(v')\right|.
\end{align}
Then by a similar inequality as \eqref{getdecay} that
\begin{align*}
\max\{\sqrt{\bar{\mu}_{\de,\rho}}(u)\sqrt{\bar{\mu}_{\de,\rho}}(u'),\sqrt{\bar{\mu}_{\de,\rho}}(u)\sqrt{\bar{\mu}_{\de,\rho}}(v'),\sqrt{\bar{\mu}_{\de,\rho}}(u')\sqrt{\bar{\mu}_{\de,\rho}}(v')\}\leq \rho^{-1}e^{-\frac{|u|^2}{4}},
\end{align*}
we yield
\begin{align}\label{Ipri456}
\dis I'_4+I'_5+I'_6&\leq CT \rho^{-1}\left( \left\|f^{n+1}\right\|_{L^{\infty}_{T}L^{\infty}_{v,x}}+\left\| f^{n}\right\|_{L^{\infty}_{T}L^{\infty}_{v,x}} \right)\left\|f^{n+1}- f^{n}\right\|_{L^{\infty}_{T}L^{\infty}_{v,x}}.
\end{align}
Combining \eqref{1stdiff}, \eqref{Ipri1}, \eqref{Ipri23} and \eqref{Ipri456}, it follows that
\begin{align}\label{2nddiff}
\dis &\int^t_0\left| (g_1^{n+1}-g_1^n)(\tau,x-v(t-\tau),v) \right|d\tau\notag\\
&\leq  CT\left[\rho^{-\frac{1}{2}}+ \rho^{-\frac{3}{2}}+\rho^{-1}\left( \left\|f^{n+1}\right\|_{L^{\infty}_{T}L^{\infty}_{v,x}}+\left\| f^{n}\right\|_{L^{\infty}_{T}L^{\infty}_{v,x}} \right)\right]\left\|f^{n+1}- f^{n}\right\|_{L^{\infty}_{T}L^{\infty}_{v,x}}.
\end{align}
Then by \eqref{diff}, \eqref{2nddiff} and Lemma \ref{localbound}, we have
\begin{align*}
\dis I_{3}+I_{4}+I_{5}\leq   &CT\left(\rho^{-\frac{1}{2}}+ \rho^{-\frac{3}{2}}+\rho^{-1} \left\|f_0\right\|_{L^{\infty}_{T}L^{\infty}_{v,x}}\right)\left\|f^{n+1}- f^{n}\right\|_{L^{\infty}_{T}L^{\infty}_{v,x}}\notag\\
&\qquad\left( \left\|w_\be f_0 \right\|_{L^{\infty}_{v,x}}+\int_0^t \left|w_\beta(v)K_\de f^{n+1}(s,x-v(t-s),v)\right|ds\right.\notag\\
&\qquad\qquad\left.
+\int_0^t \left|w_\beta(v)\Ga_{\de +}(f^{n+1})(s,x-v(t-s),v)\right|ds \right).
\end{align*}
Recalling our assumption that $T\leq T'_1$, similar calculations as in \eqref{estI1}, \eqref{estI2} and \eqref{localbound} yield
\begin{align}\label{I345}
\dis I_{3}+&I_{4}+I_{5}\leq   CT\left(\rho^{-\frac{1}{2}}+ \rho^{-\frac{3}{2}}+\rho^{-1} \left\|f_0\right\|_{L^{\infty}_{T}L^{\infty}_{v,x}}\right)\left\|f^{n+1}- f^{n}\right\|_{L^{\infty}_{T}L^{\infty}_{v,x}}\notag\\
&
\left[ \left\|w_\be f_0 \right\|_{L^{\infty}_{v,x}}+CTC_{5\rho}\left(\left\|w_\be f^{n+1} \right\|_{{L^{\infty}_{T}L^{\infty}_{v,x}}}+\left\|w_\be f^{n+1} \right\|_{{L^{\infty}_{T}L^{\infty}_{v,x}}}^2+\left\|w_\be f^{n+1} \right\|_{{L^{\infty}_{T}L^{\infty}_{v,x}}}^3  \right)\right]\notag\\
\qquad&\leq CT\left( \rho^{-\frac{1}{2}}+ \rho^{-\frac{3}{2}}+\rho^{-1} \right)\left(\left\|w_\be f_0\right\|_{L^{\infty}_{v,x}}+\left\|w_\be f_0\right\|^2_{L^{\infty}_{v,x}}\right)\left\|f^{n+1}- f^{n}\right\|_{L^{\infty}_{T}L^{\infty}_{v,x}},
\end{align}
where $C_{5\rho}$ is defined in \eqref{C5rho}.

Recall from \eqref{diff} that
\begin{align*}I_6=\int_0^t \left|w_\beta(v)\left(K_\de f^{n+1}-K_\de f^{n}\right)(s,x-v(t-s),v)\right| ds.\end{align*}
Since $K_\de$ is linear, we have by Lemma \ref{KdeK0} that
\begin{align} \label{I6}
I_6\leq &C_{2\rho}\left\|w_\be f^{n+1}- w_\be f^{n}\right\|_{L^{\infty}_{T}L^{\infty}_{v,x}} \int_0^t \int_{\R^3}\left|k_w(v,\eta)\right|d\eta ds\notag\\
\leq& CT C_{2\rho}\left\|w_\be f^{n+1}- w_\be f^{n}\right\|_{L^{\infty}_{T}L^{\infty}_{v,x}},
\end{align}
where $C_{2\rho}$ and $k_w$ are defined in \eqref{DefC12rho} and \eqref{Defkw} respectively. 
\end{proof}

At last we turn to $I_7$ where
\begin{align}\label{defI7}I_7&=\int_0^t \left|w_\beta(v)\left(\Ga_{\de+}(f^{n+1})-\Ga_{\de+}(f^{n})\right)(s,x-v(t-s),v)\right| ds\notag\\
&\leq \frac{w_\beta(v)}{\sqrt{\bar{\mu}_{\de,\rho}(v)}}\int_0^t \int_{\R^3}\int_{\S^2}B(v-u,\theta)\notag\\
&\quad\left[\sqrt{\bar{\mu}_{\de,\rho}}(u')\sqrt{\bar{\mu}_{\de,\rho}}(v')\left\|f^{n+1}(s,u')f^{n+1}(s,v')-f^n(s,u')f^n(s,v')\right\|_{{L^{\infty}_{x}}}\left|1+\de\mu_{\de,\rho}(v)+\de\mu_{\de,\rho}(u)\right|\right.\notag\\
&\quad +\de\sqrt{\bar{\mu}_{\de,\rho}}(v')\sqrt{\bar{\mu}_{\de,\rho}}(u)\left\|f^{n+1}(s,v')f^{n+1}(s,u)-f^n(s,v')f^n(s,u)\right\|_{{L^{\infty}_{x}}}\left|\mu_{\de,\rho}(v)+\mu_{\de,\rho}(u')\right|\notag\\
&\quad+\de\sqrt{\bar{\mu}_{\de,\rho}}(u')\sqrt{\bar{\mu}_{\de,\rho}}(u)\left\|f^{n+1}(s,u')f^{n+1}(s,u)-f^n(s,u')f^n(s,u)\right\|_{{L^{\infty}_{x}}}\left|\mu_{\de,\rho}(v)+\mu_{\de,\rho}(v')\right|\notag\\
&\quad\left.+\de\sqrt{\bar{\mu}_{\de,\rho}}(u')\sqrt{\bar{\mu}_{\de,\rho}}(v')\sqrt{\bar{\mu}_{\de,\rho}}(u)\right.\notag\\
&\qquad \qquad\times\left.\left\|f^{n+1}(s,u')f^{n+1}(s,v')f^{n+1}(s,u)-f^n(s,u')f^n(s,v')f^{n}(s,u)\right\|_{{L^{\infty}_{x}}}\right]d\om duds\notag\\
&=I_{71}+I_{72}+I_{73}+I_{74}.
\end{align}
As usual, we first obtain the decay $e^{-\frac{|u|^2}{4}}$ from $\bar{\mu}_{\de,\rho}$. Then we use the similar split as \eqref{splitting} to get the estimate for $\left\|w_\be f^{n+1}- w_\be f^{n}\right\|_{L^{\infty}_{T}L^{\infty}_{v,x}}$. For
\begin{align*}
\dis I_{71}&=\frac{w_\beta(v)}{\sqrt{\bar{\mu}_{\de,\rho}(v)}}\int_0^t \int_{\R^3}\int_{\S^2}B(v-u,\theta)\sqrt{\bar{\mu}_{\de,\rho}}(u')\sqrt{\bar{\mu}_{\de,\rho}}(v')\notag\\
&\quad\left\|f^{n+1}(s,u')f^{n+1}(s,v')-f^n(s,u')f^n(s,v')\right\|_{{L^{\infty}_{x}}}\left|1+\de\mu_{\de,\rho}(v)+\de\mu_{\de,\rho}(u)\right|d\om duds,
\end{align*}
recalling we have $0\leq \de \mu_{\de,\rho}\leq 1$, \eqref{mumu/mu}, \eqref{splitting} and $w_\be(v)\leq Cw_\be(u')w_\be(v')$, it follows that
\begin{align}\label{I71}
\dis I_{71}&\leq CC_{3\rho}\int_0^t \int_{\R^3}\int_{\S^2}B(v-u,\theta)e^{-\frac{|u|^2}{4}}\notag\\
&\qquad\qquad\qquad\left( \left\|w_\be f^{n+1}\right\|_{L^{\infty}_{T}L^{\infty}_{v,x}}+\left\| w_\be f^{n}\right\|_{L^{\infty}_{T}L^{\infty}_{v,x}} \right)\left\|w_\be f^{n+1}- w_\be f^{n}\right\|_{L^{\infty}_{T}L^{\infty}_{v,x}}d\om duds\notag\\
&\leq CTC_{3\rho}\left( \left\|w_\be f^{n+1}\right\|_{L^{\infty}_{T}L^{\infty}_{v,x}}+\left\| w_\be f^{n}\right\|_{L^{\infty}_{T}L^{\infty}_{v,x}} \right)\left\|w_\be f^{n+1}- w_\be f^{n}\right\|_{L^{\infty}_{T}L^{\infty}_{v,x}},
\end{align}
where $C_{3\rho}$ is defined in \eqref{C3rho}. 

$I_{72}$ can be estimated similarly. The only difference is how we get $e^{-\frac{|u|^2}{4}}$. By direct calculations as \eqref{I321} and \eqref{I322}, we have the following inequalities:
\begin{align*}
\dis w_\be (v)\frac{\sqrt{\bar{{\mu}}_{\de,\rho}}(v')\sqrt{\bar{{\mu}}_{\de,\rho}}(u)}{\sqrt{\bar{{\mu}}_{\de,\rho}}(v)}\mu_{\de,\rho}(v)&=w_\be(v)\frac{\sqrt\rho e^{{|v'|^2}/{4}}}{\de+\rho e^{{|v'|^2}/{2}}}\frac{\sqrt\rho e^{{|u|^2}/{4}}}{\de+\rho e^{{|u|^2}/{2}}}\frac{\de+\rho e^{{|v|^2}/{2}}}{\sqrt\rho e^{{|v|^2}/{4}}}\frac{1}{\de+\rho e^\frac{|v|^2}{2}}\notag\\
&\leq C\,C_{4\rho}e^{-{|u|^2}/{4}},\\
w_\be (v)\frac{\sqrt{\bar{{\mu}}_{\de,\rho}}(v')\sqrt{\bar{{\mu}}_{\de,\rho}}(u)}{\sqrt{\bar{{\mu}}_{\de,\rho}}(v)}\mu_{\de,\rho}(u')&=w_\be(v)\frac{\sqrt\rho e^{{|v'|^2}/{4}}}{\de+\rho e^{{|v'|^2}/{2}}}\frac{\sqrt\rho e^{{|u|^2}/{4}}}{\de+\rho e^{{|u|^2}/{2}}}\frac{\de+\rho e^{{|v|^2}/{2}}}{\sqrt\rho e^{{|v|^2}/{4}}}\frac{1}{\de+\rho e^\frac{|u'|^2}{2}}\notag\\
&\leq C\,C_{4\rho}w_\be(v')w_\be(u')e^{-{|u'|^2}/{4}}e^{-{|u|^2}/{4}}\notag\\
&\leq C\,C_{4\rho}w_\be(v')e^{-{|u|^2}/{4}},
\end{align*}
where $C_{4\rho}$ is given in \eqref{C4rho}. Then similarly as \eqref{I321}, \eqref{I322} and \eqref{I32}, we have
\begin{align}\label{I72}
\dis I_{72}&\leq CC_{4\rho}\int_0^t \int_{\R^3}\int_{\S^2}B(v-u,\theta)e^{-\frac{|u|^2}{4}}\notag\\
&\qquad\qquad\qquad\left( \left\|w_\be f^{n+1}\right\|_{L^{\infty}_{T}L^{\infty}_{v,x}}+\left\| w_\be f^{n}\right\|_{L^{\infty}_{T}L^{\infty}_{v,x}} \right)\left\|w_\be f^{n+1}- w_\be f^{n}\right\|_{L^{\infty}_{T}L^{\infty}_{v,x}}d\om duds\notag\\
&\leq CTC_{4\rho}\left( \left\|w_\be f^{n+1}\right\|_{L^{\infty}_{T}L^{\infty}_{v,x}}+\left\| w_\be f^{n}\right\|_{L^{\infty}_{T}L^{\infty}_{v,x}} \right)\left\|w_\be f^{n+1}- w_\be f^{n}\right\|_{L^{\infty}_{T}L^{\infty}_{v,x}}.
\end{align}

Similar the arguments as in \eqref{I331} and \eqref{I72} show that
\begin{align}\label{I73}
\dis I_{73}\leq CTC_{4\rho}\left( \left\|w_\be f^{n+1}\right\|_{L^{\infty}_{T}L^{\infty}_{v,x}}+\left\| w_\be f^{n}\right\|_{L^{\infty}_{T}L^{\infty}_{v,x}} \right)\left\|w_\be f^{n+1}- w_\be f^{n}\right\|_{L^{\infty}_{T}L^{\infty}_{v,x}}.
\end{align}

Recall
\begin{align}\label{defI74}
I_{74}&=\frac{w_\beta(v)}{\sqrt{\bar{\mu}_{\de,\rho}(v)}}\int_0^t \int_{\R^3}\int_{\S^2}B(v-u,\theta)\de\sqrt{\bar{\mu}_{\de,\rho}}(u')\sqrt{\bar{\mu}_{\de,\rho}}(v')\sqrt{\bar{\mu}_{\de,\rho}}(u)\notag\\
&\qquad \qquad\times\left\|f^{n+1}(s,u')f^{n+1}(s,v')f^{n+1}(s,u)-f^n(s,u')f^n(s,v')f^{n}(s,u)\right\|_{{L^{\infty}_{x}}}d\om duds.
\end{align}
$I_{74}$ contains $|f^{n+1}(u')f^{n+1}(v')f^{n+1}(u)-f^{n}(u')f^{n}(v')f^{n}(u)|$. We divide it into three parts as
\begin{align*}
&\left|f^{n+1}(u')f^{n+1}(v')f^{n+1}(u)-f^n(u')f^n(v')f^{n}(u)\right|\notag\\
\leq&\left|f^{n+1}(v')f^{n+1}(u') \right|\left|f^{n+1}(u)-f^n(u)\right|\notag\\
&\qquad\qquad\qquad+\left|f^{n+1}(u')f^{n}(u) \right|\left|f^{n+1}(v')-f^n(v')\right|+\left|f^{n}(v')f^{n}(u) \right|\left|f^{n+1}(u')-f^n(u')\right|.
\end{align*}
Also it is straightforward to see
\begin{align*}
\dis w_\be(v)\frac{\sqrt{\bar{{\mu}}_{\de,\rho}}(u')\sqrt{\bar{{\mu}}_{\de,\rho}}(v')}{\sqrt{\bar{{\mu}}_{\de,\rho}}(v)}\sqrt{\bar{{\mu}}_{\de,\rho}}(u)&=w_\be(v)\frac{\sqrt\rho e^{{|u'|^2}/{4}}}{\de+\rho e^{{|u'|^2}/{2}}}\frac{\sqrt\rho e^{{|v'|^2}/{4}}}{\de+\rho e^{{|v'|^2}/{2}}}\frac{\de+\rho e^{{|v|^2}/{2}}}{\sqrt\rho e^{{|v|^2}/{4}}}\frac{\sqrt\rho e^{{|u|^2}/{4}}}{\de+\rho e^{{|u|^2}/{2}}}\notag\\
&\leq \frac{C_{3\rho}}{\sqrt\rho}w_\be(v)e^{-{|u'|^2}/{4}}e^{-{|v'|^2}/{4}}e^{{|v|^2}/{4}}e^{-{|u|^2}/{4}}\notag\\
&\leq C\,C_{2\rho}w_\be(u')w_\be(v')e^{-{|u|^2}/{4}}.
\end{align*}
Substituting the above two inequalities into \eqref{defI74}, it holds that
\begin{align}\label{I74}
\dis I_{74}&\leq CTC_{2\rho}\left\|w_\be f^{n+1}- w_\be f^{n}\right\|_{L^{\infty}_{T}L^{\infty}_{v,x}}\notag\\
&\qquad\qquad\times\left( \left\|w_\be f^{n+1}\right\|^2_{L^{\infty}_{T}L^{\infty}_{v,x}}+\left\| w_\be f^{n+1}\right\|_{L^{\infty}_{T}L^{\infty}_{v,x}}\left\| w_\be f^{n}\right\|_{L^{\infty}_{T}L^{\infty}_{v,x}}+\left\| w_\be f^{n}\right\|^2_{L^{\infty}_{T}L^{\infty}_{v,x}} \right).
\end{align}

Finally, \eqref{defI7}, \eqref{I71}, \eqref{I72}, \eqref{I73} and \eqref{I74} yield that
\begin{align}\label{I7}
\dis I_{7}&\leq CT\left(C_{3\rho}+C_{4\rho}\right)\left( \left\|w_\be f^{n+1}\right\|_{L^{\infty}_{T}L^{\infty}_{v,x}}+\left\| w_\be f^{n}\right\|_{L^{\infty}_{T}L^{\infty}_{v,x}} \right)\left\|w_\be f^{n+1}- w_\be f^{n}\right\|_{L^{\infty}_{T}L^{\infty}_{v,x}}\notag\\
&\quad+CTC_{2\rho}\left\|w_\be f^{n+1}- w_\be f^{n}\right\|_{L^{\infty}_{T}L^{\infty}_{v,x}}\notag\\
&\qquad\qquad\times\left( \left\|w_\be f^{n+1}\right\|^2_{L^{\infty}_{T}L^{\infty}_{v,x}}+\left\| w_\be f^{n+1}\right\|_{L^{\infty}_{T}L^{\infty}_{v,x}}\left\| w_\be f^{n}\right\|_{L^{\infty}_{T}L^{\infty}_{v,x}}+\left\| w_\be f^{n}\right\|^2_{L^{\infty}_{T}L^{\infty}_{v,x}} \right)\notag\\
&\leq CT\left(C_{2\rho}+C_{3\rho}+C_{4\rho}\right)\left( \left\|w_\be f_0\right\|_{L^{\infty}_{v,x}}+\left\| w_\be f_0\right\|_{L^{\infty}_{v,x}}^2 \right)\left\|w_\be f^{n+1}- w_\be f^{n}\right\|_{L^{\infty}_{T}L^{\infty}_{v,x}}.
\end{align}
Then substituting \eqref{I345}, \eqref{I6} and \eqref{I7} into \eqref{diff}, it follows that
\begin{align*}
\dis & \left| w_\beta(f^{n+2}-f^{n+1})(t,x,v) \right|\notag\\
&\leq CT\left( \rho^{-\frac{1}{2}}+ \rho^{-\frac{3}{2}}+\rho^{-1} \right)\left(\left\|w_\be f_0\right\|_{L^{\infty}_{v,x}}+\left\|w_\be f_0\right\|^2_{L^{\infty}_{v,x}}\right)\left\|f^{n+1}- f^{n}\right\|_{L^{\infty}_{T}L^{\infty}_{v,x}}\notag\\
&\quad+CT C_{2\rho}\left\|w_\be f^{n+1}- w_\be f^{n}\right\|_{L^{\infty}_{T}L^{\infty}_{v,x}}\notag\\
&\quad+CT\left(C_{2\rho}+C_{3\rho}+C_{4\rho}\right)\left( \left\|w_\be f_0\right\|_{L^{\infty}_{v,x}}+\left\| w_\be f_0\right\|_{L^{\infty}_{v,x}}^2 \right)\left\|w_\be f^{n+1}- w_\be f^{n}\right\|_{L^{\infty}_{T}L^{\infty}_{v,x}}\notag\\
&\leq CTC_{5\rho}\left(1+ \left\|w_\be f_0\right\|_{L^{\infty}_{v,x}}+\left\| w_\be f_0\right\|_{L^{\infty}_{v,x}}^2 \right)\left\|w_\be f^{n+1}- w_\be f^{n}\right\|_{L^{\infty}_{T}L^{\infty}_{v,x}},
\end{align*}
where $C_{5\rho}$ is defined in \eqref{C5rho}.
Then \eqref{CSLemma} follows if we choose
\begin{align*}
T''_1=\min\left\{T'_1,\frac{1}{2C\,C_{5\rho}\left(1+ \left\|w_\be f_0 \right\|_{L^{\infty}_{v,x}}+\left\|w_\be f_0 \right\|_{{L^{\infty}_{v,x}}}^2 \right)}\right\}.
\end{align*}

By Lemma \ref{cauchyseq}, for $(t,x,v)\in [0,T''_1]\times \Omega \times \R^3$, the sequence $\{w_\be f^n\}_{n=1}^\infty$ is a Cauchy sequence. Then we take the limit to obtain a solution $f=f(t,x,v)$ to the quantum Boltzmann equation \eqref{QBE} in the sense of \eqref{mildQBE} satisfying $0\leq F(t,x,v)=\mu_{\de,\rho}(v)+\sqrt{\bar{{\mu}}_{\de,\rho}} f(t,x,v)\leq 1/\de$ and 
\begin{align*}
\left\|w_\be f 
\right\|_{{L^{\infty}_{T_1}L^{\infty}_{v,x}}}\leq 2\left\|w_\be f_0 
\right\|_{L^{\infty}_{v,x}}.
\end{align*}
Similar arguments as how we prove Lemma \ref{cauchyseq} show that if we choose
\begin{align}\label{conT1}
T_1&\leq \min\left\{T'_1,T''_1,\frac{C'''}{C_{5\rho}\left(1+ \left\|w_\be f_0 \right\|_{L^{\infty}_{v,x}}+\left\|w_\be f_0 \right\|_{L^{\infty}_{v,x}}^2 \right)}\right\},
\end{align}
for some  $C'''=C'''(\be,\ga)$ which is independent of $\rho$, the solution we obtained is unique for $(t,x,v)\in [0,T_1]\times \Omega \times \R^3$. Hence we define
\begin{align*}
T_1:=\frac{C_1}{C_{*\rho}\left(1+\|w_\be f_0\|_{L^{\infty}_{v,x}}+\|w_\be f_0\|^2_{L^{\infty}_{v,x}}\right)},
\end{align*}
where $C_1=\frac{1}{2}\min\{C',C'',C''' \}$, $C_{*\rho}=C_{5\rho}$. We see $T_1$ satisfies \eqref{conT1} and \eqref{LE} holds. Thus we finish the proof of Theorem \ref{local}.\qed

\section{Global Existence}
In this section we prove Theorem \ref{global}. First from the mild form \eqref{mildQBE}, we can rewrite the equation for $w_\be f$, which is 
\begin{align}\label{mildQBEw}
\dis w_\be(v) f(t,x,v)=&e^{-\nu_\de(v)t} w_\be(v) f_0(x-vt,v)+\int_0^t e^{-\nu_\de(v)(t-s)} w_\be(v)(K^m_{\de} f)(s,x-v(t-s),v)ds \notag\\
&+\int_0^t e^{-\nu_\de(v)(t-s)} w_\be(v)(K^c_{\de} f)(s,x-v(t-s),v)ds \notag\\
&+\int_0^t e^{-\nu_\de(v)(t-s)} w_\be(v)\Ga_\de(f)(s,x-v(t-s),v)ds,
\end{align}
where $K_\de^m$ and $K_\de^c$ are defined in \eqref{DefKmde} and \eqref{DefKcde} respectively. Throughout this section we assume that $(t,x,v)\in [0,T]\times \Omega \times \R^3$ for positive time $T$. We directly obtain from \eqref{mildQBEw} and Lemma \ref{KdeK0} that 
\begin{align}\label{1stest}
\dis \left|w_\be(v) f(t,x,v)\right|\leq&\left\|w_\be f_0 \right\|_{L^{\infty}_{v,x}}+Cm^{3+\ga} C_{2\rho}\int_0^t e^{-\nu_\de(v)(t-s)}e^{-\frac{|v|^2}{20}}w_\be(v)\|f(s)\|_{L^{\infty}_{v,x}} ds\notag \\
&+C_{2\rho}\int_0^t e^{-\nu_\de(v)(t-s)}\int_{\R^3}\left(\left|l_1(v,\eta)\right|+\left|l_2(v,\eta)\right|\right)w_\beta(v)\left|f(s,x-v(t-s),\eta)\right|d\eta ds\notag\\
&+\int_0^t e^{-\nu_\de(v)(t-s)} w_\be(v)\Ga_\de(f)(s,x-v(t-s),v)ds\notag\\
\leq&\left\|w_\be f_0 \right\|_{L^{\infty}_{v,x}}+Cm^{3+\ga} C_{2\rho}e^{-\frac{|v|^2}{20}}\frac{w_\be(v)}{\nu_\de(v)}\|w_\be f\|_{L^\infty_T L^{\infty}_{v,x}}\notag \\
&+C_{2\rho}\int_0^t e^{-\nu_\de(v)(t-s)}\int_{\R^3}l_w(v,\eta)\left|w_\be(\eta) f(s,x-v(t-s),\eta)\right|d\eta ds\notag\\
&+\int_0^t e^{-\nu_\de(v)(t-s)} w_\be(v)\Ga_\de(f)(s,x-v(t-s),v)ds,
\end{align}
where $l_w$ is defined in \eqref{Deflw}.
Notice that by Lemma \ref{KdeK0} we have $\nu_\de(v)\geq\frac{1}{C} C_{1\rho}(1+|v|)^\ga$. Substituting it into \eqref{1stest} we have 
\begin{align}\label{1stest}
\dis \left|w_\be(v) f(t,x,v)\right|\leq&\left\|w_\be f_0 \right\|_{L^{\infty}_{v,x}}+Cm^{3+\ga} \frac{C_{2\rho}}{C_{1\rho}}e^{-\frac{|v|^2}{20}}\frac{w_\be(v)}{(1+|v|)^\ga}\|w_\be f\|_{L^\infty_T L^{\infty}_{v,x}}\notag \\
&+C_{2\rho}\int_0^t e^{-\nu_\de(v)(t-s)}\int_{\R^3}l_w(v,\eta)\left|w_\be(\eta) f(s,x-v(t-s),\eta)\right|d\eta ds\notag\\
&+\int_0^t e^{-\nu_\de(v)(t-s)} w_\be(v)\left|\Ga_\de(f)(s,x-v(t-s),v)\right|ds\notag\\
\leq&\left\|w_\be f_0 \right\|_{L^{\infty}_{v,x}}+Cm^{3+\ga} \frac{C_{2\rho}}{C_{1\rho}}e^{-\frac{|v|^2}{30}}\|w_\be f\|_{L^\infty_T L^{\infty}_{v,x}}\notag \\
&+C_{2\rho}\int_0^t e^{-\nu_\de(v)(t-s)}\int_{\R^3}l_w(v,\eta)\left|w_\be(\eta) f(s,x-v(t-s),\eta)\right|d\eta ds\notag\\
&+\int_0^t e^{-\nu_\de(v)(t-s)} w_\be(v)\left|\Ga_\de(f)(s,x-v(t-s),v)\right|ds\notag\\
=&\left\|w_\be f_0 \right\|_{L^{\infty}_{v,x}}+Cm^{3+\ga} \frac{C_{2\rho}}{C_{1\rho}}e^{-\frac{|v|^2}{30}}\|w_\be f\|_{L^\infty_T L^{\infty}_{v,x}}+J_1+J_2.
\end{align}
Here we can see that in order to obtain the pointwise bound, we need to have good estimate on $w_\be(v)\left|\Ga_\de(f)(s,x-v(t-s),v)\right|$. Thus we establish the following Lemma in order to take care of this term.
\begin{lemma}\label{Gaest}
For any $p>\frac{3}{3+\ga}$, let $-3<\ga<0$, $\be > \max\{6,16/(5p-1)\}$, it holds that
\begin{align*}
\dis
\left|w_\be(v)\Ga_\de (f)(v)\right|&\leq C\,C_{5\rho}\nu(v)\left(1+\|w_\be f\|_{L^{\infty}_{v}}\right)\notag\\
&\quad\times\left\{\|w_\be f\|^{\frac{2p-1}{p}}_{L^{\infty}_{v}}\left(\int_{\R^3}\left|f(u)\right|du\right)^{\frac{1}{p}}+\|w_\be f\|^{\frac{10p-1}{5p}}_{L^{\infty}_{v}}\left(\int_{\R^3}\left|f(u)\right|du\right)^{\frac{1}{5p}}\right\},
\end{align*}
where $C_{5\rho}$ is defined in \eqref{C5rho}.
\end{lemma}
\begin{proof}
From the definition of $\Ga_\de$ \eqref{DefGade} and the fact that $0\leq\max\left\{\de,\de\mu_{\de,\rho}(v)\right\}\leq 1$ we have
\begin{align*}
\dis
&\left|w_\be(v)\Ga_\de (f)(v)\right|\leq\frac{C}{\sqrt{\bar{\mu}_{\de,\rho}(v)}}\int_{\R^3}\int_{\S^2}B(v-u,\theta)\notag\\
&\quad\left[w_\be(v)\left|\sqrt{\bar{\mu}_{\de,\rho}}f(u')\sqrt{\bar{\mu}_{\de,\rho}}f(v')\right|+w_\be(v)\left|\sqrt{\bar{\mu}_{\de,\rho}}f(u)\sqrt{\bar{\mu}_{\de,\rho}}f(v)\right|\right. \notag\\
&\quad+ w_\be(v)\left|\sqrt{\bar{\mu}_{\de,\rho}}f(v')\sqrt{\bar{\mu}_{\de,\rho}}f(u)\right|\left|\mu_{\de,\rho}(v)+\mu_{\de,\rho}(u')\right|\notag\\
&\quad+w_\be(v)\left|\sqrt{\bar{\mu}_{\de,\rho}}f(u')\sqrt{\bar{\mu}_{\de,\rho}}f(u)\right|\left|\mu_{\de,\rho}(v)+\mu_{\de,\rho}(v')\right|\notag\\
&\quad+ w_\be(v)\left|\sqrt{\bar{\mu}_{\de,\rho}}f(v')\sqrt{\bar{\mu}_{\de,\rho}}f(v)\right|\left|\mu_{\de,\rho}(u)+\mu_{\de,\rho}(u')\right|\notag\\
&\quad+ w_\be(v)\left|\sqrt{\bar{\mu}_{\de,\rho}}f(u')\sqrt{\bar{\mu}_{\de,\rho}}f(v)\right|\left|\mu_{\de,\rho}(u)+\mu_{\de,\rho}(v')\right|\notag\\
&\quad+ w_\be(v)\left|\sqrt{\bar{\mu}_{\de,\rho}}f(u)\sqrt{\bar{\mu}_{\de,\rho}}f(v)\sqrt{\bar{\mu}_{\de,\rho}}f(u')\right|+w_\be(v)\left|\sqrt{\bar{\mu}_{\de,\rho}}f(u)\sqrt{\bar{\mu}_{\de,\rho}}f(v)\sqrt{\bar{\mu}_{\de,\rho}}f(v')\right|\notag\\
&\quad\left.+w_\be(v)\left|\sqrt{\bar{\mu}_{\de,\rho}}f(u')\sqrt{\bar{\mu}_{\de,\rho}}f(v')\sqrt{\bar{\mu}_{\de,\rho}}f(u)\right|+w_\be(v)\left|\sqrt{\bar{\mu}_{\de,\rho}}f(u')\sqrt{\bar{\mu}_{\de,\rho}}f(v')\sqrt{\bar{\mu}_{\de,\rho}}f(v)\right|\right]d\om du.
\end{align*} 
There are ten terms above. In order to simplify the calculations, we can replace the terms that include $\left|\sqrt{\bar{\mu}_{\de,\rho}}f(u')\sqrt{\bar{\mu}_{\de,\rho}}f(u)\right|\left|\mu_{\de,\rho}(v)+\mu_{\de,\rho}(v')\right|$ and $\left|\sqrt{\bar{\mu}_{\de,\rho}}f(u')\sqrt{\bar{\mu}_{\de,\rho}}f(v)\right|\left|\mu_{\de,\rho}(u)+\mu_{\de,\rho}(v')\right|$ in the third row and fourth row by 
\begin{center}
$\left|\sqrt{\bar{\mu}_{\de,\rho}}f(v')\sqrt{\bar{\mu}_{\de,\rho}}f(u)\right|\left|\mu_{\de,\rho}(v)+\mu_{\de,\rho}(u')\right|$ and $\left|\sqrt{\bar{\mu}_{\de,\rho}}f(v')\sqrt{\bar{\mu}_{\de,\rho}}f(v)\right|\left|\mu_{\de,\rho}(u)+\mu_{\de,\rho}(u')\right|$
\end{center} respectively since we can exchange $u'$ and $v'$ by a rotation. Then it follows that
\begin{align}\label{Ga1}
\dis
&\left|w_\be(v)\Ga_\de (f)(v)\right|\leq J_{21}+J_{22},
\end{align} 
where
\begin{align*}
\dis
& J_{21}=\frac{C}{\sqrt{\bar{\mu}_{\de,\rho}(v)}}\int_{\R^3}\int_{\S^2}B(v-u,\theta)\left[w_\be(v)\left|\sqrt{\bar{\mu}_{\de,\rho}}f(u)\sqrt{\bar{\mu}_{\de,\rho}}f(v)\right|\right.\notag\\
&\quad + w_\be(v)\left|\sqrt{\bar{\mu}_{\de,\rho}}f(v')\sqrt{\bar{\mu}_{\de,\rho}}f(u)\right|\left|\mu_{\de,\rho}(v)+\mu_{\de,\rho}(u')\right|+ w_\be(v)\left|\sqrt{\bar{\mu}_{\de,\rho}}f(u)\sqrt{\bar{\mu}_{\de,\rho}}f(v)\sqrt{\bar{\mu}_{\de,\rho}}f(u')\right|\notag\\
&\quad\left.+w_\be(v)\left|\sqrt{\bar{\mu}_{\de,\rho}}f(u)\sqrt{\bar{\mu}_{\de,\rho}}f(v)\sqrt{\bar{\mu}_{\de,\rho}}f(v')\right|+w_\be(v)\left|\sqrt{\bar{\mu}_{\de,\rho}}f(u')\sqrt{\bar{\mu}_{\de,\rho}}f(v')\sqrt{\bar{\mu}_{\de,\rho}}f(u)\right|\right]d\om du,
\end{align*} 
and
\begin{align*}
\dis
 J_{22}=&\frac{C}{\sqrt{\bar{\mu}_{\de,\rho}(v)}}\int_{\R^3}\int_{\S^2}B(v-u,\theta)\left[w_\be(v)\left|\sqrt{\bar{\mu}_{\de,\rho}}f(u')\sqrt{\bar{\mu}_{\de,\rho}}f(v')\right|\right.\notag\\
&\qquad\qquad \qquad + w_\be(v)\left|\sqrt{\bar{\mu}_{\de,\rho}}f(v')\sqrt{\bar{\mu}_{\de,\rho}}f(v)\right|\left|\mu_{\de,\rho}(u)+\mu_{\de,\rho}(u')\right|\notag\\
&\qquad \qquad \qquad \left.+w_\be(v)\left|\sqrt{\bar{\mu}_{\de,\rho}}f(u')\sqrt{\bar{\mu}_{\de,\rho}}f(v')\sqrt{\bar{\mu}_{\de,\rho}}f(v)\right|\right]d\om du.
\end{align*} 
It follows by similar arguments as in \eqref{I71}, \eqref{I72}, \eqref{I73} and \eqref{I74} that
\begin{align}\label{1stJ21}
\dis
& J_{21}\leq C\left(C_{3\rho} +C_{4\rho}+C_{2\rho}\|w_\be f\|_{L^{\infty}_{v}}\right)\|w_\be f\|_{L^{\infty}_{v}}\int_{\R^3}\int_{\S^2}B(v-u,\theta)e^{-\frac{|u|^2}{4}}\left|f(u)\right|d\om du.
\end{align} 
We choose a constant $p$ such that $p>\frac{3}{3+\ga}$, then by H\"older's inequality it holds that
\begin{align}\label{2ndJ21}
&\int_{\R^3}\int_{\S^2}B(v-u,\theta)e^{-\frac{|u|^2}{4}}\left|f(u)\right|d\om du\notag\\
&\leq C\left(\int_{\R^3}|v-u|^{\frac{\ga p}{p-1}} e^{-\frac{|u|^2}{4}} du\right)^{\frac{p-1}{p}}\left(\int_{\R^3}\int_{\S^2} \left|f(u)\right|^pd\om du\right)^{\frac{1}{p}}\notag\\
&\leq C\nu(v)\|w_\be f\|_{L^{\infty}_{v}}^{\frac{p-1}{p}}\left(\int_{\R^3}\left|f(u)\right| du\right)^{\frac{1}{p}}.
\end{align}
In the last inequality above we have used the fact that
$$
\left(\int_{\R^3}|v-u|^{\frac{\ga p}{p-1}} e^{-\frac{|u|^2}{4}} du\right)^{\frac{p-1}{p}}\leq (1+|v|)^\ga \leq C\nu(v)
$$
since $-3<\frac{\ga p}{p-1}<0$ by our choice of $p$. Then it follows from \eqref{1stJ21} and \eqref{2ndJ21} that
\begin{align}\label{J21}
\dis
& J_{21}\leq C\nu(v)\left(C_{3\rho} +C_{4\rho}+C_{2\rho}\|w_\be f\|_{L^{\infty}_{v}}\right)\|w_\be f\|^{\frac{2p-1}{p}}_{L^{\infty}_{v}}\left(\int_{\R^3}\left|f(u)\right|du\right)^{\frac{1}{p}}.
\end{align} 
Notice that since $|v|^2\leq |u'|^2+|v'|^2$, either $|v|^2\leq 2|u'|^2$ or $|v|^2\leq 2|v'|^2$. Then there exists a positive constant $C$ such that $w_\be(v)\leq w_\be(v)\chi_{\{ |v|^2\leq 2|u'|^2 \}}+w_\be(v)\chi_{\{ |v|^2\leq 2|v'|^2 \}}\leq C\left(w_\be(u')+w_\be(v')\right)$. Also we can exchange $u'$ and $v'$ by a rotation. Then similarly as above we have
\begin{align}\label{1stJ22}
\dis
 J_{22}\leq&\frac{C}{\sqrt{\bar{\mu}_{\de,\rho}(v)}}\int_{\R^3}\int_{\S^2}B(v-u,\theta)\left[\left(w_\be(u')+w_\be(v')\right)\left|\sqrt{\bar{\mu}_{\de,\rho}}f(u')\sqrt{\bar{\mu}_{\de,\rho}}f(v')\right|\right.\notag\\
&\qquad\qquad \qquad + w_\be(v)\left|\sqrt{\bar{\mu}_{\de,\rho}}f(v')\sqrt{\bar{\mu}_{\de,\rho}}f(v)\right|\left|\mu_{\de,\rho}(u)+\mu_{\de,\rho}(u')\right|\notag\\
&\qquad \qquad \qquad \left.+w_\be(v)\left|\sqrt{\bar{\mu}_{\de,\rho}}f(u')\sqrt{\bar{\mu}_{\de,\rho}}f(v')\sqrt{\bar{\mu}_{\de,\rho}}f(v)\right|\right]d\om du\notag\\
&\leq C\left(C_{3\rho} +C_{4\rho}+C_{2\rho}\|w_\be f\|_{L^{\infty}_{v}}\right)\int_{\R^3}\int_{\S^2}B(v-u,\theta)e^{-\frac{|u|^2}{4}}\left|f(v')\right|d\om du.
\end{align} 
It holds from H\"older's inequality that
\begin{align}\label{2ndJ22}
&\int_{\R^3}\int_{\S^2}B(v-u,\theta)e^{-\frac{|u|^2}{4}}\left|f(v')\right|d\om du\notag\\
&\leq C\left(\int_{\R^3}|v-u|^{\frac{\ga p}{p-1}} e^{-\frac{|u|^2}{4}} du\right)^{\frac{p-1}{p}}\left(\int_{\R^3}\int_{\S^2} e^{-\frac{|u|^2}{4}}\left|f(v')\right|^pd\om du\right)^{\frac{1}{p}}\notag\\
&\leq C\nu(v)\left(\int_{\R^3}\int_{\S^2} e^{-\frac{|u|^2}{4}}\left|f(v')\right|^pd\om du\right)^{\frac{1}{p}}.
\end{align}
Denote $z=u-v$, $z_{\shortparallel}=(u-v)\cdot \omega$, $z_{\perp}=z-z_{\para}$, $\eta=v+z_\para$. We change variables to yield
\begin{align}\label{3rdJ22}
&\left(\int_{\R^3}\int_{\S^2} e^{-\frac{|u|^2}{4}}\left|f(v')\right|^pd\om du\right)^{\frac{1}{p}}\notag\\
&=\left( \int_{\R^3}\int_{z_\perp}e^{-\frac{|z_\perp+\eta|^2}{4}}dz_\perp|f(\eta)|^p\frac{1}{|\eta-v|^{2}} d\eta\right)^{\frac{1}{p}}\notag\\
&\leq C \left( \int_{\R^3}\frac{(1+|\eta|)^{-4}}{|\eta-v|^{\frac{5}{2}}} d\eta\right)^{\frac{4}{5p}}\left( \int_{\R^3}(1+|\eta|)^{16}|f(\eta)|^{5p} d\eta\right)^{\frac{1}{5p}}\notag\\
&\leq C\|w_{\frac{16}{5p-1}} f\|_{L^{\infty}_{v}}^{\frac{5p-1}{5p}}\left(\int_{\R^3}\left|f(u)\right| du\right)^{\frac{1}{5p}}.
\end{align}
By choosing $\be>\frac{16}{5p-1}$, it follows from \eqref{1stJ22}, \eqref{2ndJ22} and \eqref{3rdJ22} that
\begin{align}\label{J22}
\dis
& J_{22}\leq C\nu(v)\left(C_{3\rho} +C_{4\rho}+C_{2\rho}\|w_\be f\|_{L^{\infty}_{v}}\right)\|w_\be f\|^{\frac{10p-1}{5p}}_{L^{\infty}_{v}}\left(\int_{\R^3}\left|f(u)\right|du\right)^{\frac{1}{5p}}.
\end{align} 
Our Lemma follows from \eqref{J21}, \eqref{J22} and the definition of $C_{5\rho}$ that $C_{5\rho}=C_{2\rho} +C_{3\rho}+C_{4\rho}$.
\end{proof}
Using the pointwise inequality \eqref{1stest}, we have the following lemma.

\begin{lemma}\label{globalest}
Let $-3<\ga<0$, $\be > \max\{3,3-\ga,16/(5p-1)\}$, and the constant $p$ be given in Lemma \ref{Gaest}. It holds that
\begin{align}\label{lemmaglobal}
&\dis \|w_\be f\|_{L^\infty_T L^{\infty}_{v,x}}\leq C_3 C_{6\rho}\left(\left\|w_\be f_0 \right\|_{L^{\infty}_{v,x}}+\left\|w_\be f_0 \right\|_{L^{\infty}_{v,x}}^2+\left\|w_\be f_0 \right\|_{L^{\infty}_{v,x}}^3\right)\notag\\
&\quad+ C_3 C_{6\rho}\left(1+\|w_\be f\|_{L^\infty_TL^{\infty}_{v,x}}\right)\left\{\|w_\be f\|^{\frac{2p-1}{p}}_{L^\infty_TL^{\infty}_{v,x}}\| f\|_{L^\infty_{T_1,T}L^{\infty}_{x}L^{1}_{v}}^{\frac{1}{p}}+\|w_\be f\|^{\frac{10p-1}{5p}}_{{L^\infty_TL^{\infty}_{v,x}}{L^{\infty}_{v,x}}}\| f\|_{{L^\infty_{T_1,T}}L^{\infty}_{x}L^{1}_{v}}^{\frac{1}{5p}}\right\}\notag\\
&\quad+C_3 C_{\rho,N_\rho}\left( \la_\rho^{-\frac{3}{2}}\sqrt{\CE_{\de,\rho}(F_0)}+ \la_\rho^{-3}\CE_{\de,\rho}(F_0)\right),
\end{align}
where the time $T_1$ is given in \eqref{T_1}, the constants $C_{6\rho}$, $C_{\rho,N}$, $N_\rho$ and $\la_\rho$ are defined in \eqref{C6rho}, \eqref{CrhoN}, \eqref{defNrho} and \eqref{deflarho} respectively, and $C_3=C_3(\be,\ga)$.
\end{lemma}
\begin{proof}
Now by Lemma \ref{Gaest} and \eqref{1stest}, it is straightforward to see that
\begin{align}\label{J2}
\dis J_2&=\int_0^t e^{-\nu_\de(v)(t-s)} w_\be(v)\left|\Ga_\de(f)(s,x-v(t-s),v)\right|ds\notag\\
&\leq  C\left(\int_0^t e^{-\nu_\de(v)(t-s)} \nu(v)ds\right)C_{5\rho}\sup_{0\leq t\leq T}\left(1+\|w_\be f(t)\|_{{L^{\infty}_{v,x}}}\right)\notag\\
&\qquad\times\left\{\|w_\be f(t)\|^{\frac{2p-1}{p}}_{{L^{\infty}_{v,x}}}\| f(t)\|_{L^{\infty}_{x}L^{1}_{v}}^{\frac{1}{p}}+\|w_\be f(t)\|^{\frac{10p-1}{5p}}_{{L^{\infty}_{v,x}}}\| f(t)\|_{L^{\infty}_{x}L^{1}_{v}}^{\frac{1}{5p}}\right\}\notag\\
&\leq  C\,\frac{C_{5\rho}}{C_{1\rho}}\sup_{0\leq t\leq T}\left(1+\|w_\be f(t)\|_{{L^{\infty}_{v,x}}}\right)\notag\\
&\qquad\times\left\{\|w_\be f(t)\|^{\frac{2p-1}{p}}_{{L^{\infty}_{v,x}}}\| f(t)\|_{L^{\infty}_{x}L^{1}_{v}}^{\frac{1}{p}}+\|w_\be f(t)\|^{\frac{10p-1}{5p}}_{{L^{\infty}_{v,x}}}\| f(t)\|_{L^{\infty}_{x}L^{1}_{v}}^{\frac{1}{5p}}\right\},
\end{align}
where in the last inequality above, we have used the fact that
\begin{align*}
\int_0^t e^{-\nu_\de(v)(t-s)} \nu(v)ds\leq \frac{C}{C_{1\rho}\nu(v)}\nu(v)\leq\frac{C}{C_{1\rho}}.
\end{align*}

Next we turn to $J_1$. Recall
\begin{align*}
\dis J_1=C_{2\rho}\int_0^t e^{-\nu_\de(v)(t-s)}\int_{\R^3}l_w(v,\eta)\left|w_\be(\eta) f(s,x-v(t-s),\eta)\right|d\eta ds.
\end{align*}
We substitute \eqref{1stest} into $\left|w_\be(\eta) f(s,x-v(t-s),\eta)\right|$ in $J_1$ to yield
\begin{align}\label{1stJ1}
\dis J_1&\leq C_{2\rho}\int_0^t e^{-\nu_\de(v)(t-s)}\int_{\R^3}l_w(v,\eta)\left\|w_\be f_0 \right\|_{L^{\infty}_{v,x}}d\eta ds\notag\\
&\quad +Cm^{3+\ga} \frac{C^2_{2\rho}}{C_{1\rho}}\int_0^t e^{-\nu_\de(v)(t-s)}\int_{\R^3}\left|l_w(v,\eta)\right|e^{-\frac{|\eta|^2}{30}}\|w_\be f\|_{L^\infty_T L^{\infty}_{v,x}}d\eta ds\notag\\
&\quad +C^2_{2\rho}\int_0^t e^{-\nu_\de(v)(t-s)}\int_{\R^3}l_w(v,\eta)\notag\\
&\qquad\qquad\qquad\qquad\times\int_0^s e^{-\nu_\de(\eta)(s-s_1)}\int_{\R^3}\left|l_w(\eta,\xi)\right|\left|w_\be(\xi) f(s_1,x_1-\eta(s-s_1),\xi)\right|d\xi ds_1d\eta ds\notag\\
&\quad +C_{2\rho}\int_0^t e^{-\nu_\de(v)(t-s)}\int_{\R^3}l_w(v,\eta)\notag\\
&\qquad\qquad\qquad\qquad\times\int_0^s e^{-\nu_\de(\eta)(s-s_1)} w_\be(\eta)\left|\Ga_\de(f)(s_1,x_1-\eta(s-s_1),\eta)\right|ds_1d\eta ds\notag\\
&=J_{10}+J_{11}+J_{12}+J_{13},
\end{align}
where $x_1=x-v(t-s)$.
From \eqref{Prol1} in Lemma \ref{KmKc} we have
\begin{align}\label{J10}
\dis J_{10}&\leq C\,C_{2\rho} m^{\ga-1}\left\|w_\be f_0 \right\|_{L^{\infty}_{v,x}}\int_0^t e^{-\nu_\de(v)(t-s)} \frac{\nu(v)}{(1+|v|)^2}ds\notag\\
&\leq C\frac{C_{2\rho}}{C_{1\rho}} m^{\ga-1}\left\|w_\be f_0 \right\|_{L^{\infty}_{v,x}}\frac{1}{\nu(v)}\frac{\nu(v)}{(1+|v|)^2}\notag\\
&\leq C\frac{C_{2\rho}}{C_{1\rho}} m^{\ga-1}\left\|w_\be f_0 \right\|_{L^{\infty}_{v,x}}.
\end{align}
Since $e^{-\frac{|\eta|^2}{30}}\leq C\nu(\eta)$, it follows that
\begin{align}\label{J11}
\dis J_{11}&\leq Cm^{3+\ga} \frac{C^2_{2\rho}}{C_{1\rho}}\int_0^t e^{-\nu_\de(v)(t-s)}\int_{\R^3}l_w(v,\eta)\nu(\eta)\|w_\be f\|_{L^\infty_T L^{\infty}_{v,x}}d\eta ds\notag\\
&\leq Cm^{3+\ga} \frac{C^2_{2\rho}}{C^2_{1\rho}}\|w_\be f\|_{L^\infty_T L^{\infty}_{v,x}}.
\end{align}
Then we estimate $J_{13}$. Applying Lemma \ref{Gaest} again, one can get that
\begin{align}\label{J13}
\dis J_{13}&\leq C\,C_{2\rho}C_{5\rho}\int_0^t e^{-\nu_\de(v)(t-s)}\int_{\R^3}l_w(v,\eta)\int_0^s e^{-\nu_\de(\eta)(s-s_1)} \nu(\eta)ds_1d\eta ds \notag\\
&\qquad\times\sup_{0\leq t\leq T}\left(1+\|w_\be f(t)\|_{{L^{\infty}_{v,x}}}\right)\left\{\|w_\be f(t)\|^{\frac{2p-1}{p}}_{{L^{\infty}_{v,x}}}\| f(t)\|_{L^{\infty}_{x}L^{1}_{v}}^{\frac{1}{p}}+\|w_\be f(t)\|^{\frac{10p-1}{5p}}_{{L^{\infty}_{v,x}}}\| f(t)\|_{L^{\infty}_{x}L^{1}_{v}}^{\frac{1}{5p}}\right\}\notag\\
&\leq C\frac{C_{2\rho}C_{5\rho}}{C^2_{1\rho}}m^{\ga-1}\sup_{0\leq t\leq T}\left(1+\|w_\be f(t)\|_{{L^{\infty}_{v,x}}}\right)\notag\\
&\qquad\qquad\qquad\qquad\times\left\{\|w_\be f(t)\|^{\frac{2p-1}{p}}_{{L^{\infty}_{v,x}}}\| f(t)\|_{L^{\infty}_{x}L^{1}_{v}}^{\frac{1}{p}}+\|w_\be f(t)\|^{\frac{10p-1}{5p}}_{{L^{\infty}_{v,x}}}\| f(t)\|_{L^{\infty}_{x}L^{1}_{v}}^{\frac{1}{5p}}\right\}.
\end{align}
The estimate of $J_{12}$ is more delicate. We divide it into four cases. First in order to simplify the integral, recalling $x_1=x-v(t-s)$, we rewrite $J_{12}$ by Fubini's theorem that
\begin{align*}
\dis J_{12}&= C^2_{2\rho}\int_0^t e^{-\nu_\de(v)(t-s)}\int_{\R^3}\left|l_w(v,\eta)\right|\notag\\
&\qquad\qquad\qquad\qquad\times\int_0^s e^{-\nu_\de(\eta)(s-s_1)}\int_{\R^3}l_w(\eta,\xi)\left|w_\be(\xi) f(s_1,x_1-\eta(s-s_1),\xi)\right|d\xi ds_1d\eta ds\notag\\
&= C^2_{2\rho}\int_0^t \int_0^s\int_{\R^3}\int_{\R^3} e^{-\nu_\de(v)(t-s)} e^{-\nu_\de(\eta)(s-s_1)}\notag\\
&\qquad\qquad\qquad\qquad\qquad\times l_w(v,\eta)l_w(\eta,\xi)\left|w_\be(\xi) f(s_1,x_1-\eta(s-s_1),\xi)\right|d\xi d\eta ds_1 ds.
\end{align*}
Then we divide the above integral into four parts as follows.
\begin{align}\label{sepJ12}
\dis J_{12}&\leq C^2_{2\rho}\int_0^t \int_0^s\int_{\R^3}\int_{\R^3} e^{-\nu_\de(v)(t-s)} e^{-\nu_\de(\eta)(s-s_1)}\chi_{\{|v|\geq N \}}\notag\\
&\qquad\qquad\qquad\qquad\qquad\times l_w(v,\eta)l_w(\eta,\xi)\left|w_\be(\xi) f(s_1,x_1-\eta(s-s_1),\xi)\right|d\xi d\eta ds_1 ds\notag\\
&\qquad+C^2_{2\rho}\int_0^t \int_0^s\int_{\R^3}\int_{\R^3} e^{-\nu_\de(v)(t-s)} e^{-\nu_\de(\eta)(s-s_1)}\chi_{\{|v|\leq N,|\eta|\geq 2N \}}\chi_{\{|\eta|\leq 2N,|\xi|\geq 3N \}}\notag\\
&\qquad\qquad\qquad\qquad\qquad\times l_w(v,\eta)l_w(\eta,\xi) \left|w_\be(\xi) f(s_1,x_1-\eta(s-s_1),\xi)\right|d\xi d\eta ds_1 ds \notag\\
&\qquad+ C^2_{2\rho}\int_0^t \int_{s-\la}^s\int_{\R^3}\int_{\R^3} e^{-\nu_\de(v)(t-s)} e^{-\nu_\de(\eta)(s-s_1)}\notag\\
&\qquad\qquad\qquad\qquad\qquad\times l_w(v,\eta)l_w(\eta,\xi)\left|w_\be(\xi) f(s_1,x_1-\eta(s-s_1),\xi)\right|d\xi d\eta ds_1 ds\notag\\
&\qquad+ C^2_{2\rho}\int_0^t \int_0^{s-\la}\int_{\R^3}\int_{\R^3} e^{-\nu_\de(v)(t-s)} e^{-\nu_\de(\eta)(s-s_1)}\chi_{\{|v|\leq N,|\eta|\leq 2N, |\xi|\leq 3N \}}\notag\\
&\qquad\qquad\qquad\qquad\qquad\times l_w(v,\eta)l_w(\eta,\xi)\left|w_\be(\xi) f(s_1,x_1-\eta(s-s_1),\xi)\right|d\xi d\eta ds_1 ds\notag\\
&=J_{121}+J_{122}+J_{123}+J_{124}.
\end{align}
It is straightforward to see that
\begin{align*}
\dis J_{121}&\leq C^2_{2\rho}\int_0^t \int_0^s\int_{\R^3}\int_{\R^3} e^{-\nu_\de(v)(t-s)} e^{-\nu_\de(\eta)(s-s_1)}\chi_{\{|v|\geq N \}}\notag\\
&\qquad\qquad\qquad\qquad\qquad\times l_w(v,\eta)l_w(\eta,\xi)\|w_\be f\|_{L^\infty_T L^{\infty}_{v,x}}d\xi d\eta ds_1 ds.
\end{align*}
In the last integral above, the only term that contains time variables $s_1$ and $s$ is $e^{-\nu_\de(v)(t-s)} e^{-\nu_\de(v)(s-s_1)}$. Then we can integrate with respect to $s_1$ and $s$ first and use \eqref{Prol1} again to get
\begin{align}\label{2ndJ121}
\dis J_{121}&\leq C\frac{C^2_{2\rho}}{C^2_{1\rho}}\int_{\R^3}\int_{\R^3} \chi_{\{|v|\geq N \}}\frac{1}{\nu(v)\nu(\eta)}\notag\\
&\qquad\qquad\qquad\qquad\qquad\times l_w(v,\eta)l_w(\eta,\xi)\|w_\be f\|_{L^\infty_T L^{\infty}_{v,x}}d\xi d\eta.
\end{align}
Then on the right-hand side of \eqref{2ndJ121}, using the fact from \eqref{Prol2} that
\begin{align*}
&\int_{\R^3}\int_{\R^3}\frac{1}{\nu(v)\nu(\eta)}l_w(v,\eta)l_w(\eta,\xi)d\xi d\eta\notag\\
&\leq Cm^{\ga-1}\int_{\R^3}\frac{1}{\nu(v)}l_w(v,\eta)d\eta\notag\\
&\leq Cm^{2\ga-2}\frac{1}{(1+|v|)^2},
\end{align*}
we obtain that
\begin{align}\label{J121}
\dis J_{121}&\leq C\frac{C^2_{2\rho}}{C^2_{1\rho}}\int_{\R^3}\int_{\R^3} \chi_{\{|v|\geq N \}}\frac{1}{\nu(v)\nu(\eta)}l_w(v,\eta)l_w(\eta,\xi)\|w_\be f\|_{L^\infty_T L^{\infty}_{v,x}}d\xi d\eta\notag\\
&\leq Cm^{2\ga-2}\frac{C^2_{2\rho}}{C^2_{1\rho}}\frac{1}{(1+|v|)^2}\chi_{\{|v|\geq N \}}\|w_\be f\|_{L^\infty_T L^{\infty}_{v,x}}\notag\\
&\leq \frac{Cm^{2\ga-2}}{N}\frac{C^2_{2\rho}}{C^2_{1\rho}}\|w_\be f\|_{L^\infty_T L^{\infty}_{v,x}}.
\end{align}
Notice that if $|v|\leq N,|\eta|\geq 2N$ or $|\eta|\leq 2N,|\xi|\geq 3N$, we have either $|v-\eta|\geq N$ or $|\eta-\xi|\geq N$. Thus,  we have $e^{-\frac{|v-\eta|^2}{20}}e^{-\frac{|\eta-\xi|^2}{20}}\leq \frac{C}{N}$ for some constant $C$. Then similar arguments as \eqref{J121} show that
\begin{align}\label{J122}
\dis J_{122}&\leq C^2_{2\rho}\int_0^t \int_0^s\int_{\R^3}\int_{\R^3} e^{-\nu_\de(v)(t-s)} e^{-\nu_\de(\eta)(s-s_1)}\chi_{\{|v|\leq N,|\eta|\geq 2N \}}\chi_{\{|\eta|\leq 2N,|\xi|\geq 3N \}}\notag\\
&\qquad\qquad\qquad\qquad\qquad\times l_w(v,\eta)l_w(\eta,\xi)\|w_\be f\|_{L^\infty_T L^{\infty}_{v,x}}d\xi d\eta ds_1 ds\notag\\
&\leq C\frac{C^2_{2\rho}}{C^2_{1\rho}}\int_{\R^3}\int_{\R^3} \chi_{\{|v|\leq N,|\eta|\geq 2N \}}\chi_{\{|\eta|\leq 2N,|\xi|\geq 3N \}}\frac{1}{\nu(v)\nu(\eta)}\notag\\
&\qquad\qquad\times l_w(v,\eta)e^{\frac{|v-\eta|^2}{20}}e^{-\frac{|v-\eta|^2}{20}}l_w(\eta,\xi)e^{\frac{|\eta-\xi|^2}{20}}e^{-\frac{|\eta-\xi|^2}{20}}\|w_\be f\|_{L^\infty_T L^{\infty}_{v,x}}d\xi d\eta\notag\\
&\leq\frac{C}{N}\frac{C^2_{2\rho}}{C^2_{1\rho}}\int_{\R^3}\int_{\R^3}\frac{1}{\nu(v)\nu(\eta)}l_w(v,\eta)e^{\frac{|v-\eta|^2}{20}}l_w(\eta,\xi)e^{\frac{|\eta-\xi|^2}{20}}\|w_\be f\|_{L^\infty_T L^{\infty}_{v,x}}d\xi d\eta\notag\\
&\leq \frac{Cm^{2\ga-2}}{N}\frac{C^2_{2\rho}}{C^2_{1\rho}}\|w_\be f\|_{L^\infty_T L^{\infty}_{v,x}}.
\end{align}
$J_{123}$ can be similarly calculated as above, which yields
\begin{align}\label{J123}
\dis J_{123}&=C^2_{2\rho}\int_0^t \int_{s-\la}^s\int_{\R^3}\int_{\R^3} e^{-\nu_\de(v)(t-s)} e^{-\nu_\de(\eta)(s-s_1)}\notag\\
&\qquad\qquad\qquad\qquad\qquad\times l_w(v,\eta)l_w(\eta,\xi)\left|w_\be(\xi) f(s_1,x_1-\eta(s-s_1),\xi)\right|d\xi d\eta ds_1 ds\notag\\
&\leq C^2_{2\rho}\la\int_0^t \int_{\R^3}\int_{\R^3} e^{-\nu_\de(v)(t-s)}l_w(v,\eta)l_w(\eta,\xi)\|w_\be f\|_{L^\infty_T L^{\infty}_{v,x}}d\xi d\eta ds\notag\\
&\leq C\frac{C^2_{2\rho}}{C_{1\rho}}\la\int_{\R^3}\int_{\R^3} \frac{1}{\nu(v)}l_w(v,\eta)l_w(\eta,\xi)\|w_\be f\|_{L^\infty_T L^{\infty}_{v,x}}d\xi d\eta\notag\\
&\leq Cm^{2\ga-2}\frac{C^2_{2\rho}}{C_{1\rho}}\la\|w_\be f\|_{L^\infty_T L^{\infty}_{v,x}}.
\end{align}
Now we focus on $J_{124}$ where
\begin{align*}
\dis J_{124}&=C^2_{2\rho}\int_0^t \int_0^{s-\la}\int_{\R^3}\int_{\R^3} e^{-\nu_\de(v)(t-s)} e^{-\nu_\de(\eta)(s-s_1)}\chi_{\{|v|\leq N,|\eta|\leq 2N, |\xi|\leq 3N \}}\notag\\
&\qquad\qquad\qquad\qquad\qquad\times l_w(v,\eta)l_w(\eta,\xi)\left|w_\be(\xi) f(s_1,x_1-\eta(s-s_1),\xi)\right|d\xi d\eta ds_1 ds.
\end{align*}
Although all the velocity variables are bounded, $|v-\eta|$ or $|\eta-\xi|$ might be small, then it is difficult to have a good pointwise bound of the function $l_w$ due to the possible singularity from $\frac{1}{|v-\eta|}$ or $\frac{1}{|\eta-\xi|}$. However, if we take integral, $l_w$ is bounded in the following sense,
\begin{align*}
\sup_{|v|\leq 3N}\int_{|\eta|\leq 3N}l_w(v,\eta)d\eta &\leq C(\ga) m^{\ga-1}.
\end{align*}
Then for $0<\kappa\ll 1$ there exists a smooth function $l_\kappa=l_\kappa(v,\eta)$ with compact support such that
\begin{align*}
\sup_{|v|\leq 3N}\int_{|\eta|\leq 3N}|l_w(v,\eta)-l_\kappa(v,\eta)|d\eta &\leq C(\ga)\kappa^{3+\ga}.
\end{align*}
Moreover we have the pointwise bound of $l_\kappa$ which is 
\begin{align*}
\sup_{|v|\leq 3N,|\eta|\leq 3N}|l_\kappa(v,\eta)| &\leq C(\ga)N^\be\frac{1}{\kappa^{\frac{3-\ga}{2}}}.
\end{align*}
Such approximation can be directly obtained by removing the singularity using a smooth cut-off function to restrict $l_w$ in the region $\{|v|\leq 3N,|\eta|\leq 3N,|v-\eta|\geq \kappa\}$. Now we choose $\kappa=N^{-\frac{7}{3+\ga}}$ and denote $l_N=l_\kappa|_{\kappa=N^{-\frac{7}{3+\ga}}}$ to get that
\begin{align}\label{lN1}
\sup_{|v|\leq 3N}\int_{|\eta|\leq 3N}|l_w(v,\eta)-l_N(v,\eta)|d\eta &\leq \frac{C}{N^7},
\end{align}
and
\begin{align}\label{lN2}
\sup_{|v|\leq 3N,|\eta|\leq 3N}|l_N(v,\eta)| &\leq C(\ga)C_N,
\end{align}
where 
\begin{align}\label{DEFCN}
C_N=N^{\frac{7(3-\ga)}{6+2\ga}+\be}.
\end{align}
It is direct to see that $l_N$ also satisfies $\sup_{|v|\leq 3N}\int_{|\eta|\leq 3N}|l_N(v,\eta)|d\eta \leq C(\ga).$ Since we need to figure out how the constants depend on the parameter $\rho$, then we should calculate exactly how the pointwise and integral bounds of the approximation function $l_N$ depend on $N$ for later use. If $\rho$ is a given number instead of a parameter, we will only need \eqref{lN1} and \eqref{lN2} without the explicit formula for $C_N$.
Using the approximation function $l_N$, it follows that
\begin{align}\label{1stJ124}
\dis J_{124}&\leq C^2_{2\rho}\int_0^t \int_0^{s-\la}\int_{\R^3}\int_{\R^3} e^{-\nu_\de(v)(t-s)} e^{-\nu_\de(\eta)(s-s_1)}\chi_{\{|v|\leq N,|\eta|\leq 2N, |\xi|\leq 3N \}}\notag\\
&\qquad\qquad\qquad\qquad\times\left|l_w(v,\eta)\right|\left|l_w(\eta,\xi)-\l_N(\eta,\xi)\right|\left|w_\be(\xi) f(s_1,x_1-\eta(s-s_1),\xi)\right|d\xi d\eta ds_1 ds\notag\\
&\quad+ C^2_{2\rho}\int_0^t \int_0^{s-\la}\int_{\R^3}\int_{\R^3} e^{-\nu_\de(v)(t-s)} e^{-\nu_\de(\eta)(s-s_1)}\chi_{\{|v|\leq N,|\eta|\leq 2N, |\xi|\leq 3N \}}\notag\\
&\qquad\qquad\qquad\qquad\times\left|l_N(\eta,\xi)\right|\left|l_w(v,\eta)-\l_N(v,\eta)\right|\left|w_\be(\xi) f(s_1,x_1-\eta(s-s_1),\xi)\right|d\xi d\eta ds_1 ds\notag\\
&\quad+ C^2_{2\rho}\int_0^t \int_0^{s-\la}\int_{\R^3}\int_{\R^3} e^{-\nu_\de(v)(t-s)} e^{-\nu_\de(\eta)(s-s_1)}\chi_{\{|v|\leq N,|\eta|\leq 2N, |\xi|\leq 3N \}}\notag\\
&\qquad\qquad\qquad\qquad\times\left|l_N(v,\eta)\l_N(v,\eta)\right|\left|w_\be(\xi) f(s_1,x_1-\eta(s-s_1),\xi)\right|d\xi d\eta ds_1 ds.
\end{align}
Applying the pointwise bound $\left|w_\be(\xi) f(s_1,x_1-\eta(s-s_1),\xi)\right|\leq \|w_\be f\|_{L^\infty_T L^{\infty}_{v,x}}$, then integrating with respect to $s_1$, $s$, $\xi$ and $\eta$ respectively, one gets that
\begin{align}\label{2ndJ124}
\dis J_{124}&\leq C\frac{C^2_{2\rho}}{C^2_{1\rho}N^7}\|w_\be f\|_{L^\infty_T L^{\infty}_{v,x}}\int_{\R^3}\frac{\left|l_w(v,\eta)\right|}{\nu(v)\nu(\eta)}\chi_{\{|v|\leq N,|\eta|\leq 2N\}}d\eta\notag\\
&\quad+ C\frac{C^2_{2\rho}}{C^2_{1\rho}}\|w_\be f\|_{L^\infty_T L^{\infty}_{v,x}}\int_{\R^3}\frac{\left|l_w(v,\eta)-\l_N(v,\eta)\right|}{\nu(v)\nu(\eta)}\chi_{\{|v|\leq N,|\eta|\leq 2N\}}d\eta\notag\\
&\quad+ C\,C^2_{2\rho}C_N^2\int_0^t \int_0^{s-\la}\int_{\R^3}\int_{\R^3} e^{-\nu_\de(v)(t-s)} e^{-\nu_\de(\eta)(s-s_1)}\chi_{\{|v|\leq N,|\eta|\leq 2N, |\xi|\leq 3N \}}\notag\\
&\qquad\qquad\qquad\qquad\qquad\times\left|w_\be(\xi) f(s_1,x_1-\eta(s-s_1),\xi)\right|d\xi d\eta ds_1 ds.
\end{align}
Using the inequalities that
\begin{align*}
\int_{\R^3}\frac{\left|l_w(v,\eta)\right|}{\nu(\eta)\nu(\xi)}\chi_{\{|v|\leq N,|\eta|\leq 2N\}}d\eta\leq CN^6\int_{\R^3}\left|l_w(v,\eta)\right|d\eta\leq CN^6
\end{align*}
and
\begin{align*}
\int_{\R^3}\frac{\left|l_w(v,\eta)-\l_N(v,\eta)\right|}{\nu(\eta)\nu(\xi)}\chi_{\{|v|\leq N,|\eta|\leq 2N\}}d\eta\leq CN^6\int_{\R^3}\left|l_w(v,\eta)-\l_N(v,\eta)\right|d\eta\leq \frac{C}{N},
\end{align*}
it follows from \eqref{2ndJ124} that
\begin{align}\label{3rdJ124}
\dis J_{124}&\leq C\frac{C^2_{2\rho}}{C^2_{1\rho}N}\|w_\be f\|_{L^\infty_T L^{\infty}_{v,x}}\notag\\
&\quad+ C\,C^2_{2\rho}C_N^2\int_0^t \int_0^{s-\la}\int_{\R^3}\int_{\R^3} e^{-\nu_\de(v)(t-s)} e^{-\nu_\de(\eta)(s-s_1)}\chi_{\{|\eta|\leq 2N, |\xi|\leq 3N \}}\notag\\
&\qquad\qquad\qquad\qquad\qquad\times\left|w_\be(\xi) f(s_1,x_1-\eta(s-s_1),\xi)\right|d\xi d\eta ds_1 ds.
\end{align}
We need to treat the second term on the right-hand side of \eqref{3rdJ124} carefully. Recall from Lemma \ref{KdeK0} that
$$
\nu_\de(v)\geq C\, C_{1\rho}(1+|v|)^\ga.
$$
By $|v|\leq N$, $|\eta|\leq 2N$ we have
\begin{align}\label{nuvbound}
\nu_\de(v)\geq C\, C_{1\rho}N^\ga.
\end{align}
Then we can bound  the second term on the right-hand side of \eqref{3rdJ124} by
\begin{align}\label{entroterm}
&C\,C^2_{2\rho}C_N^2\int_0^t \int_0^{s-\la}\int_{\R^3}\int_{\R^3} e^{-\nu_\de(v)(t-s)} e^{-\nu_\de(\eta)(s-s_1)}\chi_{\{|\eta|\leq 2N, |\xi|\leq 3N \}}\notag\\
&\qquad\qquad\qquad\qquad\qquad\times\left|w_\be(\xi) f(s_1,x_1-\eta(s-s_1),\xi)\right|d\xi d\eta ds_1 ds\notag\\
&\leq C\,C^2_{2\rho}C_N^2\int_0^t \int_0^{s-\la}\int_{\R^3}\int_{\R^3} e^{-C_{1\rho}N^\ga(t-s)} e^{-C_{1\rho}N^\ga(s-s_1)}\chi_{\{|\eta|\leq 2N, |\xi|\leq 3N \}}\notag\\
&\qquad\qquad\qquad\qquad\qquad\times\left|w_\be(\xi) f(s_1,x_1-\eta(s-s_1),\xi)\right|d\xi d\eta ds_1 ds.
\end{align}
By observing that the only element in the last integral above that contains velocity variables $\eta$, $\xi$ is  $w_\be(\xi) f(s_1,x_1-\eta(s-s_1),\xi)$, we can consider $\iint\chi_{\{|\eta|\leq 2N, |\xi|\leq 3N \}}\left|w_\be(\xi) f(s_1,x_1-\eta(s-s_1),\xi)\right|d\xi d\eta$ first. Applying change of variable $y=x_1-\eta(s-s_1)$, we obtain
\begin{align}\label{UseEntropy}
&\iint_{|\eta|\leq2N, |\xi|\leq3N}\left|(w_\be f)(s_1,x_1-\eta(s-s_1),\xi)\right| d\eta d\xi\notag\\
&\leq CN^\be\iint_{|\eta|\leq2N, |\xi|\leq3N}\left(\frac{\left|F-\mu \right|}{\sqrt{\mu}}\right)(s_1,x_1-\eta(s-s_1),\xi)\chi_{\{|F(s_1,x_1-\eta(s-s_1),\xi)-\mu(\xi)|\leq \mu(\xi)\}} d\eta d\xi\notag\\
&\quad+CN^\be\iint_{|\eta|\leq2N, |\xi|\leq3N}\left|(F-\mu) (s_1,x_1-\eta(s-s_1),\xi)\right|\chi_{\{|F(s_1,x_1-\eta(s-s_1),\xi)-\mu(\xi)|\geq \mu(\xi)\}} d\eta d\xi\notag\\
&\leq CN^\be\frac{1+(s-s_1)^\frac{3}{2}}{(s-s_1)^\frac{3}{2}}\left\{\int_{\Omega}\int_{ |\xi|\leq3N}\left(\frac{\left|F-\mu \right|^2}{\mu}\right)(s_1,y,\xi)\chi_{\{|F(s_1,y,\xi)-\mu(\xi)|\leq \mu(\xi)\}} dy d\xi\right\}^\frac{1}{2}\notag\\
&\quad+ CN^\be\frac{1+(s-s_1)^3}{(s-s_1)^3}\int_{\Omega}\int_{ |\xi|\leq3N}\left|(F-\mu) (s_1,y,\xi)\right|\chi_{\{|F(s_1,y,\xi)-\mu(\xi)|\geq \mu(\xi)\}} dyd\xi. 
\end{align}
 Substituting \eqref{UseEntropy} into \eqref{entroterm} and using Lemma \ref{Taylor}, we obtain
\begin{align}\label{entrocontrol}
&C\,C^2_{2\rho}C_N^2\int_0^t \int_0^{s-\la}\int_{\R^3}\int_{\R^3} e^{-\nu_\de(v)(t-s)} e^{-\nu_\de(\eta)(s-s_1)}\chi_{\{|\eta|\leq 2N, |\xi|\leq 3N \}}\notag\\
&\qquad\qquad\qquad\qquad\qquad\times\left|w_\be(\xi) f(s_1,x_1-\eta(s-s_1),\xi)\right|d\xi d\eta ds_1 ds\notag\\
&\leq C\frac{C^2_{2\rho}}{C_{1\rho}N^\ga}C_N^2N^\be\left(\frac{1}{C_{1\rho}N^\ga}+\frac{1}{\left(C_{1\rho}N^\ga\right)^{5/2}}+\frac{1}{\left(C_{1\rho}N^\ga\right)^{4}} \right) \left( \la^{-\frac{3}{2}}\sqrt{\CE(F_0)}+ \la^{-3}\CE(F_0)\right)\notag\\
&=C\, C_{\rho,N}\left( \la^{-\frac{3}{2}}\sqrt{\CE_{\de,\rho}(F_0)}+ \la^{-3}\CE_{\de,\rho}(F_0)\right),
\end{align}
where 
\begin{align}\label{CrhoN}
C_{\rho,N}=\frac{C^2_{2\rho}}{C_{1\rho}N^\ga}C_N^2N^\be\left(\frac{1}{C_{1\rho}N^\ga}+\frac{1}{\left(C_{1\rho}N^\ga\right)^{5/2}}+\frac{1}{\left(C_{1\rho}N^\ga\right)^{4}} \right)
\end{align}
with $C_N$ defined in \eqref{DEFCN}.
It follows from \eqref{3rdJ124} and \eqref{entrocontrol} that
\begin{align}\label{J124}
\dis J_{124}&\leq C\frac{C^2_{2\rho}}{C^2_{1\rho}N}\|w_\be f\|_{L^\infty_T L^{\infty}_{v,x}}+C\, C_{\rho,N}\left( \la^{-\frac{3}{2}}\sqrt{\CE_{\de,\rho}(F_0)}+ \la^{-3}\CE_{\de,\rho}(F_0)\right).
\end{align}
Combining \eqref{sepJ12}, \eqref{J121}, \eqref{J122}, \eqref{J123} and \eqref{J124}, we obtain
\begin{align}\label{J12}
\dis J_{12}&\leq \frac{Cm^{2\ga-2}}{N}\frac{C^2_{2\rho}}{C^2_{1\rho}}\|w_\be f\|_{L^\infty_T L^{\infty}_{v,x}}+Cm^{2\ga-2}\frac{C^2_{2\rho}}{C_{1\rho}}\la\|w_\be f\|_{L^\infty_T L^{\infty}_{v,x}}\notag\\
&\qquad\qquad\qquad+ C\frac{C^2_{2\rho}}{C^2_{1\rho}N}\|w_\be f\|_{L^\infty_T L^{\infty}_{v,x}}+C\, C_{\rho,N}\left( \la^{-\frac{3}{2}}\sqrt{\CE_{\de,\rho}(F_0)}+ \la^{-3}\CE_{\de,\rho}(F_0)\right).
\end{align}
Then by \eqref{1stJ1}, \eqref{J10}, \eqref{J11}, \eqref{J12} and \eqref{J13}, it holds that
\begin{align}\label{J1}
\dis J_1&\leq C\frac{C_{2\rho}}{C_{1\rho}} m^{\ga-1}\left\|w_\be f_0 \right\|_{L^{\infty}_{v,x}}+Cm^{3+\ga} \frac{C^2_{2\rho}}{C^2_{1\rho}}\|w_\be f\|_{L^\infty_T L^{\infty}_{v,x}}+\frac{Cm^{2\ga-2}}{N}\frac{C^2_{2\rho}}{C^2_{1\rho}}\|w_\be f\|_{L^\infty_T L^{\infty}_{v,x}}\notag\\
&+Cm^{2\ga-2}\frac{C^2_{2\rho}}{C_{1\rho}}\la\|w_\be f\|_{L^\infty_T L^{\infty}_{v,x}}+ C\frac{C^2_{2\rho}}{C^2_{1\rho}N}\|w_\be f\|_{L^\infty_T L^{\infty}_{v,x}}+C\, C_{\rho,N}\left( \la^{-\frac{3}{2}}\sqrt{\CE_{\de,\rho}(F_0)}+ \la^{-3}\CE_{\de,\rho}(F_0)\right)\notag\\
&+C\frac{C_{2\rho}C_{5\rho}}{C^2_{1\rho}}m^{\ga-1}\sup_{0\leq t\leq T}\left(1+\|w_\be f(t)\|_{{L^{\infty}_{v,x}}}\right)\notag\\
&\qquad\qquad\qquad\qquad\times\left\{\|w_\be f(t)\|^{\frac{2p-1}{p}}_{{L^{\infty}_{v,x}}}\| f(t)\|_{L^{\infty}_{x}L^{1}_{v}}^{\frac{1}{p}}+\|w_\be f(t)\|^{\frac{10p-1}{5p}}_{{L^{\infty}_{v,x}}}\| f(t)\|_{L^{\infty}_{x}L^{1}_{v}}^{\frac{1}{5p}}\right\}.
\end{align}
Finally, we obtain the $L^\infty$ estimate by \eqref{1stest}, \eqref{J2} and \eqref{J1} that
\begin{align*}
&\dis \left|w_\be(v) f(t,x,v)\right|\leq\left\|w_\be f_0 \right\|_{L^{\infty}_{v,x}}+Cm^{3+\ga} \frac{C_{2\rho}}{C_{1\rho}}\|w_\be f\|_{L^\infty_T L^{\infty}_{v,x}}+C\, C_{\rho,N}\left( \la^{-\frac{3}{2}}\sqrt{\CE_{\de,\rho}(F_0)}+ \la^{-3}\CE_{\de,\rho}(F_0)\right)\notag \\
&\quad+C\frac{C_{2\rho}}{C_{1\rho}} m^{\ga-1}\left\|w_\be f_0 \right\|_{L^{\infty}_{v,x}}+Cm^{3+\ga} \frac{C^2_{2\rho}}{C^2_{1\rho}}\|w_\be f\|_{L^\infty_T L^{\infty}_{v,x}}+\frac{Cm^{2\ga-2}}{N}\frac{C^2_{2\rho}}{C^2_{1\rho}}\|w_\be f\|_{L^\infty_T L^{\infty}_{v,x}}\notag\\
&\quad+Cm^{2\ga-2}\frac{C^2_{2\rho}}{C_{1\rho}}\la\|w_\be f\|_{L^\infty_T L^{\infty}_{v,x}}+ C\frac{C^2_{2\rho}}{C^2_{1\rho}N}\|w_\be f\|_{L^\infty_T L^{\infty}_{v,x}}\notag\\
&\quad+ C\frac{C_{5\rho}}{C_{1\rho}}\left(1+ \frac{C_{2\rho}}{C_{1\rho}}m^{\ga-1} \right)\sup_{0\leq t\leq T}\left(1+\|w_\be f(t)\|_{{L^{\infty}_{v,x}}}\right)\notag\\
&\qquad\qquad\qquad\qquad\times\left\{\|w_\be f(t)\|^{\frac{2p-1}{p}}_{{L^{\infty}_{v,x}}}\| f(t)\|_{L^{\infty}_{x}L^{1}_{v}}^{\frac{1}{p}}+\|w_\be f(t)\|^{\frac{10p-1}{5p}}_{{L^{\infty}_{v,x}}}\| f(t)\|_{L^{\infty}_{x}L^{1}_{v}}^{\frac{1}{5p}}\right\}.
\end{align*}
We first choose some small $\ep_1=\ep_1(\be,\ga)$, let
\begin{align*}
\dis m_\rho=\min\left\{   \left(\frac{C_{1\rho}}{C_{2\rho}}\right)^\frac{1}{3+\ga} ,\left(\frac{C^2_{1\rho}}{C^2_{2\rho}}\right)^\frac{1}{3+\ga}       \right\},
\end{align*}
and
$m= \ep_1 m_\rho$ to get
\begin{align}\label{2Linftyest}
&\dis \|w_\be f\|_{L^\infty_T L^{\infty}_{v,x}}\leq\left\|w_\be f_0 \right\|_{L^{\infty}_{v,x}}+C\frac{C_{2\rho}}{C_{1\rho}} m_\rho^{\ga-1}\left\|w_\be f_0 \right\|_{L^{\infty}_{v,x}}+C\, C_{\rho,N}\left( \la^{-\frac{3}{2}}\sqrt{\CE_{\de,\rho}(F_0)}+ \la^{-3}\CE_{\de,\rho}(F_0)\right)\notag\\
&\quad+Cm_\rho^{2\ga-2}\frac{C^2_{2\rho}}{C_{1\rho}}\la\|w_\be f\|_{L^\infty_T L^{\infty}_{v,x}}+ C\frac{C^2_{2\rho}}{C^2_{1\rho}N}\|w_\be f\|_{L^\infty_T L^{\infty}_{v,x}}+\frac{Cm_\rho^{2\ga-2}}{N}\frac{C^2_{2\rho}}{C^2_{1\rho}}\|w_\be f\|_{L^\infty_T L^{\infty}_{v,x}}\notag\\
&\quad+ C\frac{C_{5\rho}}{C_{1\rho}}\left(1+ \frac{C_{2\rho}}{C_{1\rho}}m_\rho^{\ga-1} \right)\sup_{0\leq t\leq T}\left(1+\|w_\be f(t)\|_{{L^{\infty}_{v,x}}}\right)\notag\\
&\qquad\qquad\qquad\qquad\times\left\{\|w_\be f(t)\|^{\frac{2p-1}{p}}_{{L^{\infty}_{v,x}}}\| f(t)\|_{L^{\infty}_{x}L^{1}_{v}}^{\frac{1}{p}}+\|w_\be f(t)\|^{\frac{10p-1}{5p}}_{{L^{\infty}_{v,x}}}\| f(t)\|_{L^{\infty}_{x}L^{1}_{v}}^{\frac{1}{5p}}\right\}.
\end{align}
Similarly we choose $\ep_2=\ep_2(\be,\ga)$ such that if we let 
\begin{align}\label{defNrho}
\dis N=N_\rho:=\frac{1}{\eps_2}\max\left\{ \frac{C^2_{2\rho}}{C^2_{1\rho}} ,\frac{m_\rho^{2\ga-2}C^2_{2\rho}}{C^2_{1\rho}}  \right\},
\end{align}
and
\begin{align}\label{deflarho}
\dis \la=\la_\rho:= \eps_2\frac{C_{1\rho}}{m_\rho^{2\ga-2}C^2_{2\rho}}  ,
\end{align}
 it follows from \eqref{2Linftyest} that
\begin{align*}
&\dis \|w_\be f\|_{L^\infty_T L^{\infty}_{v,x}}\leq\left\|w_\be f_0 \right\|_{L^{\infty}_{v,x}}+C\frac{C_{2\rho}}{C_{1\rho}} m_\rho^{\ga-1}\left\|w_\be f_0 \right\|_{L^{\infty}_{v,x}}+C\, C_{\rho,N_\rho}\left( \la_\rho^{-\frac{3}{2}}\sqrt{\CE_{\de,\rho}(F_0)}+ \la_\rho^{-3}\CE_{\de,\rho}(F_0)\right)\notag\\
&\quad+ C\frac{C_{5\rho}}{C_{1\rho}}\left(1+ \frac{C_{2\rho}}{C_{1\rho}}m_\rho^{\ga-1} \right)\sup_{0\leq t\leq T}\left(1+\|w_\be f(t)\|_{{L^{\infty}_{v,x}}}\right)\notag\\
&\qquad\qquad\qquad\qquad\times\left\{\|w_\be f(t)\|^{\frac{2p-1}{p}}_{{L^{\infty}_{v,x}}}\| f(t)\|_{L^{\infty}_{x}L^{1}_{v}}^{\frac{1}{p}}+\|w_\be f(t)\|^{\frac{10p-1}{5p}}_{{L^{\infty}_{v,x}}}\| f(t)\|_{L^{\infty}_{x}L^{1}_{v}}^{\frac{1}{5p}}\right\}.
\end{align*}
Notice that if we choose $\be>3$, we have $\| f(t)\|_{L^{\infty}_{x}L^{1}_{v}}\leq C\|w_\be f(t)\|_{{L^{\infty}_{v,x}}}$. Then by Theorem \ref{local} we have
\begin{align*}
\dis \|w_\be f\|_{L^\infty_T L^{\infty}_{v,x}}\leq&\left\|w_\be f_0 \right\|_{L^{\infty}_{v,x}}+C\frac{C_{2\rho}}{C_{1\rho}} m_\rho^{\ga-1}\left\|w_\be f_0 \right\|_{L^{\infty}_{v,x}}+C\, C_{\rho,N_\rho}\left( \la_\rho^{-\frac{3}{2}}\sqrt{\CE(F_0)}+ \la_\rho^{-3}\CE(F_0)\right)\notag\\
&\quad+ C\frac{C_{5\rho}}{C_{1\rho}}\left(1+ \frac{C_{2\rho}}{C_{1\rho}}m_\rho^{\ga-1} \right)\left(1+\|w_\be f_0\|_{{L^{\infty}_{v,x}}}\right)\|w_\be f_0\|^2_{{L^{\infty}_{v,x}}}\notag\\
&\quad+ C\frac{C_{5\rho}}{C_{1\rho}}\left(1+ \frac{C_{2\rho}}{C_{1\rho}}m_\rho^{\ga-1} \right)\sup_{T_1\leq t\leq T}\left(1+\|w_\be f(t)\|_{{L^{\infty}_{v,x}}}\right)\notag\\
&\qquad\qquad\qquad\qquad\times\left\{\|w_\be f(t)\|^{\frac{2p-1}{p}}_{{L^{\infty}_{v,x}}}\| f(t)\|_{L^{\infty}_{x}L^{1}_{v}}^{\frac{1}{p}}+\|w_\be f(t)\|^{\frac{10p-1}{5p}}_{{L^{\infty}_{v,x}}}\| f(t)\|_{L^{\infty}_{x}L^{1}_{v}}^{\frac{1}{5p}}\right\}.\end{align*}
We simplify the above inequality by writing
\begin{align}\label{Linftyest}
 \|w_\be f\|_{L^\infty_T L^{\infty}_{v,x}}\leq& C\,C_{6\rho}\left(\left\|w_\be f_0 \right\|_{L^{\infty}_{v,x}}+\left\|w_\be f_0 \right\|_{L^{\infty}_{v,x}}^2+\left\|w_\be f_0 \right\|_{L^{\infty}_{v,x}}^3\right)\notag\\
&\quad+ C\frac{C_{5\rho}}{C_{1\rho}}\left(1+ \frac{C_{2\rho}}{C_{1\rho}}m_\rho^{\ga-1} \right)\sup_{T_1\leq t\leq T}\left(1+\|w_\be f(t)\|_{{L^{\infty}_{v,x}}}\right)\notag\\
&\qquad\qquad\qquad\qquad\times\left\{\|w_\be f(t)\|^{\frac{2p-1}{p}}_{{L^{\infty}_{v,x}}}\| f(t)\|_{L^{\infty}_{x}L^{1}_{v}}^{\frac{1}{p}}+\|w_\be f(t)\|^{\frac{10p-1}{5p}}_{{L^{\infty}_{v,x}}}\| f(t)\|_{L^{\infty}_{x}L^{1}_{v}}^{\frac{1}{5p}}\right\}\notag\\
&\quad+C\, C_{\rho,N_\rho}\left( \la_\rho^{-\frac{3}{2}}\sqrt{\CE_{\de,\rho}(F_0)}+ \la_\rho^{-3}\CE_{\de,\rho}(F_0)\right),
\end{align}
where
\begin{align}\label{C6rho}
\dis C_{6\rho}=1+\frac{C_{5\rho}}{C_{1\rho}}\left(1+ \frac{C_{2\rho}}{C_{1\rho}}m_\rho^{\ga-1} \right)+\frac{C_{2\rho}}{C_{1\rho}}m_\rho^{\ga-1}.
\end{align}
Then \eqref{lemmaglobal} follows from \eqref{Linftyest}.
\end{proof}

We can see from \eqref{lemmaglobal} that if $\left\|w_\be f_0 \right\|_{L^{\infty}_{v,x}}$ and $\CE_{\de,\rho}(F_0)$ are uniformly small in $\de$, then we can close the a priori estimate to obtain a global solution. But we want to remove the smallness assumption on $\left\|w_\be f_0 \right\|_{L^{\infty}_{v,x}}$ so that the initial datum can have arbitrary large $L^\infty$ norm. In the meantime the smallness $\left\|w_\be f_0 \right\|_{L^{\infty}_{v,x}}$ is replaced by the smallness $\left\| f_0 \right\|_{L^{1}_{x}L^\infty_v}$. In this case, $F_0$ can reach the upper bound $\frac{1}{\de}$ and the lower bound $0$. In the view of \eqref{lemmaglobal}, this requires us to prove that the smallness of $\sup_{0\leq\de\leq1}\left\{ \CE_{\de,\rho}(F_0)+\|f_0\|_{L^1_xL^\infty_v}\right\}$ implies the smallness of $\| f\|_{{L^\infty_{T_1,T}}L^{\infty}_{x}L^{1}_{v}}$.

\begin{lemma}\label{L1est}
Let $\ga$, $\be$ and $p$ satisfy the assumption in Theorem \ref{global} and $T_1$ be the constant given in Theorem \ref{local}. Then for any $T>T_1$ and $(t,x)\in[T_1,T]\times \Omega$, it holds that
\begin{align}\label{L1Est}
\int_{\R^3}\left|f(t,x,v)\right|dv\leq &\int_{\R^3}e^{-\nu_\de(v)t}\left|f_0(x-vt,v)\right|dv+C\left(m^{3+\ga}\frac{C_{2\rho}}{C_{1\rho}}+\frac{1}{N}\frac{C_{2\rho}}{C_{1\rho}}+C_{2\rho}\la\right)\|w_\be f\|_{L^\infty_TL^{\infty}_{v,x}}\notag\\
&+C\left(C_{5\rho}\la+\frac{1}{N}\frac{C_{5\rho}}{C_{1\rho}}\right)\left(1+\|w_\be f\|_{L^\infty_TL^{\infty}_{v,x}}\right)\|w_\be f\|_{L^\infty_TL^{\infty}_{v,x}}^2\notag\\
&+C\|w_\be f\|^\frac{p-1}{p}_{L^\infty_TL^{\infty}_{v,x}} C'_{\rho,N}\left( \la^{-\frac{3}{2}}\sqrt{\CE_{\de,\rho}(F_0)}+ \la^{-3}\CE_{\de,\rho}(F_0)\right)^\frac{1}{p}\notag\\
&+CN^3\|w_\be f\|^\frac{5p-1}{5p}_{L^\infty_TL^{\infty}_{v,x}}  C'_{\rho,N}\left( \la^{-\frac{3}{2}}\sqrt{\CE_{\de,\rho}(F_0)}+ \la^{-3}\CE_{\de,\rho}(F_0)\right)^\frac{1}{5p}\notag\\
&+C\, \bar{C}_{\rho,N}\left( \la^{-\frac{3}{2}}\sqrt{\CE_{\de,\rho}(F_0)}+ \la^{-3}\CE_{\de,\rho}(F_0)\right),
\end{align}
where $\bar{C}_{\rho,N}$, $C'_{\rho,N}$ are given in \eqref{CbarrhoN} and \eqref{CprirhoN} respectively, $m,N,\la>0$ are constants which will be determined later.
\end{lemma}
\begin{proof}
Let $(t,x)\in[T_1,T]\times \Omega$, using the mild form \eqref{mildQBE}, we have
\begin{equation}\label{L1first}
\int_{\R^3}\left|f(t,x,v)\right|dv\leq \int_{\R^3}e^{-\nu_\de(v)t}\left|f_0(x-vt,v)\right|dv+H_1+H_2+H_3,
\end{equation}
where
\begin{align}
&H_1:=\int_0^t\int_{\R^3} e^{-\nu_\de(v)(t-s)}\left|(K_\de^mf)(s,x-v(t-s),v)\right|dvds, \notag\\
&H_2:=\int_0^t\int_{\R^3} e^{-\nu_\de(v)(t-s)}\left|( K_\de^cf)(s,x-v(t-s),v)\right|dvds,\notag\\
&H_3:=\int_0^t\int_{\R^3} e^{-\nu_\de(v)(t-s)}\left| \Ga_\de(f)(s,x-v(t-s),v)\right|dvds.\notag
\end{align}
By Lemma \ref{KdeK0}, one gets that
\begin{align}\label{H1}
H_1&\leq Cm^{3+\ga}C_{2\rho}\| f\|_{L^\infty_TL^{\infty}_{v,x}}\int_0^t\int_{\R^3} e^{-\nu_\de(v)(t-s)}e^{-\frac{|v|^2}{20}}dvds \notag\\
&\leq Cm^{3+\ga}\frac{C_{2\rho}}{C_{1\rho}}\| f\|_{L^\infty_TL^{\infty}_{v,x}}\int_{\R^3}\frac{1}{\nu(v)}e^{-\frac{|v|^2}{20}}dv \notag\\
&\leq Cm^{3+\ga}\frac{C_{2\rho}}{C_{1\rho}}\| f\|_{L^\infty_TL^{\infty}_{v,x}}.
\end{align}
Once again using  Lemma \ref{KdeK0} and recalling the definition of $l(v,\eta)$ \eqref{lveta}, it follows that
\begin{align*}
H_2&\leq C_{2\rho}\int_0^t\int_{\R^3} e^{-\nu_\de(v)(t-s)}\int_{\R^3}l(v,\eta)\left|f(s,x-v(t-s),\eta)\right|d\eta dvds.
\end{align*}
A similar argument as \eqref{sepJ12} shows that
\begin{align*}
\dis H_{2}&\leq C_{2\rho}\int_0^t\int_{\R^3}\int_{\R^3} e^{-\nu_\de(v)(t-s)}\chi_{\{|v|\geq N \}}l(v,\eta)\frac{1}{w_\be(\eta)}\left|w_\be(\eta) f(s,x-v(t-s),\eta)\right|d\eta dvds\notag\\
&\ +C_{2\rho}\int_0^t\int_{\R^3}\int_{\R^3} e^{-\nu_\de(v)(t-s)}\chi_{\{|v|\leq N,|\eta|\geq 2N \}}l(v,\eta)\frac{1}{w_\be(\eta)}\left|w_\be(\eta) f(s,x-v(t-s),\eta)\right|d\eta dvds \notag\\
&\ + C_{2\rho}\int_{t-\la}^t\int_{\R^3}\int_{\R^3} e^{-\nu_\de(v)(t-s)}l(v,\eta)\frac{1}{w_\be(\eta)}\left|w_\be(\eta) f(s,x-v(t-s),\eta)\right|d\eta dvds\notag\\
&\ + C_{2\rho}\int_{0}^{t-\la}\int_{\R^3}\int_{\R^3} e^{-\nu_\de(v)(t-s)}\chi_{\{|v|\leq N,|\eta|\leq 2N \}}l(v,\eta)\frac{1}{w_\be(\eta)}\left|w_\be(\eta) f(s,x-v(t-s),\eta)\right|d\eta dvds\notag\\
&=H_{21}+H_{22}+H_{23}+H_{24}.
\end{align*}
It follows from our assumption $\be>6>3-\ga$, Lemma \ref{KmKc} and similar arguments in \eqref{J121} that
\begin{align}\label{H21}
\dis H_{21}&\leq C_{2\rho}\| w_\be f\|_{L^\infty_TL^{\infty}_{v,x}}\int_0^t\int_{\R^3}\int_{\R^3} e^{-\nu_\de(v)(t-s)}\chi_{\{|v|\geq N \}}l(v,\eta)\frac{1}{w_\be(\eta)}d\eta dvds\notag\\
&\leq C\frac{C_{2\rho}}{C_{1\rho}}\|w_\be  f\|_{L^\infty_T L^{\infty}_{v,x}}\int_{\R^3}\int_{\R^3}\chi_{\{|v|\geq N \}}\frac{1}{\nu(v)}l(v,\eta)\frac{1}{w_\be(\eta)}d\eta dv \notag\\
&\leq C\frac{C_{2\rho}}{C_{1\rho}}\|w_\be  f\|_{L^\infty_T L^{\infty}_{v,x}}\int_{\R^3}\chi_{\{|v|\geq N \}}\frac{1}{w_\be(v)\nu(v)(1+|v|)^2}dv \notag\\
&\leq \frac{C}{N}\frac{C_{2\rho}}{C_{1\rho}}\|w_\be f\|_{L^\infty_T L^{\infty}_{v,x}}.
\end{align}
In the last step above we have used the inequality
\begin{align*}
&\int_{\R^3}\chi_{\{|v|\geq N \}}\frac{1}{w_\be(v)\nu(v)(1+|v|)^2}dv\notag\\
&=\int_{\R^3}\chi_{\{|v|\geq N \}}\frac{1}{(1+|v|)^{\be+\ga+2}}dv\notag\\
&\leq \frac{C}{N}\int_{\R^3}\frac{1}{(1+|v|)^{\be+\ga}}dv\leq \frac{C}{N}.
\end{align*}
By the same approach as \eqref{J122} and \eqref{H21} we also have
\begin{align}\label{H22}
\dis H_{22}&\leq  \frac{C}{N}\frac{C_{2\rho}}{C_{1\rho}}\|w_\be f\|_{L^\infty_T L^{\infty}_{v,x}}.
\end{align}
$H_{23}$ can be treated in the same way as in \eqref{J123} that
\begin{align}\label{H23}
\dis H_{23}&\leq C_{2\rho}\|w_\be  f\|_{L^\infty_TL^{\infty}_{v,x}}\int_{t-\la}^t\int_{\R^3}\int_{\R^3} e^{-\nu_\de(v)(t-s)}l(v,\eta)\frac{1}{w_\be(\eta)}d\eta dvds \notag\\
&\leq C_{2\rho}\la\|w_\be  f\|_{L^\infty_TL^{\infty}_{v,x}}\int_{\R^3}\int_{\R^3}l(v,\eta)\frac{1}{w_\be(\eta)}d\eta dv \notag\\
&\leq C\,C_{2\rho}\la\|w_\be f\|_{L^\infty_T L^{\infty}_{v,x}}.
\end{align}
Constructing the similar approximation function $\bar{l}_N=\bar{l}_N(v,\eta)$ as in \eqref{1stJ124}, we split $H_{24}$ in the following way
\begin{align*}
\dis H_{24}&\leq C_{2\rho}\int_0^{t-\la}\int_{\R^3}\int_{\R^3} e^{-\nu_\de(v)(t-s)} \chi_{\{|v|\leq N,|\eta|\leq 2N\}}\notag\\
&\qquad\qquad\qquad\qquad\qquad\times\left|l(v,\eta)-\bar{l}_N(v,\eta)\right|\left|w_\be(\eta) f(s,x-v(t-s),\eta)\right|d\eta dv ds\notag\\
&\quad+ C_{2\rho}\int_0^{t-\la}\int_{\R^3}\int_{\R^3} e^{-\nu_\de(v)(t-s)}\chi_{\{|v|\leq N,|\eta|\leq 2N \}}\notag\\
&\qquad\qquad\qquad\qquad\qquad\times\left|\bar{l}_N(v,\eta)\right|\left|f(s,x-v(t-s),\eta)\right|d\eta dv ds\notag\\
&\leq C\frac{C_{2\rho}}{C_{1\rho}N}\|w_\be f\|_{L^\infty_T L^{\infty}_{v,x}}\notag\\
&\quad+ C\,C_{2\rho}C_N\int_0^{t-\la}\int_{\R^3}\int_{\R^3} e^{-\nu_\de(v)(t-s)}\chi_{\{|v|\leq N,|\eta|\leq 2N \}}\left| f(s,x-v(t-s),\eta)\right|d\eta dv ds.
\end{align*}
Using the entropy condition as in \eqref{UseEntropy} and \eqref{entrocontrol}, we obtain that
\begin{align}\label{H24}
\dis H_{24}&\leq C\frac{C_{2\rho}}{C_{1\rho}N}\|w_\be f\|_{L^\infty_T L^{\infty}_{v,x}}+ C\, \bar{C}_{\rho,N}\left( \la^{-\frac{3}{2}}\sqrt{\CE_{\de,\rho}(F_0)}+ \la^{-3}\CE_{\de,\rho}(F_0)\right),
\end{align}
where
\begin{align}\label{CbarrhoN}
\bar{C}_{\rho,N}=C_{2\rho}C_N\left(\frac{1}{C_{1\rho}N^\ga}+\frac{1}{\left(C_{1\rho}N^\ga\right)^{5/2}}+\frac{1}{\left(C_{1\rho}N^\ga\right)^{4}} \right).
\end{align}
Combining \eqref{H21}, \eqref{H22}, \eqref{H23} and \eqref{H24}, it holds that
\begin{align}\label{H2}
\dis H_2\leq\frac{C}{N}\frac{C_{2\rho}}{C_{1\rho}}\|w_\be f\|_{L^\infty_T L^{\infty}_{v,x}}+C\,&C_{2\rho}\la\|w_\be f\|_{L^\infty_T L^{\infty}_{v,x}}\notag\\&+C\, \bar{C}_{\rho,N}\left( \la^{-\frac{3}{2}}\sqrt{\CE_{\de,\rho}(F_0)}+ \la^{-3}\CE_{\de,\rho}(F_0)\right).
\end{align}
The last term we need to estimate is $H_3$. Rewrite $H_3$ as follows:
\begin{align}\label{splitH3}
H_3=&\int_0^t\int_{\R^3} e^{-\nu_\de(v)(t-s)}\frac{1}{w_\be(v)}\left|w_\be(v) \Ga_\de(f)(s,x-v(t-s),v)\right|dvds\notag\\
=&\int_{t-\la}^t\int_{\R^3} e^{-\nu_\de(v)(t-s)}\frac{1}{w_\be(v)}\left|w_\be(v) \Ga_\de(f)(s,x-v(t-s),v)\right|dvds\notag\\
&+\int_0^{t-\la}\chi_{\{|v|\geq N \}}\int_{\R^3} e^{-\nu_\de(v)(t-s)}\frac{1}{w_\be(v)}\left|w_\be(v) \Ga_\de(f)(s,x-v(t-s),v)\right|dvds\notag\\
&+\int_0^{t-\la}\chi_{\{|v|\leq N \}}\int_{\R^3} e^{-\nu_\de(v)(t-s)}\frac{1}{w_\be(v)}\left|w_\be(v) \Ga_\de(f)(s,x-v(t-s),v)\right|dvds\notag\\
=&H_{31}+H_{32}+H_{33}.
\end{align}
By Lemma \eqref{Gaest} and similar arguments as in \eqref{H21} and \eqref{H23}, we have
\begin{align}\label{H31H32}
\dis H_{31}+H_{32}\leq C\left(C_{5\rho}\la+\frac{1}{N}\frac{C_{5\rho}}{C_{1\rho}}\right)\left(1+\|w_\be f\|_{L^\infty_TL^{\infty}_{v,x}}\right)\|w_\be f\|_{L^\infty_TL^{\infty}_{v,x}}^2.
\end{align}

Denoting $x_1=x-v(t-s)$ and dividing $\left|w_\be(v) \Ga_\de(f)(s,x-v(t-s),v)\right|$ into two parts as in \eqref{Ga1}, we obtain
\begin{align}\label{splitGa2}
\dis \left|w_\be(v)\Ga_\de(f)(s,x-v(t-s),v)\right|dvds \leq G_1+G_2, 
\end{align}
where
\begin{align*}
\dis
G_1=&\frac{C}{\sqrt{\bar{\mu}_{\de,\rho}(v)}}\int_{\R^3}\int_{\S^2}B(v-u,\theta)\left[w_\be(v)\left|\sqrt{\bar{\mu}_{\de,\rho}}f(s,x_1,u)\sqrt{\bar{\mu}_{\de,\rho}}f(s,x_1,v)\right|\right.\notag\\
&\qquad\qquad\qquad + w_\be(v)\left|\sqrt{\bar{\mu}_{\de,\rho}}f(s,x_1,v')\sqrt{\bar{\mu}_{\de,\rho}}f(s,x_1,u)\right|\left|\mu_{\de,\rho}(v)+\mu_{\de,\rho}(u')\right|\notag\\
&\qquad\qquad\qquad+ w_\be(v)\left|\sqrt{\bar{\mu}_{\de,\rho}}f(s,x_1,u)\sqrt{\bar{\mu}_{\de,\rho}}f(s,x_1,v)\sqrt{\bar{\mu}_{\de,\rho}}f(s,x_1,u')\right|\notag\\
&\qquad\qquad\qquad+w_\be(v)\left|\sqrt{\bar{\mu}_{\de,\rho}}f(s,x_1,u)\sqrt{\bar{\mu}_{\de,\rho}}f(s,x_1,v)\sqrt{\bar{\mu}_{\de,\rho}}f(s,x_1,v')\right|\notag\\
&\qquad\qquad\qquad\left.+w_\be(v)\left|\sqrt{\bar{\mu}_{\de,\rho}}f(s,x_1,u')\sqrt{\bar{\mu}_{\de,\rho}}f(s,x_1,v')\sqrt{\bar{\mu}_{\de,\rho}}f(s,x_1,u)\right|\right]d\om du,
\end{align*} 
and
\begin{align*}
\dis
 G_2=&\frac{C}{\sqrt{\bar{\mu}_{\de,\rho}(v)}}\int_{\R^3}\int_{\S^2}B(v-u,\theta)\left[w_\be(v)\left|\sqrt{\bar{\mu}_{\de,\rho}}f(s,x_1,u')\sqrt{\bar{\mu}_{\de,\rho}}f(s,x_1,v')\right|\right.\notag\\
&\qquad\qquad \qquad + w_\be(v)\left|\sqrt{\bar{\mu}_{\de,\rho}}f(s,x_1,v')\sqrt{\bar{\mu}_{\de,\rho}}f(s,x_1,v)\right|\left|\mu_{\de,\rho}(u)+\mu_{\de,\rho}(u')\right|\notag\\
&\qquad \qquad \qquad \left.+w_\be(v)\left|\sqrt{\bar{\mu}_{\de,\rho}}f(s,x_1,u')\sqrt{\bar{\mu}_{\de,\rho}}f(s,x_1,v')\sqrt{\bar{\mu}_{\de,\rho}}f(s,x_1,v)\right|\right]d\om du.
\end{align*} 
Recalling from \eqref{C5rho} that $C_{5\rho}=C_{2\rho}+C_{3\rho}+C_{4\rho}$, similar arguments as in \eqref{1stJ21} and \eqref{1stJ22} show that
\begin{align}\label{1stG1}
\dis
& G_1\leq C\,C_{5\rho}\nu(v)\left(1+\|w_\be f\|_{L^\infty_TL^{\infty}_{v,x}}\right)\int_{\R^3}\int_{\S^2}B(v-u,\theta)e^{-\frac{|u|^2}{4}}\left|f(s,x_1,u)\right|d\om du,
\end{align} 
and
\begin{align}\label{1stG2}
\dis
& G_2 \leq C\,C_{5\rho}\nu(v)\left(1+\|w_\be f\|_{L^\infty_TL^{\infty}_{v,x}}\right)\int_{\R^3}\int_{\S^2}B(v-u,\theta)e^{-\frac{|u|^2}{4}}\left|f(s,x_1,v')\right|d\om du.
\end{align} 
Notice that when we estimate $G_1$, we do not write it as in \eqref{J22}. This is because the variable $u$ in \eqref{J22} is not the original velocity $u$ defined by \eqref{velocity}. We first consider the case when $u\geq N$ for the variable $u$ in \eqref{1stG1} and \eqref{1stG2}. Then we have
\begin{align}\label{2ndG1}
\int_{u\geq N}\left|f(s,x_1,u)\right|du=\int_{u\geq N}\frac{1}{{w_{\be/2}} (u)}\left|w_{\be/2}(u) f(s,x_1,u)\right|du,
\end{align}
and
\begin{align}\label{2ndG2}
\int_{u\geq N}\int_{\S^2}e^{-\frac{|u|^2}{4}}\left|f(s,x_1,v')\right|^pd\om du=\int_{u\geq N}\int_{\S^2}e^{-\frac{|u|^2}{8}}e^{-\frac{|u|^2}{8}}\left|f(s,x_1,v')\right|^pd\om du.
\end{align}
By the discussion above, we take the sum of $G_1$ and $G_2$ and use similar arguments in \eqref{2ndJ21}, \eqref{2ndJ22} and \eqref{3rdJ22} to obtain
\begin{align*}
\dis
 G_1+G_2\leq& C\frac{C_{5\rho}}{N}\nu(v)\left(1+\|w_\be f\|_{L^\infty_TL^{\infty}_{v,x}}\right)\|w_\be f\|^{\frac{2p-1}{p}}_{L^\infty_TL^{\infty}_{v,x}}\left(\int_{\R^3}\left|w_{\be/2}(u)f(s,x_1,u)\right|du\right)^{\frac{1}{p}}\notag\\
&+C\,C_{5\rho}\nu(v)\left(1+\|w_\be f\|_{L^\infty_TL^{\infty}_{v,x}}\right)\int_{u\leq N}\int_{\S^2}B(v-u,\theta)e^{-\frac{|u|^2}{4}}\left|f(s,x_1,u)\right|d\om du\notag\\
&+C\frac{C_{5\rho}}{N}\nu(v)\left(1+\|w_\be f\|_{L^\infty_TL^{\infty}_{v,x}}\right)\|w_\be f\|^{\frac{10p-1}{5p}}_{L^\infty_TL^{\infty}_{v,x}}\left(\int_{\R^3}\left|f(s,x_1,u)\right|du\right)^{\frac{1}{5p}}\notag\\
&+C\,C_{5\rho}\nu(v)\left(1+\|w_\be f\|_{L^\infty_TL^{\infty}_{v,x}}\right)\int_{u\leq N}\int_{\S^2}B(v-u,\theta)e^{-\frac{|u|^2}{4}}\left|f(s,x_1,v')\right|d\om du.
\end{align*} 
On the right-hand sides of \eqref{2ndG1} and \eqref{2ndG2}, we have either $\frac{1}{w_{\be/2} (u)}\leq \frac{C}{N}$ or  $e^{-\frac{|u|^2}{8}}\leq \frac{C}{N^p}$. Using the argument in \eqref{3rdJ22}, we conclude that 
\begin{align}\label{2ndG3132}
\dis
G_1+G_2&\leq C\frac{C_{5\rho}}{N}\nu(v)\left(1+\|w_\be f\|_{L^\infty_TL^{\infty}_{v,x}}\right)\|w_\be f\|^{2}_{L^\infty_TL^{\infty}_{v,x}}\notag\\
&+C\,C_{5\rho}\nu(v)\left(1+\|w_\be f\|_{L^\infty_TL^{\infty}_{v,x}}\right)\int_{u\leq N}\int_{\S^2}B(v-u,\theta)e^{-\frac{|u|^2}{4}}\left|f(s,x_1,u)\right|d\om du\notag\\
&+C\,C_{5\rho}\nu(v)\left(1+\|w_\be f\|_{L^\infty_TL^{\infty}_{v,x}}\right)\int_{u\leq N}\int_{\S^2}B(v-u,\theta)e^{-\frac{|u|^2}{4}}\left|f(s,x_1,v')\right|d\om du,
\end{align} 
by the fact that
\begin{align*}
&\int_{\R^3}\left|w_{\be/2}(u)f(t,x_1,u)\right|du=\int_{\R^3}\left|\frac{1}{w_{\be/2}(u)}w_{\be}(u)f(s,x_1,u)\right|du\notag\\
&\leq \|w_\be f\|_{L^\infty_TL^{\infty}_{v,x}}\int_{\R^3}\frac{1}{w_{\be/2}(u)}du\leq \|w_\be f\|_{L^\infty_TL^{\infty}_{v,x}},
\end{align*}
for $\be>6$. Combining \eqref{splitGa2} and \eqref{2ndG3132}, we have
\begin{align}\label{1stH33}
H_{33}&=\int_0^{t-\la}\chi_{\{|v|\leq N \}}\int_{\R^3} e^{-\nu_\de(v)(t-s)}\frac{1}{w_\be(v)}\left|w_\be(v) \Ga_\de(f)(s,x-v(t-s),v)\right|dvds\notag\\
&\leq C\frac{C_{5\rho}}{C_{1\rho}N}\left(1+\|w_\be f\|_{L^\infty_TL^{\infty}_{v,x}}\right)\|w_\be f\|^{2}_{L^\infty_TL^{\infty}_{v,x}}\notag\\
&\quad+\int_0^{t-\la}\iint_{\{|v|\leq N, |u|\leq N\}} e^{-\nu_\de(v)(t-s)}\frac{e^{-\frac{|u|^2}{4}}}{w_\be(v)}B(v-u,\theta)\left|f(s,x_1,u)\right|dudvds\notag\\
&\quad+\int_0^{t-\la}\iint_{\{|v|\leq N, |u|\leq N\}}\int_{\S^2} e^{-\nu_\de(v)(t-s)}\frac{e^{-\frac{|u|^2}{4}}}{w_\be(v)}B(v-u,\theta)\left|f(s,x_1,u')\right|dwdudvds\notag\\
&= C\frac{C_{5\rho}}{C_{1\rho}N}\left(1+\|w_\be f\|_{L^\infty_TL^{\infty}_{v,x}}\right)\|w_\be f\|^{2}_{L^\infty_TL^{\infty}_{v,x}}+H_{331}+H_{332}.
\end{align}
For $H_{331}$, we directly apply H\"older's inequality and \eqref{nuvbound} to obtain
\begin{align}\label{1stH331}
\dis H_{331}&\leq\int_0^{t-\la}\iint_{\{|v|\leq N, |u|\leq N\}} e^{-C_{1\rho}N^\ga(t-s)}\frac{e^{-\frac{|u|^2}{4}}}{w_\be(v)}B(v-u,\theta)\left|f(s,x_1,u)\right|dudvds\notag\\
&\leq \int_0^{t-\la}e^{-C_{1\rho}N^\ga(t-s)}\left(\iint_{\{|v|\leq N, |u|\leq N\}}|v-u|^{\frac{\ga p}{p-1}} \frac{e^{-\frac{|u|^2}{4}}}{w_\be(v)} dudv\right)^{\frac{p-1}{p}}\notag\\
&\qquad\qquad\qquad\qquad\qquad\qquad\times \left(\iint_{\{|v|\leq N, |u|\leq N\}}\left|f(s,x_1,u)\right|^pdudv\right)^\frac{1}{p}ds\notag\\
&\leq C\int_0^{t-\la}e^{-C_{1\rho}N^\ga(t-s)}\left(\iint_{\{|v|\leq N, |u|\leq N\}}\left|f(s,x-v(t-s),u)\right|^pdudv\right)^\frac{1}{p}ds.
\end{align}
It follows from a similar argument as in \eqref{entrocontrol} that
\begin{align}\label{H331}
\dis H_{331}&\leq C\|w_\be f\|^\frac{p-1}{p}_{L^\infty_TL^{\infty}_{v,x}}\int_0^{t-\la}e^{-C_{1\rho}N^\ga(t-s)}\left(\iint_{\{|v|\leq N, |u|\leq N\}}\left|f(s,x-v(t-s),u)\right|dudv\right)^\frac{1}{p}ds\notag\\
&\leq C\|w_\be f\|^\frac{p-1}{p}_{L^\infty_TL^{\infty}_{v,x}}\left(\frac{1}{C_{1\rho}N^\ga}+\frac{1}{\left(C_{1\rho}N^\ga\right)^{5/2}}+\frac{1}{\left(C_{1\rho}N^\ga\right)^{4}} \right) \left( \la^{-\frac{3}{2}}\sqrt{\CE(F_0)}+ \la^{-3}\CE(F_0)\right)^\frac{1}{p}\notag\\
&=C\|w_\be f\|^\frac{p-1}{p}_{L^\infty_TL^{\infty}_{v,x}} C'_{\rho,N}\left( \la^{-\frac{3}{2}}\sqrt{\CE_{\de,\rho}(F_0)}+ \la^{-3}\CE_{\de,\rho}(F_0)\right)^\frac{1}{p},
\end{align}
where
\begin{align}\label{CprirhoN}
 C'_{\rho,N}=\left(\frac{1}{C_{1\rho}N^\ga}+\frac{1}{\left(C_{1\rho}N^\ga\right)^{5/2}}+\frac{1}{\left(C_{1\rho}N^\ga\right)^{4}} \right) .
\end{align}
For $H_{332}$, we first apply H\"older's inequality as \eqref{1stH331} to get
\begin{align*}
\dis H_{332}&\leq C\int_0^{t-\la}e^{-C_{1\rho}N^\ga(t-s)}\left(\iint_{\{|v|\leq N, |u|\leq N\}}\int_{\S^2}e^{-\frac{|u|^2}{4}}\left|f(s,x-v(t-s),v')\right|^pd\om dudv\right)^\frac{1}{p}ds.
\end{align*}
Then by the same transformation as in \eqref{3rdJ22} and notice if $v\leq N$ and $u\leq N$, we have $|\eta|=|v+z_\para|\leq |u|+2|v|\leq 3N$. It holds that
\begin{align*}
\dis H_{332}&\leq C\int_0^{t-\la}e^{-C_{1\rho}N^\ga(t-s)}\notag\\
&\qquad\times\left(\iint_{\{|v|\leq N, |\eta|\leq 3N\}}\int_{z_\perp}e^{-\frac{|z_\perp+\eta|^2}{4}}|f(s,x-v(t-s),\eta)|^p\frac{1}{|\eta-v|^{2}} dz_\perp d\eta dv\right)^\frac{1}{p}ds\notag\\
&\leq C\int_0^{t-\la}e^{-C_{1\rho}N^\ga(t-s)}\left(\iint_{\{|v|\leq N, |\eta|\leq 3N\}}|f(s,x-v(t-s),\eta)|^p\frac{1}{|\eta-v|^{2}} d\eta dv\right)^\frac{1}{p}ds.
\end{align*}
Similar arguments as in \eqref{3rdJ22} show that
\begin{align}\label{H332}
\dis H_{332}&\leq CN^3\|w_\be f\|^\frac{5p-1}{5p}_{L^\infty_TL^{\infty}_{v,x}}\int_0^{t-\la}e^{-C_{1\rho}N^\ga(t-s)}\notag\\
&\qquad\qquad\qquad\times\left(\iint_{\{|v|\leq N, |\eta|\leq 3N\}}|f(s,x-v(t-s),\eta)|d\eta dv\right)^\frac{1}{5p}ds\notag\\
&\leq CN^3\|w_\be f\|^\frac{5p-1}{5p}_{L^\infty_TL^{\infty}_{v,x}}  C'_{\rho,N}\left( \la^{-\frac{3}{2}}\sqrt{\CE_{\de,\rho}(F_0)}+ \la^{-3}\CE_{\de,\rho}(F_0)\right)^\frac{1}{5p}.
\end{align}
By \eqref{1stH33}, \eqref{H331} and \eqref{H332} we have
\begin{align}\label{H33}
H_{33}\leq &C\frac{C_{5\rho}}{C_{1\rho}N}\left(1+\|w_\be f\|_{L^\infty_TL^{\infty}_{v,x}}\right)\|w_\be f\|^{2}_{L^\infty_TL^{\infty}_{v,x}}\notag\\
&+C\|w_\be f\|^\frac{p-1}{p}_{L^\infty_TL^{\infty}_{v,x}} \bar{C}_{\rho,N}\left( \la^{-\frac{3}{2}}\sqrt{\CE_{\de,\rho}(F_0)}+ \la^{-3}\CE_{\de,\rho}(F_0)\right)^\frac{1}{p}\notag\\
&+CN^3\|w_\be f\|^\frac{5p-1}{5p}_{L^\infty_TL^{\infty}_{v,x}}  C'_{\rho,N}\left( \la^{-\frac{3}{2}}\sqrt{\CE_{\de,\rho}(F_0)}+ \la^{-3}\CE_{\de,\rho}(F_0)\right)^\frac{1}{5p}.
\end{align}
Then it holds from \eqref{splitH3}, \eqref{H31H32} and \eqref{H33} that
\begin{align}\label{H3}
\dis H_3\leq&C\left(C_{5\rho}\la+\frac{1}{N}\frac{C_{5\rho}}{C_{1\rho}}\right)\left(1+\|w_\be f\|_{L^\infty_TL^{\infty}_{v,x}}\right)\|w_\be f\|_{L^\infty_TL^{\infty}_{v,x}}^2\notag\\
&+C\frac{C_{5\rho}}{C_{1\rho}N}\left(1+\|w_\be f\|_{L^\infty_TL^{\infty}_{v,x}}\right)\|w_\be f\|^{2}_{L^\infty_TL^{\infty}_{v,x}}\notag\\
&+C\|w_\be f\|^\frac{p-1}{p}_{L^\infty_TL^{\infty}_{v,x}} \bar{C}_{\rho,N}\left( \la^{-\frac{3}{2}}\sqrt{\CE_{\de,\rho}(F_0)}+ \la^{-3}\CE_{\de,\rho}(F_0)\right)^\frac{1}{p}\notag\\
&+CN^3\|w_\be f\|^\frac{5p-1}{5p}_{L^\infty_TL^{\infty}_{v,x}}  C'_{\rho,N}\left( \la^{-\frac{3}{2}}\sqrt{\CE_{\de,\rho}(F_0)}+ \la^{-3}\CE_{\de,\rho}(F_0)\right)^\frac{1}{5p}.
\end{align}
Hence, \eqref{L1Est} follows from \eqref{L1first}, \eqref{H1}, \eqref{H2} and \eqref{H3}.
\end{proof}
With the preparation above, we can prove Theorem \ref{global}. From Lemma \ref{globalest}, we first assume
\begin{align*}\left\|w_\be f \right\|_{{L^{\infty}_{T}L^{\infty}_{v,x}}}\leq 2A_0,
\end{align*}
where
\begin{align*}
A_0&=3C_3 C_{6\rho}M^3+C_3 C_{\rho,N_\rho}\left( \la_\rho^{-\frac{3}{2}}\sqrt{\CE_{\de,\rho}(F_0)}+ \la_\rho^{-3}\CE_{\de,\rho}(F_0)\right).
\end{align*}
Then by \eqref{lemmaglobal} we have
\begin{align}\label{apriori1}
&\dis \|w_\be f\|_{L^\infty_T L^{\infty}_{v,x}}\leq A_0\notag\\
&\qquad\qquad+ C_3 C_{6\rho}\left(1+2A_0\right)\left\{(2A_0)^{\frac{2p-1}{p}}\| f\|_{L^\infty_{T_1,T}L^{\infty}_{x}L^{1}_{v}}^{\frac{1}{p}}+(2A_0)^{\frac{10p-1}{5p}}\| f\|_{{L^\infty_{T_1,T}}L^{\infty}_{x}L^{1}_{v}}^{\frac{1}{5p}}\right\}.
\end{align}
A direct calculation shows that
\begin{align*}
A_0&\leq 3C_3 C_{7\rho}M^3+C_3 C_{7\rho}\left( \sqrt{\CE_{\de,\rho}(F_0)}+ \CE_{\de,\rho}(F_0)\right),
\end{align*}
where 
\begin{align*}
C_{7\rho}=C_{6\rho}+C_{\rho,N_\rho}\left( \la_\rho^{-\frac{3}{2}}+ \la_\rho^{-3}\right).
\end{align*}
We first require 
\begin{align*}
\CE_{\de,\rho}(F_0)\leq M^3
\end{align*}
to get
\begin{align}\label{close3}
A_0&\leq 5C_3 C_{7\rho}M^3.
\end{align}
Hence, we see that $\bar{C}_{*\rho}$ in \eqref{GE} should be defined by
\begin{align}\label{barC*rho}
\bar{C}_{*\rho}=C_{7\rho}.
\end{align}
If we can choose $\| f\|_{{L^\infty_{T_1,T}}L^{\infty}_{x}L^{1}_{v}}$ so small such that
\begin{align}\label{req2}
\| f\|_{{L^\infty_{T_1,T}}L^{\infty}_{x}L^{1}_{v}}\leq& \min\left\{ \left(\frac{3C_{6\rho}}{5C_{7\rho}}\frac{1}{4C_3C_{6\rho}2^\frac{2p-1}{p}}\frac{1}{(1+10C_3C_{7\rho}M^3)(5C_3C_{7\rho}M^3)^\frac{p-1}{p}}\right)^p, \right.\notag\\
&\qquad\qquad\left.   \left(\frac{3C_{6\rho}}{5C_{7\rho}}\frac{1}{4C_3C_{6\rho}2^\frac{5p-1}{5p}}\frac{1}{(1+10C_3C_{7\rho}M^3)(5C_3C_{7\rho}M^3)^\frac{5p-1}{5p}}\right)^{5p}      \right\},
\end{align}
we then have from \eqref{close3} that
\begin{align}\label{close1}
&C_3 C_{6\rho}\left(1+2A_0\right)(2A_0)^{\frac{2p-1}{p}}\| f\|_{L^\infty_{T_1,T}L^{\infty}_{x}L^{1}_{v}}^{\frac{1}{p}}\notag\\
&\leq \frac{3C_{6\rho}}{5C_{7\rho}}\frac{C_3 C_{6\rho}\left(1+10C_3C_{7\rho}M^3\right)(10C_3C_{7\rho}M^3)^{\frac{2p-1}{p}}}{4C_3C_{6\rho}2^\frac{2p-1}{p}(1+10C_3C_{7\rho}M^3)(5C_3C_{7\rho}M^3)^\frac{p-1}{p}}\notag\\
&\leq\frac{3}{4}C_3C_{6\rho}M^3\leq\frac{1}{4}A_0.
\end{align}
Similarly as above, we can prove that
\begin{align}\label{close2}
&C_3 C_{6\rho}\left(1+2A_0\right)(2A_0)^{\frac{10p-1}{5p}}\| f\|_{{L^\infty_{T_1,T}}L^{\infty}_{x}L^{1}_{v}}^{\frac{1}{5p}}\leq \frac{1}{4}A_0.
\end{align}
We can close our a priori assumption by \eqref{apriori1}, \eqref{close1} and \eqref{close2} that
\begin{align}\label{finalclose}
&\dis \|w_\be f\|_{L^\infty_T L^{\infty}_{v,x}}\leq A_0+\frac{A_0}{4}+\frac{A_0}{4}\leq \frac{3A_0}{2}.
\end{align}
We first relax the requirement \eqref{req2} to be
\begin{align}\label{req3}
\| f\|_{{L^\infty_{T_1,T}}L^{\infty}_{x}L^{1}_{v}}\leq& C_4 C_{8\rho},
\end{align}
where $C_4=C_4(\be,\ga,M)$ depends only on $\be$, $\ga$ and $M$, and
\begin{align*}
C_{8\rho}=\frac{1}{C_{7\rho}^{2p-1}(1+C_{7\rho}^p)+C_{7\rho}^{10p-1}(1+C_{7\rho}^{5p})}.
\end{align*}
Then we need to prove that we can actually let  $\| f\|_{{L^\infty_{T_1,T}}L^{\infty}_{x}L^{1}_{v}}$ meet the requirement \eqref{req3}. From \eqref{L1Est} in Lemma \ref{L1est} and \eqref{close3} we have

\begin{align*}
\int_{\R^3}\left|f(t,x,v)\right|dv\leq &\int_{\R^3}e^{-\nu_\de(v)t}\left|f_0(x-vt,v)\right|dv+C\left(m^{3+\ga}\frac{C_{2\rho}}{C_{1\rho}}+\frac{1}{N}\frac{C_{2\rho}}{C_{1\rho}}+C_{2\rho}\la\right)5C_3 C_{7\rho}M^3\notag\\
&+C\left(C_{5\rho}\la+\frac{1}{N}\frac{C_{5\rho}}{C_{1\rho}}\right)\left(1+5C_3 C_{7\rho}M^3\right)(5C_3 C_{7\rho}M^3)^2\notag\\
&+C(5C_3 C_{7\rho}M^3)^\frac{p-1}{p} C'_{\rho,N}\left( \la^{-\frac{3}{2}}\sqrt{\CE_{\de,\rho}(F_0)}+ \la^{-3}\CE_{\de,\rho}(F_0)\right)^\frac{1}{p}\notag\\
&+CN^3(5C_3 C_{7\rho}M^3)^\frac{5p-1}{5p}  C'_{\rho,N}\left( \la^{-\frac{3}{2}}\sqrt{\CE_{\de,\rho}(F_0)}+ \la^{-3}\CE_{\de,\rho}(F_0)\right)^\frac{1}{5p}\notag\\
&+C\, \bar{C}_{\rho,N}\left( \la^{-\frac{3}{2}}\sqrt{\CE_{\de,\rho}(F_0)}+ \la^{-3}\CE_{\de,\rho}(F_0)\right).
\end{align*}
Similarly as how we define $m_\rho$, $N_\rho$, we define
\begin{align*}
\bar{m}_\rho=&\ep_3\left(\frac{C_{8\rho}C_{1\rho}}{C_{2\rho}C_{7\rho}}\right)^{\frac{1}{3+\ga}},\notag\\
\bar{\la}_\rho=&\ep_3\min\{\frac{C_{8\rho}}{C_{2\rho}C_{7\rho}},\frac{C_{8\rho}}{C_{5\rho}C^2_{7\rho}},\frac{C_{8\rho}}{C_{5\rho}C^3_{7\rho}}\},\notag\\
\bar{N}_\rho=&\frac{1}{\ep_3}\max\{ \frac{C_{2\rho} C_{7\rho}}{C_{8\rho}C_{1\rho}}, \frac{C_{5\rho} C^3_{7\rho}}{C_{8\rho}C_{1\rho}}, \frac{C_{5\rho} C^2_{7\rho}}{C_{8\rho}C_{1\rho}} \},
\end{align*}
where $\ep_3=\ep_3(\be,\ga,M)$ is a small constant, such that
\begin{align*}
\int_{\R^3}\left|f(t,x,v)\right|dv\leq &\int_{\R^3}\left|f_0(x-vt,v)\right|dv\notag\\
&+C(5C_3 C_{7\rho}M^3)^\frac{p-1}{p} C'_{\rho,\bar{N}_\rho}\left( \bar{\la}_\rho^{-\frac{3}{2}}\sqrt{\CE_{\de,\rho}(F_0)}+ \bar{\la}_\rho^{-3}\CE_{\de,\rho}(F_0)\right)^\frac{1}{p}\notag\\
&+C\bar{N}_\rho^3(5C_3 C_{7\rho}M^3)^\frac{5p-1}{5p}  C'_{\rho,\bar{N}_\rho}\left( \bar{\la}_\rho^{-\frac{3}{2}}\sqrt{\CE_{\de,\rho}(F_0)}+ \bar{\la}_\rho^{-3}\CE_{\de,\rho}(F_0)\right)^\frac{1}{5p}\notag\\
&+C\, \bar{C}_{\rho,\bar{N}_\rho}\left( \bar{\la}_\rho^{-\frac{3}{2}}\sqrt{\CE_{\de,\rho}(F_0)}+ \bar{\la}_\rho^{-3}\CE_{\de,\rho}(F_0)\right)+\frac{1}{4}C_4 C_{8\rho}.
\end{align*}
We can see if we define
\begin{align*}
C_{9\rho}=\min&\left\{1,\frac{C_{8\rho}^{2p}}{\left(C_{7\rho}\right)^{2(p-1)}\left(C'_{\rho,\bar{N}_\rho}\right)^{2p} \left( \bar{\la}_\rho^{-\frac{3}{2}}+ \bar{\la}_\rho^{-3}\right)^2} ,\right.\notag\\
& \qquad\qquad \left. \frac{C_{8\rho}^{10p}}{\left(C_{7\rho}\right)^{2(5p-1)}\left(C'_{\rho,\bar{N}_\rho}\right)^{10p} \left( \bar{\la}_\rho^{-\frac{3}{2}}+ \bar{\la}_\rho^{-3}\right)^2} ,\left(\frac{C_{8\rho}}{\bar{C}_{\rho,\bar{N}_\rho}\left( \bar{\la}_\rho^{-\frac{3}{2}}+ \bar{\la}_\rho^{-3}\right)} \right)^2        \right\},
\end{align*}
then there exists $\ep_4=\ep_4(\be,\ga,M)$ such that if
\begin{align*}
\CE_{\de,\rho}(F_0)\leq \ep_4 C_{9\rho},
\end{align*}
we have 
\begin{align}\label{bridgeapriori}
\int_{\R^3}\left|f(t,x,v)\right|dv\leq &\int_{\R^3}\left|f_0(x-vt,v)\right|dv+\frac{C_4C_{8\rho}}{2}.
\end{align}
Recall from Theorem \ref{local} that $T_1=\frac{C_1}{C_{*\rho}\left(1+\|w_\be f_0\|_{L^{\infty}_{v,x}}+\|w_\be f_0\|^2_{L^{\infty}_{v,x}}\right)}>\frac{C}{C^3_{*\rho}M^2}$. We consider the case that $t\geq T_1$.
If $\Omega=\R^3$,
 \begin{align*}
\int_{\R^3}\left|f_0(x-vt,v)\right|dv \leq \frac{1}{T_1^3}\|f_0\|_{L^1_xL^\infty_v}\leq \frac{C}{C^3_{*\rho}}M^6\|f_0\|_{L^1_xL^\infty_v}.
\end{align*}
Then we choose some small $\ep_5(\be,\ga,M)$ such that
\begin{align*}
\|f_0\|_{L^1_xL^\infty_v}\leq\ep_5(\be,\ga,M)C^3_{*\rho}C_{8\rho}
\end{align*}
 and $\int_{\R^3}\left|f_0(x-vt,v)\right|dv \leq \frac{1}{2}C_4C_{8\rho}$. Then together with  \eqref{bridgeapriori}, \eqref{req3} is satisfied. Hence, \eqref{close2} holds true and we close the a priori estimate by \eqref{finalclose}.  Finally, we let $\ep=\min\{\ep_4,\ep_5  \}$ and
\begin{align}\label{R3barMrho}
\bar{M}_\rho=\min\{1,C_{9\rho},C_{*\rho}C_{8\rho} \}
\end{align}
to obtain \eqref{smallness}, \eqref{GE}.
If $\Omega=\T^3$, by $\int_{\{|v|\leq M_1\}}\left|f_0(x-vt,v)\right|dv\leq C\frac{(1+M_1t)^3}{t^3}\int_\Omega\|f_0(y)\|_{L^\infty_v}dy$, we obtain
 \begin{align*}
\dis \int_{\R^3}\left|f_0(x-vt,v)\right|dv &\leq  \int_{\{|v|\geq M_1\}}\left|f_0(x-vt,v)\right|dv+\int_{\{|v|\leq M_1\}}\left|f_0(x-vt,v)\right|dv\notag\\
&\leq \int_{\{|v|\geq M_1\}}\left|f_0(x-vt,v)\right|dv+C\left\{ M_1^3\|f_0\|_{L^1_xL^\infty_v} +M^6\|f_0\|_{L^1_xL^\infty_v}\right\} \notag\\
&\leq  M_1^{3-\be} \left\| w_\be f_0\right\|_{L^{\infty}_{v,x}}+CM_1^3\|f_0\|_{L^1_xL^\infty_v} +CM^6\|f_0\|_{L^1_xL^\infty_v}.
\end{align*}
By choosing $M_1=\left( \frac{\left\| w_\be f_0\right\|_{L^{\infty}_{v,x}}}{ \|f_0\|_{L^1_xL^\infty_v}} \right)^\frac{1}{\be}$, we have

 \begin{align*}
\dis\int_{\R^3}\left|f_0(x-vt,v)\right|dv &\leq  C\left\| w_\be f_0\right\|^{\frac{3}{\be}}_{L^{\infty}_{v,x}} \|f_0\|^{1-\frac{3}{\be}}_{L^1_xL^\infty_v}+\frac{C}{C^3_{*\rho}M^2}M^6\|f_0\|_{L^1_xL^\infty_v}\notag\\
&\leq  CM^{\frac{3}{\be}} \|f_0\|^{1-\frac{3}{\be}}_{L^1_xL^\infty_v}+\frac{C}{C^3_{*\rho}M^2}M^6\|f_0\|_{L^1_xL^\infty_v}.
\end{align*}
Then similarly as in \eqref{R3barMrho}, we define
\begin{align}\label{T3barMrho}
\bar{M}_\rho=\min\{1,C_{9\rho},C_{*\rho}C_{8\rho},C_{8\rho}^{\frac{3}{\be-3}} \}
\end{align}
to obtain \eqref{smallness} and \eqref{GE}. Therefore we finish the proof of Theorem \ref{global}.
\qed

\medskip
\noindent {\bf Acknowledgments:}\,
This work is supported by the Hong Kong PhD Fellowship Scheme. The author would like to thank Professor Renjun Duan for his support during the PhD studies in the Chinese University of Hong Kong and also thank Professor Ning Jiang for suggesting this problem in a seminar that was held at CUHK in September 2021.

\end{document}